\documentclass[11pt,letterpaper]{amsart}
\usepackage[centertags]{amsmath}
\usepackage{amsfonts}
\usepackage{amssymb}
\usepackage{amsthm}
\usepackage{newlfont}
\usepackage{amscd}
\usepackage{mathrsfs}
\usepackage[all,cmtip]{xy}
\usepackage{tikz-cd}
\tikzcdset{arrow style=Latin Modern, row sep/normal=1.35em, column sep/normal=1.8em, cramped}
\usepackage{tikz}
\usepackage[pagebackref=false,colorlinks]{hyperref}

\definecolor{mycolor}{HTML}{F7F8E0}
\definecolor{myorange}{RGB}{245,156,74}
\definecolor{cadetgrey}{rgb}{0.57, 0.64, 0.69}
\definecolor{calpolypomonagreen}{rgb}{0.12, 0.3, 0.17}

\hypersetup{pdffitwindow=true,linkcolor=calpolypomonagreen,citecolor=calpolypomonagreen,urlcolor=calpolypomonagreen}
\usepackage{cite}
\usepackage{fancyhdr}
\usepackage{stmaryrd}
\usepackage{makecell}

\usepackage[OT2,OT1]{fontenc}
\newcommand\cyr{%
\renewcommand\rmdefault{wncyr}%
\renewcommand\sfdefault{wncyss}%
\renewcommand\encodingdefault{OT2}%
\normalfont
\selectfont}
\DeclareTextFontCommand{\textcyr}{\cyr}

\usepackage[toc, page] {appendix}

\makeatletter
\let\@wraptoccontribs\wraptoccontribs
\makeatother

\numberwithin{equation}{section}

\newtheorem{thm}{Theorem}[section]

\newtheorem{thm-quote}[thm]{``Theorem"}
\newtheorem{cor}[thm]{Corollary}
\newtheorem{lem}[thm]{Lemma}
\newtheorem{prop}[thm]{Proposition}

\newtheorem{assu}[thm]{Assumption}
\newtheorem{conj}[thm]{Conjecture}

\theoremstyle{definition}
\newtheorem{defn}[thm]{Definition}

\newtheorem{rem}[thm]{Remark}

\newtheorem{ques}[thm]{Question}

\newcommand{\KS}{\mathbf{KS}}
\newcommand{\KSbar}{\overline{\mathbf{KS}}}

\newcommand{\ks}{\boldsymbol{\kappa}}
\newcommand{\kn}{\widetilde{\boldsymbol{\delta}}}
\newcommand{\widedelta}{\widetilde{\delta}}

\newcommand{\sha}{\textrm{{\cyr SH}}}

\begin{document}

\title{The refined Tamagawa number conjectures for $\mathrm{GL}_2$}
\author{Chan-Ho Kim}
\thanks{Chan-Ho Kim was partially supported 
by a KIAS Individual Grant (SP054103) via the Center for Mathematical Challenges at Korea Institute for Advanced Study,
by the National Research Foundation of Korea(NRF) grant funded by the Korea government(MSIT) (No. 2018R1C1B6007009, 2019R1A6A1A11051177), 
by research funds for newly appointed professors of Jeonbuk National University in 2024, and
by Global-Learning \& Academic research institution for Master’s$\cdot$Ph.D. Students, and Postdocs (LAMP) Program of the National Research Foundation of Korea (NRF) funded by the Ministry of Education (No. RS-2024-00443714).
}
\contrib[with an appendix in collaboration with]{Robert Pollack}
\thanks{Robert Pollack was partially supported by NSF grant DMS-2302285.}
\address{(Chan-Ho Kim) Department of Mathematics and Institute of Pure and Applied Mathematics,
Jeonbuk National University,
567 Baekje-daero, Deokjin-gu, Jeonju, Jeollabuk-do 54896, Republic of Korea}
\email{chanho.math@gmail.com}
\address{(Robert Pollack) Department of Mathematics, The University of Arizona, 617 N. Santa Rita Ave., Tucson, AZ 85721-0089, USA}
\email{rpollack@arizona.edu}
\date{\today}
\begin{abstract}
Let $f$ be a cuspidal newform and $p \geq 3$ a prime such that the associated $p$-adic Galois representation has large image.
We establish a new and refined ``Birch and Swinnerton-Dyer type" formula for Bloch--Kato Selmer groups of the central critical twist of $f$ via \emph{Kolyvagin derivatives of $L$-values} instead of complex analytic or $p$-adic variation of $L$-values only under the Iwasawa main conjecture localized at the augmentation ideal. 
Our formula determines the exact rank and module structure of the Selmer groups and is insensitive to weight, the local behavior of $f$ at $p$,  and analytic rank.

As consequences, we prove the non-vanishing of Kato's Kolyvagin system and complete a ``discrete" analogue of the Beilinson--Bloch--Kato conjecture for modular forms at ordinary primes.
We also obtain the higher weight analogue of the $p$-converse to the theorem of Gross--Zagier and Kolyvagin, the $p$-parity conjecture, and a new computational upper bound of Selmer ranks.
We also discuss how to formulate the refined conjecture on the non-vanishing of Kato's Kolyvagin system for modular forms of general weight.

In the appendix with Robert Pollack, we compute several numerical examples on the structure of Selmer groups of elliptic curves and modular forms of higher weight.
Sometimes our computation provides a deeper understanding of Selmer groups than what is predicted by Birch and Swinnerton-Dyer conjecture.
\end{abstract}

\maketitle

\setcounter{tocdepth}{1}
\tableofcontents

\section*{Introduction}
This paper deals with the problem of determining the \emph{exact} ranks and all the \emph{higher} Fitting ideals of Bloch--Kato Selmer groups of elliptic curves and modular forms via a different type of variation of the associated special $L$-values in a quite general setting.

Let $p  \geq 3$  be a prime and $f \in S_k(\Gamma_0(N))$ be a newform. Denote by $W^\dagger_f$ the discrete Galois module of the central critical twist of the associated $p$-adic Galois representation $\rho_f$. It is widely believed that the Bloch--Kato Selmer group $\mathrm{Sel}(\mathbb{Q}, W^\dagger_f)$ encodes the arithmetic of $f$ and has a deep connection with the central critical $L$-value $L(f, k/2)$ and its complex or $p$-adic variation in the framework of Bloch--Kato's Tamagawa number conjecture for motives \cite{bloch-kato}.

There have been substantial progresses towards the arithmetic of special $L$-values of elliptic curves and modular forms for decades, such as the analytic rank zero and one cases, their $p$-converses, and the $p$-part of the leading term formula, by establishing various Iwasawa main conjectures. See $\S$\ref{subsubsec:known-limitations-bsd-bloch-kato} for more details. However, our understanding is still extremely limited when the analytic rank is $> 1$, the weight goes beyond the Fontaine--Laffaille range (without the good ordinary assumption), \emph{or} the associated local automorphic representation at $p$ is supercuspidal. 

Under only the large image assumption on $\rho_f$ and the Iwasawa main conjecture  localized at the augmentation ideal, we establish an explicit special $L$-value formula for the structure of $\mathrm{Sel}(\mathbb{Q}, W^\dagger_f)$ (Theorems \ref{thm:main-central-critical} and \ref{thm:non-triviality-kn}).
This formula is more refined than (all the known results on) the corresponding Tamagawa number conjecture and does not need any condition on low analytic ranks or the local behavior of $f$ at $p$ mentioned above. The main difference from other approaches is that we focus directly on \emph{Kolyvagin derivatives of special $L$-values} instead of complex analytic or $p$-adic derivatives of special $L$-values.

Since the Iwasawa main conjecture for modular forms at ordinary primes (up to a power of $p$) is confirmed, we are able to describe the following important application (Theorem \ref{thm:main-central-critical}, Theorem \ref{thm:non-triviality-kn}, and Corollary \ref{cor:discrete-bsd}).
\begin{thm}
Let $f \in S_k(\Gamma_0(N))$ be a newform and $p  \geq 5$  be a good ordinary prime for $f$ such that $\rho_f$ has large image and is $p$-distinguished.
Then the module structure of $\mathrm{Sel}(\mathbb{Q}, W^\dagger_f)$ is determined by ``Kolyvagin derivatives of special $L$-values" as in (\ref{eqn:structure-bloch-kato}) and the corresponding Kato's Kolyvagin system is non-trivial. 
In particular, the analogues of the Beilinson--Bloch--Kato conjecture and the $p$-part of the Birch and Swinnerton-Dyer formula hold for $f$ as in (\ref{eqn:bsd-type-formulas}), so the exact Selmer rank is detected by our formula.
\end{thm}
In order to obtain such an \emph{exact} special $L$-value formula for Bloch--Kato Selmer groups (at least conjecturally),  it is natural to expect that both the choice of integral canonical periods and the use of integral $p$-adic Hodge theory are essential.
Surprisingly, our formula is insensitive to the choice of integral periods and our approach completely bypasses integral $p$-adic Hodge theory.
This is an interesting and ironic part of our work.

Another counter-intuitive aspect is that the Iwasawa main conjecture \emph{localized at the augmentation ideal} is strong enough to detect the exact bound of Selmer groups. From this point of view,  we observe that the ``analytic" main conjecture formulated by J.  Pottharst \cite{jay-cyc,jay-families} also encodes the exact bound of Selmer groups although $p$ is inverted therein.

Although it is widely believed that an Euler system provides an upper bound of a Selmer group only,  many ``Euler system-type" results on the upper bounds of Selmer groups would upgrade to those with the exact bounds by adapting our approach (e.g. \cite{kim-gross-zagier}).

Our method can be described as a \emph{non-standard} combination of Mazur--Rubin's Kolyvagin system argument, the systematic iteration of the global duality argument, the functional equation of modular symbols, the self-duality of Galois representations,  and the rigidity of Kolyvagin systems.  
In particular, the roles of Kato's zeta elements and the Iwasawa main conjecture are significantly different from those in the standard Euler system argument.

\subsection*{Organization of this article}
In \S\ref{sec:main-result},  the main results (including the structure formula) and their arithmetic applications are stated.
The applications include the arithmetic of modular abelian varieties of $\mathrm{GL}_2$-type, 
the rank one $p$-converse to the theorem of Nekov\'{a}\v{r} and S.-W. Zhang, the $p$-parity conjecture for good ordinary forms,  and  a computational upper bound of Selmer ranks.
In \S\ref{sec:preliminaries}, we review the standard language of Selmer groups and modular symbols. The precise definition of Kurihara numbers,  which are ``Kolyvagin derivatives of special $L$-values", is given.
In \S\ref{sec:kato-zeta-elts-kolyvagin}, we review Kato's zeta elements and Mazur--Rubin's theory of Kolyvagin systems.
In \S\ref{sec:explicit-reciprocity-law},  we define the explicit reciprocity law with torsion coefficients and prove the formula relating Kato's Kolyvagin systems to Kurihara numbers.
In \S\ref{sec:proof-main-formulas}, we prove the main theorem by developing a ``relative" Kolyvagin system argument.
In \S\ref{sec:p-converse}, the rank one $p$-converse is proved.
In \S\ref{sec:formulation-refined-conjecture},  we formulate the conjectures on the quantitative refinement of the non-vanishing of Kato's Kolyvagin system and the collection of Kurihara numbers for modular forms.  
In Appendix \ref{sec:appendix} (written jointly with Robert Pollack), we compute several numerical examples on the structure of Selmer groups of elliptic curves and higher weight modular forms.

\section{The statement of the main theorems and their arithmetic applications} \label{sec:main-result}
\subsection{Setting the stage}
We first fix notation and briefly introduce the numerical invariants associated to Kurihara numbers in order to state the main results.
\subsubsection{} \label{subsubsec:fixed-isom}
Let $p$ be an odd prime and $f = \sum_{n \geq 1} a_n(f) q^n\in S_k(\Gamma_1(N),\psi)$ be a cuspidal newform with $k \geq 2$.
Denote by $\overline{f} = \sum_{n \geq 1} a_n(\overline{f}) q^n = \sum_{n \geq 1} \overline{a_n(f)} q^n\in S_k(\Gamma_1(N),\psi^{-1})$ the dual modular form of $f$
where $\overline{a_n(f)}$ is the complex conjugation of $a_n(f)$. Fix an abstract field isomorphism $\mathbb{C} \simeq \overline{\mathbb{Q}}_p$.
Let $F$ be the finite extension of $\mathbb{Q}_p$ obtained by adjoining the Fourier coefficients of $f$, $\mathcal{O} = \mathcal{O}_F$ the ring of integers, $\pi$ a uniformizer, and $\mathbb{F} = \mathcal{O}/\pi\mathcal{O}$ the residue field.
Let $V_f$ be the two-dimensional (cohomological) Galois representation associated to $f$ and write
$\rho_f : \mathrm{Gal}(\overline{\mathbb{Q}}/\mathbb{Q}) \to \mathrm{Aut}_{F}(V_f) \simeq \mathrm{GL}_2(F) $ \cite{deligne-modular-galois,jannsen-mixed-motives,scholl-modular-motives}. 
Denote by $T_f$ a Galois stable $\mathcal{O}$-lattice of $V_f$ and write $W_f = V_f / T_f$.  
For $r \in \mathbb{Z}$ and a Galois module $M$, write $M(r) = M \otimes \chi^{\otimes r}_{\mathrm{cyc}}$ where $\chi_{\mathrm{cyc}}$ is the $p$-adic cyclotomic character.

We say that \textbf{$\rho_f$ has large image} if
the image of $\mathrm{Gal}(\overline{\mathbb{Q}}/\mathbb{Q}(\zeta_{p^\infty}))$ under $\rho_f$ contains a conjugate of $\mathrm{SL}_2(\mathbb{Z}_p)$.
For a non-CM modular form $f$, $\rho_f$ has large image for all but finitely many primes \cite{ribet-img-1}.

When $\psi = \mathbf{1}$ (of conductor $N$), we have $f = \overline{f}$ and $k$ is even.
In this case, denote the self-dual twists by $V^\dagger_f = V_f(k/2)$, $T^\dagger_f = T_f(k/2)$, and $W^\dagger_f = W_f(k/2)$ for convenience.

\subsubsection{}
For an integer $m \geq 1$, denote by $\mathcal{N}_m$ be the set of square-free products of primes $\ell$ such that $\ell \nmid Np$, $\ell \equiv 1 \pmod{\pi^m}$ and $a_\ell(\overline{f}) \equiv \psi^{-1}(\ell) \cdot \ell^{k-1} + 1 \pmod{\pi^m}$.
Denote by $\mathcal{P}_m$ the set of primes in $\mathcal{N}_m$.
For $n \in \mathcal{N}_1$, let $I_n$ be the ideal of $\mathcal{O}$ generated by $\ell-1$ and $1-a_\ell(\overline{f}) + \psi^{-1}(\ell) \cdot \ell^{k-1}$ for primes $\ell$ dividing $n$, and denote by $\nu(n)$ the number of prime divisors of $n$. 
\subsubsection{} \label{subsubsec:numerical-invariants-kurihara-numbers}
We now introduce ``Kolyvagin derivatives of special $L$-values", which can be viewed as a discrete variation of the $L$-values.
Denote by $L(\overline{f}, s)$ the complex $L$-function of $\overline{f}$.
Let 
$$\kn^{\mathrm{min},r} = \left\lbrace \widedelta^{\mathrm{min},r}_n \in \mathcal{O}/I_n\mathcal{O} \right\rbrace_{ n \in \mathcal{N}_1 }$$
be the collection of Kurihara numbers for $\overline{f}$ at $s=r$  where $r$ is an integer with $1 \leq r \leq k-1$ (Definition \ref{defn:kurihara-numbers}).
It consists of a family of certain element $\widedelta^{\mathrm{min},r}_n$ in $\mathcal{O}/I_n\mathcal{O}$ (depending on $n$) which are explicitly built out from modular symbols normalized by minimal integral periods reviewed in $\S$\ref{subsec:modular-symbols-kurihara-numbers}.
The numerical invariants associated to $\kn^{\mathrm{min},r}$ are defined by
\begin{align*}
\mathrm{ord} \left(\kn^{\mathrm{min},r} \right) & =  \mathrm{min} \left\lbrace \nu(n) : \widedelta^{\mathrm{min},r}_n \neq 0 \right\rbrace, \\
\partial^{(0)} \left(\kn^{\mathrm{min},r} \right) & = \mathrm{ord}_\pi \left( \dfrac{L(\overline{f},r)}{ (-2 \pi \sqrt{-1})^r \cdot \Omega^{\pm}_{\overline{f}, \mathrm{min}} } \right) , \\
\mathrm{ord}_\pi \left(\widedelta^{\mathrm{min}, k-r}_n \right) &= \mathrm{max} \left\lbrace j : \widedelta^{\mathrm{min}, k-r}_n \in \pi^j \mathcal{O}/I_n\mathcal{O} \right\rbrace , \\
\partial^{(i)} \left(\kn^{\mathrm{min},r} \right) & = \mathrm{min} \left\lbrace \mathrm{ord}_\pi \left(\widedelta^{\mathrm{min}, k-r}_n \right) : n \in \mathcal{N}_1 \textrm{ with } \nu(n) = i \geq 0 \right\rbrace , \\
\partial^{(\infty)} \left(\kn^{\mathrm{min},r} \right) & = \mathrm{min} \left\lbrace \partial^{(i)} \left(\kn^{\mathrm{min},r} \right)  : i \geq 0 \right\rbrace = \mathrm{min} \left\lbrace \mathrm{ord}_\pi \left(\widedelta^{\mathrm{min}, k-r}_n \right) : n \in \mathcal{N}_1  \right\rbrace
\end{align*}
where the sign of the minimal integral period $\Omega^{\pm}_{\overline{f}, \mathrm{min}}$  coincides with that of $(-1)^{r-1}$.
% $\mathrm{ord} \left(\kn^{\mathrm{min},r} \right)$ is the vanishing order of $\kn^{\mathrm{min},r}$
% $\partial^{(i)} \left(\kn^{\mathrm{min},r} \right)$ is the $i$-th Taylor coefficient of $\kn^{\mathrm{min},r}$

Although $\kn^{\mathrm{min},r}$ looks unfamiliar at first, we have the following analogy between $\kn^{\mathrm{min},r}$ and the Taylor expansion of the complex $L$-function $L(\overline{f},s)$ at $s=r$ in mind
\begin{equation} \label{eqn:analogy}
\xymatrix{
\mathrm{ord} \left(\kn^{\mathrm{min},r} \right) \leftrightsquigarrow \mathrm{ord}_{s=r} L(\overline{f},s) , & \partial^{(i)} \left(\kn^{\mathrm{min},r} \right) \leftrightsquigarrow L^{(i)}(\overline{f},r)  .
}
\end{equation}
%We regard $\kn^{\mathrm{min},r}$ as the collection of \emph{Kolyvagin derivatives of $L(\overline{f},r)$}.
%\begin{align*}
%\mathrm{ord} \left(\kn^{\mathrm{min},r} \right)  & \leftrightsquigarrow \mathrm{ord}_{s=r} L(\overline{f},s) , \\
% \partial^{(i)} \left(\kn^{\mathrm{min},r} \right)  & \leftrightsquigarrow \mathrm{ord}_\pi (\textrm{a certain normalization of }  L^{(i)}(\overline{f},r) ) .
%\end{align*}
We call $\partial^{(\infty)} \left(\kn^{\mathrm{min},r} \right)$ the \textbf{analytic fudge factor of $\kn^{\mathrm{min},r}$}. 
See $\S$\ref{sec:formulation-refined-conjecture} for the naming.
%We will see that its finiteness has a deep connection with the corresponding Iwasawa main conjecture.
%We will see that
%two sequences
%\begin{itemize}
%\item $\partial^{(\mathrm{ord} (\kn^{\mathrm{min},r} )+2i)} \left(\ks^{\mathrm{Kato}, k-r} \right)$ and
%\item $\partial^{(\mathrm{ord} (\kn^{\mathrm{min},r} )+2i)} \left(\ks^{\mathrm{Kato}, k-r} \right) - \partial^{(\mathrm{ord} (\kn^{\mathrm{min},r} )+ 2i+2)} \left(\ks^{\mathrm{Kato}, k-r} \right)$
%\end{itemize}
%are decreasing.
The choice of integral periods will not affect our main theorems at all as explained in $\S$\ref{subsubsec:integral-periods}.

\subsection{The main theorems}
We are now ready to state our main technical results.
If the reader wants to check arithmetic applications first,  see Corollary \ref{cor:discrete-bsd} and other results in $\S$\ref{subsec:arithmetic-applications}.
\subsubsection{The structure theorem and a new Birch and Swinnerton-Dyer type formula}
For an $\mathcal{O}$-module $M$, let $M_{\mathrm{div}}$ denote its maximal divisible submodule and also write $M_{/\mathrm{div}} = M/ M_{\mathrm{div}}$.
\begin{thm}[Main Theorem I] \label{thm:main-central-critical}
Let $f \in S_k(\Gamma_0(N))$ be a newform and $p \geq 3$ a prime such that $\rho_f$ has large image.
If the collection of Kurihara numbers $\kn^{\mathrm{min},\dagger} = \kn^{\mathrm{min},k/2}$ for $f$ at $s = k/2$ does not vanish,
then the following arithmetic consequences follow:
\begin{enumerate}
\item[(Non-triv)] Kato's Kolyvagin system $\ks^{\mathrm{Kato}, \dagger}$ for $T^{\dagger}_f$ is non-trivial.
\item[(Str)] The structure of the Bloch--Kato Selmer group $\mathrm{Sel}(\mathbb{Q}, W^{\dagger}_f)$ is determined by $\kn^{\mathrm{min},\dagger}$, i.e.
there exists an isomorphism of co-finitely generated $\mathcal{O}$-modules
\begin{equation} \label{eqn:structure-bloch-kato}
\mathrm{Sel}(\mathbb{Q}, W^{\dagger}_f) \simeq (F/\mathcal{O})^{\oplus \mathrm{ord}(\kn^{\mathrm{min},\dagger} )} \oplus \bigoplus_{i\geq 0} \left( \mathcal{O}/\pi^{e_i} \mathcal{O} \right)^{\oplus 2} 
\end{equation}
where $e_i = \dfrac{1}{2} \cdot \left( \partial^{(\mathrm{ord}(\kn^{\mathrm{min},\dagger} ) + 2i )}(\kn^{\mathrm{min},\dagger} ) - \partial^{(\mathrm{ord}(\kn^{\mathrm{min},\dagger} ) + 2i+2)}(\kn^{\mathrm{min},\dagger} ) \right).$
\item[(``BSD")] The ``Birch and Swinnerton-Dyer type" formulas for $\mathrm{Sel}(\mathbb{Q}, W^{\dagger}_f)$ hold
\begin{align} \label{eqn:bsd-type-formulas}
\begin{split}
\mathrm{cork}_{\mathcal{O}}\left( \mathrm{Sel}(\mathbb{Q}, W^{\dagger}_f) \right) & = \mathrm{ord}(\kn^{\mathrm{min},\dagger} ), \\
  \mathrm{length}_{\mathcal{O}} \left( \mathrm{Sel}(\mathbb{Q}, W^{\dagger}_f)_{/\mathrm{div}} \right) & = \partial^{ (\mathrm{ord}(\kn^{\mathrm{min},\dagger} ) )}(\kn^{\mathrm{min},\dagger} ) - \partial^{(\infty)} (\kn^{\mathrm{min},\dagger} ).
\end{split}
\end{align}
\end{enumerate}
\end{thm}
\begin{proof}
See Theorem \ref{thm:mazur-rubin-main-conjecture}.(2) and Proposition \ref{prop:equivalence-nonvanishing} for (Non-triv).
See $\S$\ref{subsec:proof-theorem-structure-1} and $\S$\ref{subsec:proof-theorem-structure-2} for (Str), and  (``BSD") is an immediate consequence of (Str).
\end{proof}
Theorem \ref{thm:main-central-critical} is independent of weight or the local behavior of $f$ at $p$.
The non-triviality of Kato's Kolyvagin system is the cyclotomic analogue of Kolyvagin's conjecture.

The first formula in (\ref{eqn:bsd-type-formulas}) corresponds to the rank part of Birch and Swinnerton-Dyer conjecture, and the second formula in (\ref{eqn:bsd-type-formulas}) corresponds to the $p$-part of the Birch and Swinnerton-Dyer formula. Therefore, (``BSD") confirms the analogy (\ref{eqn:analogy}).

In Appendix \ref{sec:appendix} in collaboration with Robert Pollack, we determine the structure of Bloch--Kato Selmer groups of various elliptic curves and modular forms by computing $\mathrm{ord}_\pi\left( \widedelta^{\mathrm{min}, \dagger}_n \right)$ numerically.
Many of them cannot be obtained from any known results on Birch and Swinnerton-Dyer conjecture and the Tamagawa number conjecture.

\subsubsection{The non-vanishing of Kurihara numbers}
From Theorem \ref{thm:main-central-critical}, it is natural to ask when $\kn^{\mathrm{min},\dagger}$ does not vanish.
Theorem \ref{thm:non-triviality-kn} below confirms that the non-vanishing of $\kn^{\mathrm{min},\dagger}$ holds for a large class of modular forms and is expected to hold always.
%Furthermore, it is automatically verified when we compute the structure of $\mathrm{Sel}(\mathbb{Q}, W^{\dagger}_f)$ in practice. SEE 1.4.3. 

When $p \nmid N$, denote by $\alpha_p$, $\beta_p$ the roots of $X^2 - a_p(f) \cdot X + p^{k-1}$.
Denote by $\overline{\rho}_f$ the residual representation associated to $f$. Recall that a $p$-ordinary form $f$ is \textbf{$p$-distinguished} if the semi-simplification of $\overline{\rho}_f \vert_{ \mathrm{Gal}(\overline{\mathbb{Q}}_p/\mathbb{Q}_p)  }$ is the sum of two distinct characters.
\begin{thm}[Main Theorem II] \label{thm:non-triviality-kn}
Let $f \in S_k(\Gamma_0(N))$ be a newform and $p \geq 3$ a prime such that $\rho_f$ has large image.
If one of the following conditions is satisfied:
\begin{itemize}
\item[(rk0)] $\mathrm{ord}_{s=k/2}L(f,s) = 0$,
\item[(rk1 $+\epsilon$)] $\mathrm{ord}_{s=k/2}L(f,s) = 1$, $p^2 \nmid N$, $\alpha_p \neq \beta_p$ (when $p \nmid N$), and Assumption \ref{assu:abel-jacobi-injectivity} below holds, 
\item[(IMC at $\mathbf{1}$)] the Iwasawa main conjecture for $T^{\dagger}_f$ localized at the augmentation ideal holds (Conjecture \ref{conj:main-conjecture}), or
\item[(ord)] $p \geq 5$,  $f$ is good ordinary at $p$, and $f$ is $p$-distinguished (Theorem \ref{thm:skinner-urban-wan}),
\end{itemize}
then $\kn^{\mathrm{min},\dagger}$ does not vanish.
Conversely, if $\kn^{\mathrm{min},\dagger}$ does not vanish, then the Iwasawa main conjecture for $T^{\dagger}_f$ localized at the augmentation ideal holds.  In particular, either (rk0) or (rk1 $+\epsilon$) implies (IMC at $\mathbf{1}$).
\end{thm}
\begin{proof}
The (rk0) case is trivial.
The (rk1 $+\epsilon$) case follows from the settlement of Perrin-Riou's conjecture on Kato's zeta elements \cite{bertolini-darmon-venerucci}.
The equivalence between $\kn^{\mathrm{min},\dagger} \neq 0$ and (IMC at $\mathbf{1}$) follows from Theorem \ref{thm:mazur-rubin-main-conjecture}.(2) and Proposition \ref{prop:equivalence-nonvanishing} again.
The (ord) case follows from (IMC at $\mathbf{1}$) and Theorem \ref{thm:skinner-urban-wan}.
\end{proof}
By computing $\mathrm{ord}_\pi\left( \widedelta^{\mathrm{min}, k-r}_n \right)$ for many $n$, we can verify (IMC at $\mathbf{1}$) numerically.
 It is conjectured that $\alpha_p \neq \beta_p$ always holds and  is true when $k=2$. The general weight case follows from Tate's conjecture on algebraic cycles \cite{coleman-edixhoven}.
\begin{assu} \label{assu:abel-jacobi-injectivity}
Under the following standard conjectures, the non-triviality of the Heegner cycle over a suitable imaginary quadratic field in the module of CM cycles is equivalent to the non-triviality of its image in $\mathrm{H}^1(\mathbb{Q}_p, V^\dagger_f)$.
\begin{itemize}
\item[(AJ)] The rational $p$-adic Abel--Jacobi map (given in (\ref{eqn:abel-jacobi-map}) below) is injective on CM cycles on Kuga--Sato variety of weight $k$ and level $N$ over an imaginary quadratic field $\mathcal{K}$ such that every prime divisor of $Np$ splits in $\mathcal{K}$ and $L(f, \chi_{\mathcal{K}/\mathbb{Q}}, k/2) \neq 0$ where $\chi_{\mathcal{K}/\mathbb{Q}}$ is the quadratic character associated to $\mathcal{K}$.
\item[({$\textrm{loc}_p$})] The restriction map $\mathrm{loc}_p : \mathrm{Sel}(\mathbb{Q}, V^\dagger_f) \to \mathrm{H}^1(\mathbb{Q}_p, V^\dagger_f)$ is non-zero.
\end{itemize}
\end{assu}
\begin{rem} \label{rem:abel-jacobi-injectivity}
\begin{enumerate}
\item The existence of such an imaginary quadratic field $\mathcal{K}$ in (AJ) is ensured by the work of Bump--Friedberg--Hoffstein \cite{bump-friedberg-hoffstein} or Murty--Murty \cite{murty-murty-mean}.
\item  (AJ)  holds when $k=2$.
It is conjectured that the $p$-adic Abel--Jacobi map is injective for smooth projective varieties over number fields.
See \cite[(2.1) Conj.(2)]{nekovar-p-adic-abel-jacobi}, \cite[Conj. 5.3.(1)]{bloch-kato}, and \cite[9.15. Conj.]{jannsen-mixed-motives}.
This issue is also observed in \cite[Thm. 1.1.9]{liu-tian-xiao-zhang-zhu}.
%\item 
%When $f$ is good ordinary at $p$, if one replaces (rk1 $+\epsilon$) by its $p$-adic analogue, then we do not need (AJ)  \cite{nekovar-p-adic-height, shnidman-generalized-heegner-cycles, disegni-universal-gross-zagier}.
%We do not have to consider this scenario since the Iwasawa main conjecture completely removes the role of Assumption \ref{assu:abel-jacobi-injectivity} in Theorem \ref{thm:main-central-critical}.
\item ({$\textrm{loc}_p$}) follows from the finiteness of the $\pi$-primary part of Tate--Shafarevich groups of corresponding modular abelian varieties when $k= 2$ and the Selmer corank is one. See \cite{skinner-converse,kim-p-converse} for example.
%injective when $\mathrm{Sel}(\mathbb{Q}, V^\dagger_f)$ is one-dimensional
\end{enumerate}
\end{rem}
The non-triviality of $\ks^{\mathrm{Kato},\dagger}$ for elliptic curves with good ordinary reduction was studied by the author \cite{kim-structure-selmer} and Sakamoto \cite{sakamoto-p-selmer, sakamoto-p-3} independently.
During the writing of this article,  Burungale--Castella--Grossi--Skinner also obtained the non-triviality of Heegner point and Kato's Kolyvagin systems for elliptic curves with good ordinary reduction relaxing the large image assumption \cite{burungale-castella-grossi-skinner-indivisibility}.
In their work, they focused mainly on Tamagawa factors and the full Iwasawa main conjecture was used in the argument. 

We expect that the ordinary assumption in Theorem \ref{thm:non-triviality-kn} will be significantly relaxed soon  thanks to the recent advances on the Iwasawa main conjecture for modular forms at non-ordinary (and general) primes \cite{castella-liu-wan, burungale-skinner-tian-wan, wan-main-conj-ss-ec, castella-ciperiani-skinner-sprung, nakamura-kato-deformation, fouquet-wan}. This means that Theorem \ref{thm:main-central-critical} will become unconditional in a near future.

\subsection{Arithmetic applications} \label{subsec:arithmetic-applications}

\subsubsection{``Discrete" Beilinson--Bloch--Kato and the $p$-parity conjecture}
We complete the following analogue of the Beilinson--Bloch--Kato conjecture  for modular forms at good ordinary primes  \cite{bloch-algebraic-cycles-l-values,beilinson-height-pairings,bloch-kato}, which corresponds to the first analogy in (\ref{eqn:analogy}).
\begin{cor} \label{cor:discrete-bsd}
Let $p \geq 5$ be a prime and $f \in S_k(\Gamma_0(N))$ be a good ordinary and $p$-distinguished newform such that
$\rho_f$ has large image.
Then:
\begin{enumerate}
\item The cyclotomic analogue of Kolyvagin's conjecture holds, i.e.
Kato's Kolyvagin system $\ks^{\mathrm{Kato}, \dagger}$ for $T^\dagger_f$ is non-trivial.
\item The discrete analogue of the Beilinson--Bloch--Kato conjecture for $f$ holds:
$$\mathrm{ord}(\kn^{\mathrm{min},\dagger} ) = \mathrm{cork}_{\mathcal{O}}  \mathrm{Sel}(\mathbb{Q}, W^{\dagger}_f) . $$
\item The $p$-parity conjecture for $f$ holds:
$$\mathrm{ord}_{s=k/2}L(f,s) \equiv \mathrm{cork}_{\mathcal{O}}\mathrm{Sel}(\mathbb{Q}, W^{\dagger}_f)  \pmod{2}.$$
\end{enumerate}
\end{cor}
\begin{proof}
The first and the second statements follow from Theorems \ref{thm:main-central-critical} and \ref{thm:non-triviality-kn}.(ord).
The $p$-parity follows from the second statement and the functional equation for Kurihara numbers (\ref{eqn:functional-equation-delta_n}).
\end{proof}
Corollary \ref{cor:discrete-bsd} illustrates that the $p$-parity conjecture is an immediate consequence of (a small part of) the Iwasawa main conjecture.
The $p$-parity conjecture in this generality was first proved by Nekov\'{a}\v{r} \cite{nekovar-mazur-rubin}.
%See also \cite{pottharst-xiao, johansson-newton-parity} for the $p$-parity conjecture for higher weight modular forms of finite slope.

It is natural to ask the compatibility between the Beilinson--Bloch--Kato conjecture for modular forms and our discrete analogue.
The conjecture below is the comparison between two very different variations of the same $L$-value $L(f, k/2)$ and can be viewed as another quantitative refinement of the non-vanishing question of $\kn^{\mathrm{min},\dagger}$. See $\S$\ref{sec:formulation-refined-conjecture} for another direction.
\begin{conj}[Beilinson--Bloch--Kato] \label{conj:main-central-critical}
Let $f \in S_k(\Gamma_0(N))$ be a newform and $p \geq 3$ a prime such that $\rho_f$ has large image.
Then 
$$\mathrm{ord}_{s=k/2}L(f,s) = \mathrm{ord}(\kn^{\mathrm{min},\dagger} ).$$
\end{conj}

%The following corollary completes the non-vanishing of Kurihara numbers for semi-stable elliptic curves with good ordinary reduction at a prime $p \geq 11$.
%\begin{cor}
%Let $E$ be a semi-stable elliptic curve over $\mathbb{Q}$ and $p \geq 11$ a prime.
%\begin{enumerate}
%\item If $E$ has good ordinary reduction at $p$, then $\kn^{\mathrm{min},\dagger}$ does not vanish.
%\item If $E$ has analytic rank one, then  $\kn^{\mathrm{min},\dagger}$ does not vanish.
%\end{enumerate}
%\end{cor}
\subsubsection{The analytic rank zero case}
Another simple but interesting corollary is the following analytic rank zero case.
\begin{cor} \label{cor:analytic-rank-zero}
Let $f \in S_k(\Gamma_0(N))$ be a newform and $p \geq 3$ a prime such that $\rho_f$ has large image.
If $L(f, k/2) \neq 0$, then
\[
\mathrm{length}_{\mathcal{O}} \left( \mathrm{Sel}(\mathbb{Q}, W^\dagger_{f} ) \right) = \mathrm{ord}_\pi \left( \dfrac{L(f,k/2)}{ (-2 \pi \sqrt{-1})^{k/2} \cdot \Omega^{\pm}_{f, \mathrm{min}} } \right) - \partial^{(\infty)} \left( \kn^{\mathrm{min},\dagger} \right)  
\]
where  the sign of the minimal integral period  coincides with that of $(-1)^{k/2-1}$.
\end{cor}
Corollary \ref{cor:analytic-rank-zero} generalizes to all critical points and the $\psi \neq \mathbf{1}$ case. 
\begin{thm} \label{thm:main-all-critical}
Let $f \in S_k(\Gamma_1(N), \psi)$ be a newform and $p \geq 3$ a prime such that $\rho_f$ has large image.
For an integer $r$ with $1 \leq r \leq k-1$, if
$(1 - a_p(\overline{f}) p^{-r} + \psi^{-1}(p) p^{k-1-2r} ) \cdot L(\overline{f}, r) \neq 0 $, 
then
\[
\mathrm{length}_{\mathcal{O}} \left( \mathrm{Sel}(\mathbb{Q}, W_{\overline{f}}(r) ) \right) = \mathrm{ord}_\pi \left( \dfrac{L(\overline{f},r)}{ (-2 \pi \sqrt{-1})^r \cdot \Omega^{\pm}_{\overline{f}, \mathrm{min}} } \right) - \partial^{(\infty)} \left( \kn^{\mathrm{min},r} \right)  
\]
where  the sign of the minimal integral period  coincides with that of $(-1)^{r-1}$.
In particular, $\mathrm{Sel}(\mathbb{Q}, W_{\overline{f}}(r))$ is trivial for all but finitely many places of the Hecke field of $f$.
\end{thm}
\begin{proof}
See Proposition \ref{prop:corank-zero} and Remark \ref{rem:corank-zero}.
\end{proof}
Theorem \ref{thm:main-all-critical} does not require the Iwasawa main conjecture for $T_f(k-r)$ as well as any assumption on weight or local automorphic representation.
See Corollary \ref{cor:main-all-critical-before-dual-exp} and \S\ref{sec:formulation-refined-conjecture} for the precise relation between  Theorem \ref{thm:main-all-critical} and Kato--Fontaine--Perrin-Riou's formulation of Tamagawa number conjecture. 
It seems even difficult to find any precise conjectural formula for Bloch--Kato Selmer groups in this generality mainly due to the difficulty of the ``good" choice of periods.

Thanks to the work of Jacquet--Shalika \cite{jacquet-shalika} (cf. \cite[Thm. 13.5]{kato-euler-systems}), we have 
$L(\overline{f}, s) \neq 0$ if $\mathrm{Re}(s) \geq \frac{k+1}{2}$.
The functional equation implies that $L(\overline{f}, r) \neq 0$ if $r \neq k/2$.
In addition, 
$1 - a_p(\overline{f}) p^{-r} + \psi^{-1}(p) p^{k-1-2r} $ vanishes
 only when
\begin{enumerate}
\item[(NVE$_{\mathrm{cris}}$)]
$p$ does not divide $N$, $r = \frac{k-1}{2}$, and $a_p(\overline{f}) = (1 + \psi^{-1}(p) )\cdot p^{\frac{k-1}{2}}$, 
\item[(NVE$_{\mathrm{st}}$)]
$p$ divides $N$ exactly once, $\psi^{-1}(p)=1$, $r = \frac{k-2}{2}$, and $a_p(\overline{f}) =  p^{\frac{k-2}{2}}$, or
\item[(NVE$_{\psi}$)]
$p$ divides $N$, $\mathrm{ord}_p(\mathrm{cond}(\psi^{-1})) = \mathrm{ord}_p(N)$, $r = \frac{k-1}{2}$, and $a_p(\overline{f}) = p^{\frac{k-1}{2}}$.
\end{enumerate}
See \cite[Lem. 10.1.6]{epw2} for the proof. When $\psi^{-1}(p) = 1$ (e.g. $\psi = \mathbf{1}$), it is expected that (NVE$_{\mathrm{cris}}$) never happens \cite[4.1 Thm.]{coleman-edixhoven}.
Note that  (NVE$_{\mathrm{cris}}$) and (NVE$_{\mathrm{st}}$) are different from the exceptional zeros in the $p$-adic Birch and Swinnerton-Dyer conjecture (cf. \cite[\S15]{mtt} and \cite[p. 2]{benois-near-central}).

\subsubsection{Modular forms of weight two at good ordinary primes} \label{subsubsec:weight-two-good-ordinary}
When $f \in S_2(\Gamma_0(N))$ and $p^2 \nmid N$, denote by $\kn^{\mathrm{can},\dagger}$ the collection of Kurihara numbers for $f$ at $s = 1$ normalized by the canonical periods \cite{vatsal-cong} which are minimally integral due to the work of Ash--Stevens \cite{ash-stevens}.
\begin{cor} \label{cor:good-ordinary-weight-two}
Let $f \in S_2(\Gamma_0(N))$ be a newform and $p \geq 5$ a good ordinary prime for $f$ such that
 $\rho_f$ has large image.
Then $\ks^{\mathrm{Kato}, \dagger}$ is non-trivial, and equivalently $\kn^{\mathrm{can},\dagger}$ does not vanish, and
there exists an isomorphism of co-finitely generated $\mathcal{O}$-modules
$$\mathrm{Sel}(\mathbb{Q}, W_f(1)) \simeq (F/\mathcal{O})^{\oplus \mathrm{ord}(\kn^{\mathrm{can},\dagger} )} \oplus \bigoplus_{i\geq 0} \left( \mathcal{O}/\pi^{e_i} \mathcal{O} \right)^{\oplus 2} $$
where 
$e_i = \dfrac{1}{2} \cdot \left( \partial^{(\mathrm{ord}(\kn^{\mathrm{can},\dagger} ) + 2i )}(\kn^{\mathrm{can},\dagger} ) - \partial^{(\mathrm{ord}(\kn^{\mathrm{can},\dagger} ) + 2i+2)}(\kn^{\mathrm{can},\dagger} ) \right)$.
 Thus, we have
\begin{align*}
\mathrm{cork}_{\mathcal{O}}\mathrm{Sel}(\mathbb{Q}, W_f(1)) & = \mathrm{ord}(\kn^{\mathrm{can},\dagger} ), \\
  \mathrm{length}_{\mathcal{O}}\mathrm{Sel}(\mathbb{Q}, W_f(1))_{/\mathrm{div}} & = \partial^{ (\mathrm{ord}(\kn^{\mathrm{can},\dagger} ) )}(\kn^{\mathrm{can},\dagger} ) - \partial^{(\infty)} (\kn^{\mathrm{can},\dagger} ).
\end{align*}
\end{cor}
\begin{proof}
This is the $k =2$ case of Theorem \ref{thm:main-central-critical} with the combination of Theorem \ref{thm:non-triviality-kn}.(ord)
\end{proof}
It is known that the Dirichlet density of the ordinary primes for a modular form of weight two is one \cite[(2.7)]{ogus-hodge-crystalline}, \cite[$\S$7]{hida-jams-2013}. 
\subsubsection{Modular forms of weight two with analytic rank $\leq 1$}
When the analytic rank is $\leq 1$, we have the following structural result on modular abelian varieties, which improves the work of Gross--Zagier and Kolyvagin \cite{gross-zagier-original, kolyvagin-euler-systems}.  It goes beyond the standard Euler system argument,  but it does not use the Iwasawa main conjecture at all (cf. \cite{jetchev-skinner-wan, castella-ciperiani-skinner-sprung}).
\begin{cor} \label{cor:analytic-rank-one-weight-two}
Let $f \in S_2(\Gamma_0(N))$ be a newform and $p \geq 3$ a prime such that $p^2 \nmid N$ and $\rho_f$ has large image.
Denote by $A_f$ the modular abelian variety of $\mathrm{GL}_2$-type associated to $f$,  and write $R = \mathrm{End}_{\overline{\mathbb{Q}}}(A_f)$.
If $\mathrm{ord}_{s=1}L(f, s) \leq 1$, then we have
\begin{align*}
\mathrm{cork}_{\mathcal{O}}\mathrm{Sel}(\mathbb{Q}, W_f(1)) & = \mathrm{dim}_F ( A_f(\mathbb{Q}) \otimes_{R} F ) = \mathrm{ord}(\kn^{\mathrm{can},\dagger} ) = \mathrm{ord}_{s=1}L(f, s), \\
\sha(A_f/\mathbb{Q})[\pi^\infty] & \simeq \bigoplus_{i\geq 0} \left( \mathcal{O}/\pi^{e_i} \mathcal{O} \right)^{\oplus 2} 
\end{align*}
where $e_i = \dfrac{1}{2} \cdot \left( \partial^{(\mathrm{ord}(\kn^{\mathrm{can},\dagger} ) + 2i )}(\kn^{\mathrm{can},\dagger} ) - \partial^{(\mathrm{ord}(\kn^{\mathrm{can},\dagger} ) + 2i+2)}(\kn^{\mathrm{can},\dagger} ) \right)$. In particular, 
$$  \mathrm{length}_{\mathcal{O}}\sha(A_f/\mathbb{Q})[\pi^\infty] = \partial^{ (\mathrm{ord}(\kn^{\mathrm{can},\dagger} ) )}(\kn^{\mathrm{can},\dagger} ) - \partial^{(\infty)} (\kn^{\mathrm{can},\dagger} ).$$
\end{cor}
\begin{proof}
It follows easily from Theorem \ref{thm:main-central-critical} and the work of Gross--Zagier and Kolyvagin \cite{gross-zagier-original, kolyvagin-euler-systems}. Since the weight of $f$ is two and $\mathrm{ord}_{s=1}L(f, s) \leq 1$, Assumption \ref{assu:abel-jacobi-injectivity} is automatic.
\end{proof}
Comparing Corollary \ref{cor:analytic-rank-one-weight-two} with the $\pi$-part of the BSD formula for $A_f$,  it is natural to ask the relation between $\partial^{(\infty)} (\kn^{\mathrm{can},\dagger} )$ and the $\pi$-valuation of the product of local Tamagawa factors of $A_f$.
 See $\S$\ref{sec:formulation-refined-conjecture} for further details.

\subsubsection{The computational upper bound of Selmer ranks}
The following upper bound of Selmer ranks is immediate from Theorem \ref{thm:main-central-critical}.
\begin{cor} \label{cor:upper-bound}
Let $f \in S_k(\Gamma_0(N))$ be a newform and $p \geq 3$ a prime such that $\rho_f$ has large image.
If $\widedelta^{\mathrm{min},\dagger}_n \neq 0$ for some $n \in \mathcal{N}_1$, then 
$$\mathrm{cork}_{\mathcal{O}}\mathrm{Sel}(\mathbb{Q}, W^{\dagger}_f)  \leq  \nu(n) .$$
\end{cor}
This bound is completely different from the upper bound coming from $p$-adic $L$-functions via the $p$-adic Birch and Swinnerton-Dyer conjecture for modular forms of finite slope \cite[$\S$18]{kato-euler-systems}.

\subsubsection{The rank one $p$-converse for modular forms of higher weight} \label{subsubsec:p-converse-intro}
We discuss the rank one $p$-converse to the theorem of Gross--Zagier, Kolyvagin,  Nekov\'{a}\v{r}, and Zhang for modular forms of higher weight \cite{gross-zagier-original, kolyvagin-euler-systems, nekovar-kolyvagin, nekovar-p-adic-height, shou-wu-zhang-heegner-cycles}.
The theorem\footnote{This is the generalization of the theorem of Gross--Zagier and Kolyvagin to modular forms of higher weight.} of Nekov\'{a}\v{r} and Zhang is recalled in Theorem \ref{thm:nekovar-zhang}.
We recall some background briefly and state the converse result.% We continue the discussion in $\S$\ref{sec:p-converse}. 

Let $\mathrm{KS}^{k-2}_N$ be the $(k-1)$-dimensional Kuga--Sato variety of weight $k$ and full level $N$.
For a number field $K$, let $\mathrm{CH}^{k/2} (\mathrm{KS}^{k-2}_N/K)_0 $ be the group of homologically trivial cycles of codimension $k/2$ on $\mathrm{KS}^{k-2}_N$ defined over $K$ modulo rational equivalence.
Let
\begin{equation} \label{eqn:abel-jacobi-map}
\Phi_{K} \otimes \mathbb{Q}_p  : \mathrm{CH}^{k/2} (\mathrm{KS}^{k-2}_N/K)_0 \otimes \mathbb{Q}_p \to \mathrm{H}^1(K, \mathrm{H}^{k-1}_{\mathrm{\acute{e}t}} (\mathrm{KS}^{k-2}_{N, \overline{K}} , \mathbb{Z}_p(k/2))  )  \otimes \mathbb{Q}_p
\end{equation}
be the (rational) $p$-adic Abel--Jacobi map where $\mathrm{KS}^{k-2}_{N, \overline{K}}$ is the base change of $\mathrm{KS}^{k-2}_{N}$ to $\overline{K}$.
This map is conjectured to be injective (Assumption \ref{assu:abel-jacobi-injectivity}).
Let $f \in S_k(\Gamma_0(N))$ be a newform with $k \geq 4$.
Denote by 
\begin{align*}
\Phi_{f,K} \otimes F & : \mathrm{CH}^{k/2} (\mathrm{KS}^{k-2}_N/K)_0 \otimes F \to \mathrm{Sel}(K, V^\dagger_f), \\
\Phi_{f,K} \otimes F/\mathcal{O} & : \mathrm{CH}^{k/2} (\mathrm{KS}^{k-2}_N/K)_0 \otimes F/\mathcal{O} \to \mathrm{Sel}(K, W^\dagger_f) 
\end{align*}
the $f$-isotypic quotients of the rational and discrete versions of the $p$-adic Abel--Jacobi map, respectively.
The \textbf{$\pi$-primary part of the Tate--Shafarevich group of $W^\dagger_f$ over $K$} is defined by
$\sha(f^\dagger/K)[\pi^\infty] := \mathrm{coker}\left( \Phi_{f,K} \otimes F/\mathcal{O} \right) $.

For an imaginary quadratic field $\mathcal{K}$ of discriminant $-D_{\mathcal{K}} <0$ satisfying 
 $(D_{\mathcal{K}}, p)=1$ and
 every prime divisor of $N$ splits in $\mathcal{K}$ (the Heegner hypothesis),
denote by
$\mathrm{CM}_{k/2}(X(N)_{\mathcal{K}}) \subseteq \mathrm{CH}^{k/2} (\mathrm{KS}^{k-2}_N/\mathcal{K})_0$
 the module of CM cycles on $\mathrm{KS}^{k-2}_{N}$ \cite[$\S$2 and $\S$5.3]{shou-wu-zhang-heegner-cycles}. 

Let $s_{\mathcal{K}} \in \mathrm{CM}_{k/2}(X(N)_{\mathcal{K}})$ be the Heegner cycle of $f$ over such an imaginary quadratic field $\mathcal{K}$
and  $\kappa^{\mathrm{Hg},\dagger,  \mathcal{K}}_1 \in \mathrm{Sel}(\mathcal{K}, V^\dagger_f)^{-w(f)}$ be the $p$-adic Abel--Jacobi image of $s_{\mathcal{K}}$
where $\mathrm{Sel}(\mathcal{K}, V^\dagger_f)^{-w(f)}$ is the submodule of $\mathrm{Sel}(\mathcal{K}, V^\dagger_f)$, on which complex conjugation acts as $-w(f)$, and $w(f)$ is the sign of functional equation of $L(f, s)$.

\begin{thm} \label{thm:p-converse-higher-weight-ordinary}
Let $p \geq 5$ be a prime and $f \in S_k(\Gamma_0(N))$ be a good ordinary and $p$-distinguished newform with $k \geq 4$ such that $\rho_f$ has large image.
If
\begin{itemize}
\item[(cork1)]  $\mathrm{cork}_{\mathcal{O}}\mathrm{Sel}(\mathbb{Q}, W^\dagger_f) = 1$, and
\item[({$\mathrm{loc}_p$})] the restriction map $\mathrm{loc}_p : \mathrm{Sel}(\mathbb{Q}, V^\dagger_f) \to \mathrm{H}^1(\mathbb{Q}_p, V^\dagger_f)$ is non-zero (Assumption \ref{assu:abel-jacobi-injectivity}),
\end{itemize}
then $\kappa^{\mathrm{Hg},\dagger,  \mathcal{K}}_1$ is non-zero where
 $\mathcal{K}$ is an imaginary quadratic field such that every prime divisor of $Np$ splits in $\mathcal{K}$  and  $ L(f, \chi_{\mathcal{K}/\mathbb{Q}}, k/2) \neq 0$.
If we further assume that
\begin{itemize}
\item[(non-deg)] the height pairing on $\mathrm{CM}_{k/2}(X(N)_{\mathcal{K}})$ is non-degenerate over the same imaginary quadratic field $\mathcal{K}$, 
\end{itemize}
then
$$\mathrm{ord}_{s=k/2}L(f, s)  =  \mathrm{cork}_{\mathcal{O}} \left( \mathrm{Sel}(\mathbb{Q}, W^\dagger_f) \right) = \mathrm{dim}_{F} \left( \mathrm{im}(\Phi_{f, \mathbb{Q}} \otimes F) \right)    = 1, $$
and $\mathrm{length}_{\mathcal{O}} \sha(f^\dagger/\mathbb{Q})[\pi^\infty]  < \infty $.
\end{thm}
\begin{proof}
It follows from Theorem \ref{thm:skinner-urban-wan} and Theorem \ref{thm:p-converse-higher-weight}.
\end{proof}

%$, \mathrm{length}_{\mathcal{O}}\sha(f^\dagger/\mathbb{Q})[\pi^\infty] < \infty$

%the following exact sequence
%\[
%\xymatrix{
%0 \ar[r] & \mathrm{Im} \left( \Phi_{f,K} \otimes F/\mathcal{O} \right)  \ar[r] & \mathrm{Sel}(K, W^\dagger_f) \ar[r] & \sha(f^\dagger/K)[\pi^\infty] \ar[r] & 0 .
%}
%\]

\subsection{Remarks on Theorems \ref{thm:main-central-critical} and  \ref{thm:main-all-critical}}

\subsubsection{} \label{subsubsec:known-limitations-bsd-bloch-kato}
In the framework of the Tamagawa number conjecture \cite{bloch-kato,fontaine-special-values,kato-iwasawa-hodge,kato-lecture-1,fontaine-perrin-riou,perrin-riou-book, kato-icm-2002, fukaya-kato}, all the known results on the exact bound of Bloch--Kato Selmer groups of modular forms require \emph{all} the assumptions below:
\begin{itemize}
\item[(low)] the analytic rank $\leq 1$,
\item[(IMC)] the \emph{full} Iwasawa main conjecture, \emph{and}
\item[(ordFL)] the good ordinary assumption ($p \nmid N \cdot a_p(f) $) or the Fontaine--Laffaille assumption ($p \nmid N$ and $2 \leq k \leq p-1$).
\end{itemize}
Any of these assumptions has played a crucial role to deduce the exact Birch and Swinnerton-Dyer type formula in the literature
\cite{coates-wiles-bsd-1977, gross-zagier-original, rubin-tate-shafarevich, kolyvagin-euler-systems,rubin-main-conj-cm,rubin-p-converse, kato-euler-systems,skinner-urban-icm, kobayashi-gross-zagier,skinner-urban, wei-zhang-mazur-tate,wan_hilbert, berti-bertolini-venerucci, jetchev-skinner-wan, castella-cambridge,wan-iwasawa-2017, skinner-converse, burungale-tian-p-converse, castella-grossi-lee-skinner, thackeray-thesis-jnt, kim-p-converse,wan-main-conj-ss-ec,castella-ciperiani-skinner-sprung,kazim-pollack-sasaki,burungale-skinner-tian-wan, fouquet-wan, sweeting-kolyvagin, longo-vigni-tamagawa}.
All the approaches are based heavily on the Euler system method and various generalizations of Gross--Zagier formula as well as the verification of certain Iwasawa main conjectures.

When the analytic rank is zero, we refer to \cite[Thm. 3.36]{skinner-urban}, \cite[Prop. 14.21]{kato-euler-systems}, and \cite[Cor. 6.2]{wan-iwasawa-2017} 
in order to observe how Hida theory \cite[$\S$3.3.11]{skinner-urban} and Fontaine--Laffaille modules \cite[$\S$14.17]{kato-euler-systems} are used.
These kinds of restrictions still remain in the similar results for motives arising from Shimura varieties associated to higher rank groups.
%Theorem \ref{thm:main-central-critical} can be viewed as a refinement of the combination of the Iwasawa main conjecture, the Euler characteristic formula \cite[Thm. 4.1]{greenberg-lnm}, and the interpolation formula of $p$-adic $L$-functions \cite{mtt} for the p-part  of Bloch--Kato Selmer groups of rank zero.

In our main results, all three assumptions above are significantly weakened.
In particular, Theorem \ref{thm:main-all-critical} is completely independent of (IMC) and (ordFL).
This is possible since our approach does not use any Hida theory or integral $p$-adic Hodge theory.

%Comparing with the classical application of the Iwasawa main conjecture to the Tamagawa number conjecture (briefly recalled in $\S$\ref{subsubsec:classical-applications}), Theorem \ref{thm:main-central-critical} indicates that the Iwasawa main conjecture should be considered \emph{prime by prime}. 

\subsubsection{}
Kolyvagin first studied the \emph{structure} of Selmer groups of elliptic curves over an imaginary quadratic field satisfying the Heegner hypothesis via Heegner point Kolyvagin systems \cite{kolyvagin-structure-sha, kolyvagin-selmer}.
Later, Mazur--Rubin developed the theory of Kolyvagin systems and obtained the structure theorem of $p$-strict Selmer groups of elliptic curves over $\mathbb{Q}$ via Kato's Kolyvagin systems \cite{mazur-rubin-book}.
In order to utilize these structure theorems practically, one needs to compute the divisibility index of each Kato's or Heegner point Kolyvagin system class explicitly.
However, this computation seems quite non-trivial due to their abstract nature.

Kurihara independently developed the theory of Kolyvagin systems of Gauss sum type and obtained the structure theorem of Selmer groups of elliptic curves over $\mathbb{Q}$ via Kurihara numbers \cite{kurihara-iwasawa-2012, kurihara-munster, kurihara-analytic-quantities}.
His argument requires several non-trivial assumptions including the Iwasawa main conjecture and the $\mu = 0$ conjecture.

In this sense, our formulas take only the advantages from both approaches. % since they are numerically computable and do not require the full Iwasawa main conjecture or the $\mu = 0$ conjecture.

\subsubsection{}
%When we practically compute the structure of $\mathrm{Sel}(\mathbb{Q}, W^\dagger_f)$ following Theorem \ref{thm:main-central-critical}, the non-vanishing assumption of $\kn^{\mathrm{min}, \dagger}$ is not essential.
It seems that there is no good way to compute Fontaine--Perrin-Riou's local Tamagawa ideals \cite{fontaine-perrin-riou} for \emph{higher weight} modular forms in general.
In other words,  it seems difficult to verify the $p$-part of the Tamagawa number conjecture for modular forms numerically even when the analytic rank is zero.
Our formula completely avoids the computation of local Tamagawa ideals (in particular, at $p$) but captures the exact bound of Selmer group.
Indeed, the analytic fudge factor $\partial^{(\infty)} \left( \kn^{\mathrm{min},r} \right)$ plays the role of local Tamagawa ideals as in Theorems \ref{thm:main-central-critical} and \ref{thm:main-all-critical}. 
%Also, the $n$-descent method \cite[Chap. X]{silverman} does not generalize to higher weight modular forms.
Since we bypass the computation of the local Tamagawa ideals, we do not have any exact control of $\partial^{(\infty)} \left( \kn^{\mathrm{min},r} \right)$, and this is the cost we should pay.  See $\S$\ref{sec:formulation-refined-conjecture}.
However,  if the reader focuses mainly on Bloch--Kato Selmer groups, this cost is negligible.

\subsubsection{} \label{subsubsec:integral-periods}
It is widely believed that the choice of the ``correct" integral complex periods is essential to obtain the formulas for the exact bound of Bloch--Kato Selmer groups.
Although a delicate choice of \emph{canonical} integral periods is an important topic in the arithmetic of modular forms (e.g. \cite{vatsal-cong}), our approach completely bypasses this issue. \emph{Any} non-minimal integral periods can play the exactly same role in Theorems  \ref{thm:main-central-critical} and \ref{thm:main-all-critical} without any modification because the difference will be cancelled out.
The minimal integral periods are used only for convenience.
This feature essentially  comes from the rigidity of Kolyvagin systems and this nature was not observed well in the case of elliptic curves \cite{kim-structure-selmer}.

%\subsubsection{} \label{subsubsec:classical-applications}
%As described in the following diagram, Theorem \ref{thm:main-central-critical} can be viewed as a refinement of the combination of the Euler characteristic formula  \cite[Thm. 4.1]{greenberg-lnm} and the interpolation formula of $p$-adic $L$-functions \cite{mtt} for the study of Bloch--Kato Selmer groups of rank zero.
%\begin{equation} \label{eqn:application-diagram}
%\begin{split}
%\xymatrix{
%{\substack{\textrm{Iwasawa main conjecture} \\ \textrm{without $p$-adic $L$-functions} } } \ar@{<=>}[d]^-{\textrm{if ordinary}} \ar@{=>}[r] & {\substack{\textrm{Iwasawa main conjecture} \\ \textrm{without $p$-adic $L$-functions} \\ \textrm{at the augmentation ideal}  \\ \textrm{(weaker input)} } } \ar@{<=>}[r] & {\substack{\textrm{The non-triviality of } \\ \textrm{$\ks^{\mathrm{Kato}, \dagger}$} } } \ar@{<=>}[d] \\
%{\substack{\textrm{Iwasawa main conjecture} \\ \textrm{with $p$-adic $L$-functions} } } \ar@{~>}[d]^-{ {\substack{\textrm{Euler characteristic formula,} \\ \textrm{the interpolation formula} \\ \textrm{for $p$-adic $L$-functions} } } } & & {\substack{\textrm{The non-triviality of } \\ \textrm{$\kn^{\mathrm{min}, \dagger}$} } }
% \ar@{~>}[d]_-{ \textrm{Theorem \ref{thm:main-central-critical}} } \\
%{\substack{\textrm{The exact bound of } \\ \textrm{Bloch--Kato Selmer groups} \\ \textbf{when finite} } } & & {\substack{\textrm{The structure of} \\ \textrm{ Bloch--Kato Selmer groups} \\ \textrm{(better output)} } }
%}
%\end{split}
%\end{equation}

\section{Selmer groups and modular symbols} \label{sec:preliminaries}
\subsection{Galois cohomology} \label{subsec:galois-cohomology}
\subsubsection{}
Let $p$ be an odd prime.
Let $F$ be a finite extension of $\mathbb{Q}_p$ with ring of integers $\mathcal{O}$ and uniformizer $\pi$.
Let $V$ be a finite dimensional $F$-vector space endowed with continuous action of $\mathrm{Gal}(\overline{\mathbb{Q}}/\mathbb{Q})$.
Let $T \subseteq V$ be a Galois stable $\mathcal{O}$-lattice and write $W = V/T$.
\subsubsection{} 
Let $K$ be a non-archimedean local field of residual characteristic $\ell \neq p$.
The finite local condition of $\mathrm{H}^1(K, V)$ is defined  by the unramified subspace
$\mathrm{H}^1_f(K, V) = \mathrm{H}^1_{\mathrm{ur}}(K, V)$.
Also, $\mathrm{H}^1_f(K, T)$ is defined by the preimage of $\mathrm{H}^1_f(K, V)$ in $\mathrm{H}^1(K, T)$ and $\mathrm{H}^1_f(K, W)$ is defined by the image of $\mathrm{H}^1_f(K, V)$ in $\mathrm{H}^1(K, W)$.
For each $m \geq 1$,  $\mathrm{H}^1_f(K, W[\pi^m])$ and $\mathrm{H}^1_f(K, T/\pi^mT)$ are defined similarly.

%For an $\mathcal{O}$-module $M$, let $M_{\mathrm{div}}$ denote its maximal divisible submodule and also write $M_{/\mathrm{div}} = \dfrac{M}{M_{\mathrm{div}}}$.
%\begin{lem}[NEEDED?]
%Let $\ell$ be a prime not equal to $p$ or $\infty$. Then
%\begin{enumerate}
%\item $\mathrm{H}^1_f(K, W) = \mathrm{H}^1_{\mathrm{ur}}(K,W)_{\mathrm{div}} \subseteq \mathrm{H}^1_{\mathrm{ur}}(K,W)$.
%\item $\mathrm{H}^1_{\mathrm{ur}}(K,T)  \subseteq \mathrm{H}^1_{f}(K,T)$ with finite index and $\mathrm{H}^1_{/f}(K, T)$ is torsion-free.
%\item Write $\mathcal{W} = (W^{I_K})_{/\mathrm{div}}$.
%There are natural isomorphisms
%\[
%\xymatrix{
%\dfrac{\mathrm{H}^1_{\mathrm{ur}}(K,W)}{\mathrm{H}^1_f(K, W)}
%\simeq
%\dfrac{\mathcal{W}}{(\mathrm{Fr} - 1)\mathcal{W}} ,
%&
%\dfrac{\mathrm{H}^1_f(K, T)}{\mathrm{H}^1_{\mathrm{ur}}(K,T)}
%\simeq
%\mathcal{W}^{\mathrm{Fr}=1}
%}
%\]
%where $\mathrm{Fr}$ is the arithmetic Frobenius in $G_K/I_K$.
%\item If $T$ is unramified, then 
%$\mathrm{H}^1_f(K, T) = \mathrm{H}^1_{\mathrm{ur}}(K,T)$, and 
%$\mathrm{H}^1_f(K, W) = \mathrm{H}^1_{\mathrm{ur}}(K,W)$.
%\end{enumerate}
%\end{lem}
%\begin{proof}
%See \cite[Lemma 1.3.5]{rubin-book}.
%\end{proof}
%Considering the map $\mathrm{Fr} - 1 : \mathcal{W} \to \mathcal{W}$ and its kernel and cokernel, we have equality
%$\#\frac{\mathrm{H}^1_{\mathrm{ur}}(K,W)}{\mathrm{H}^1_f(K, W)} = \#\frac{\mathrm{H}^1_f(K, T)}{\mathrm{H}^1_{\mathrm{ur}}(K,T)}$
%and the size captures the local Tamagawa factor following Fontaine--Perrin-Riou \cite[I.4.2.2.~Prop.]{fontaine-perrin-riou}.
%See \cite[$\S$1.3]{rubin-book} for details.

\subsubsection{}
Let $K$ be a finite extension of $\mathbb{Q}_p$.
We assume that $V$ is de Rham as a representation of $\mathrm{Gal}(\overline{K}/K)$.
Define
$\mathrm{H}^1_f(K, V) = \mathrm{ker} \left( \mathrm{H}^1(K, V) \to \mathrm{H}^1(K, \mathbf{B}_{\mathrm{cris}} \otimes V) \right) $
where $\mathbf{B}_{\mathrm{cris}}$ is Fontaine's crystalline period ring.
%We define $F$-subspaces
%$\mathrm{H}^1_e(K, V) \subseteq \mathrm{H}^1_f(K, V) \subseteq \mathrm{H}^1_g(K, V) \subseteq \mathrm{H}^1(K, V)$
%by
%\begin{align*}
%\mathrm{H}^1_e(K, V) & = \mathrm{ker} \left( \mathrm{H}^1(K, V) \to \mathrm{H}^1(K, \mathbf{B}^{\varphi=1}_{\mathrm{cris}} \otimes V) \right) , \\
%\mathrm{H}^1_f(K, V) & = \mathrm{ker} \left( \mathrm{H}^1(K, V) \to \mathrm{H}^1(K, \mathbf{B}_{\mathrm{cris}} \otimes V) \right) , \\
%\mathrm{H}^1_g(K, V) & = \mathrm{ker} \left( \mathrm{H}^1(K, V) \to \mathrm{H}^1(K, \mathbf{B}_{\mathrm{dR}} \otimes V) \right) .
%\end{align*}
Also, $\mathrm{H}^1_f(K, T)$, $\mathrm{H}^1_f(K, T/\pi^mT)$, $\mathrm{H}^1_f(K, W)$, and $\mathrm{H}^1_f(K, W[\pi^m])$ are defined similarly.
%For $* \in \lbrace e,f,g \rbrace$, $\mathrm{H}^1_*(K, T)$ and $\mathrm{H}^1_*(K, W)$ are defined by using the inverse image and the image, respectively, as before.
See \cite[$\S$3]{bloch-kato} for details.
\begin{thm}[Bloch--Kato] \label{thm:bloch-kato-exponential}
If $V$ is  de Rham and $\mathrm{det}(1 - \varphi \cdot X : \mathbf{D}_{\mathrm{dR}}(V))\vert_{X=1} \neq 0$, then
$$\mathrm{exp} : \mathbf{D}_{\mathrm{dR}}(V) / \mathrm{Fil}^0\mathbf{D}_{\mathrm{dR}}(V) \simeq \mathrm{H}^1_f(K, V).$$
\end{thm}
\begin{proof}
See \cite[Thm. 4.1.(ii)]{bloch-kato}.
\end{proof}
We recall
$\mathrm{det}(1 - \varphi \cdot X : \mathbf{D}_{\mathrm{dR}}( V_{\overline{f}}(r) ))\vert_{X=1} = 1 - a_p(\overline{f}) \cdot p^{-r} - \psi^{-1}(p) \cdot p^{k-1-2r}$
\cite[$\S$14.10]{kato-euler-systems}.
\subsubsection{}
When $\ell \equiv 1 \pmod{\pi^m}$ and $T/\pi^mT$ is unramified at $\ell$, the transverse local condition for $T/\pi^mT$ is defined by
$\mathrm{H}^1_{\mathrm{tr}}(\mathbb{Q}_\ell, T/\pi^mT ) = 
\mathrm{H}^1(\mathbb{Q}_\ell(\zeta_\ell)/\mathbb{Q}_\ell, \mathrm{H}^0(\mathbb{Q}_\ell(\zeta_\ell), T/\pi^mT) ) $
where $\zeta_\ell$ is a primitive $\ell$-th root or unity.
For $\ell \in \mathcal{N}_m$,  we have the \textbf{finite-singular comparison isomorphism}
$\phi^{\mathrm{fs}}_\ell : \mathrm{H}^1_{f}(\mathbb{Q}_\ell, T/\pi^mT) \simeq \mathrm{H}^1_{/f}(\mathbb{Q}_\ell, T/\pi^mT)$
where $\mathrm{H}^1_{/f}(-) = \mathrm{H}^1(-) / \mathrm{H}^1_f(-)$. 
We make a choice of generator of $\mathbb{F}^\times_\ell$ in the isomorphism, and it is compatible with the choice made in $\S$\ref{subsubsec:kurihara-numbers} below.
If $\ell \in \mathcal{N}_m$, we identify $\mathrm{H}^1_{/f}(\mathbb{Q}_\ell, T/\pi^mT) = \mathrm{H}^1_{\mathrm{tr}}(\mathbb{Q}_\ell, T/\pi^m T)$.
See \cite[$\S$1.2]{mazur-rubin-book} for details.

\subsubsection{}
The localization map at $\ell$ and its singular quotient are denoted by
\[
\xymatrix{
\mathrm{loc}_\ell : \mathrm{H}^1(\mathbb{Q}, T/\pi^mT) \to \mathrm{H}^1(\mathbb{Q}_\ell, T/\pi^mT) , 
& \mathrm{loc}^s_\ell : \mathrm{H}^1(\mathbb{Q}, T/\pi^mT) \to \mathrm{H}^1_{/f}(\mathbb{Q}_\ell, T/\pi^mT) .
}
\]

\subsection{Selmer structures and Selmer groups} \label{subsec:selmer-structure}
\subsubsection{}
Let $\Sigma$ be a finite set of places of $\mathbb{Q}$ containing $p$, $\infty$ and all ramified primes of $T$.
Denote by $\mathbb{Q}_{\Sigma}$ the maximal extension of $\mathbb{Q}$ unramified outside $\Sigma$.

Let $\mathcal{F}$ be a Selmer structure on $T /\pi^m T$, i.e a collection of the following data
\begin{itemize}
\item a finite set $\Sigma(\mathcal{F})$ of places of $\mathbb{Q}$, including $\infty$, $p$ and all primes where $T /\pi^m T$ is ramified,
\item for each $\ell \in \Sigma(\mathcal{F})$, a local condition on $T/\pi^m T$ is given, i.e. a choice of 
$\mathrm{H}^1_{\mathcal{F}}(\mathbb{Q}_\ell, T /\pi^m T) \subseteq \mathrm{H}^1(\mathbb{Q}_\ell, T /\pi^m T) $.
\end{itemize}
For a prime $q \not \in \Sigma$, we fix $\mathrm{H}^1_{\mathcal{F}}(\mathbb{Q}_q, T/\pi^m T)  = \mathrm{H}^1_{f}(\mathbb{Q}_q, T/\pi^m T)$.

The Selmer group of $T /\pi^m T$ (with respect to Selmer structure $\mathcal{F}$) is defined by
$$\mathrm{Sel}_{\mathcal{F}}(\mathbb{Q}, T /\pi^m T)  = \mathrm{ker} \left( \mathrm{H}^1(\mathbb{Q}_{\Sigma(\mathcal{F})}/\mathbb{Q}, T/\pi^m T) \to \bigoplus_{\ell \in \Sigma(\mathcal{F})} \dfrac{\mathrm{H}^1(\mathbb{Q}_\ell, T/ \pi^m T)}{\mathrm{H}^1_{\mathcal{F}}(\mathbb{Q}_\ell, T/ \pi^m T) } \right) .$$
Write $W^*(1) = \mathrm{Hom}(T, \mu_{p^\infty})$.
For a Selmer structure $\mathcal{F}$ on $T/\pi^m T$, the dual Selmer structure $\mathcal{F}^*$ on $W^*(1)[\pi^m]$ is defined by taking $\mathrm{H}^1_{\mathcal{F}^*}(\mathbb{Q}_\ell, W^*(1)[\pi^m]) = \mathrm{H}^1_{\mathcal{F}}(\mathbb{Q}_\ell, T/\pi^m T)^\perp$
with respect to the local Tate pairing for all $\ell \in \Sigma(\mathcal{F})$. 
The corresponding Selmer group $\mathrm{Sel}_{\mathcal{F}^*}(\mathbb{Q}, W^*(1)[\pi^m])$ is defined similarly.

\subsubsection{}
The \textbf{Bloch--Kato Selmer structure $\mathcal{F}_{\mathrm{BK}}$} is defined by
$\mathrm{H}^1_{\mathcal{F}_{\mathrm{BK}}}(\mathbb{Q}_\ell, T/\pi^m T) = \mathrm{H}^1_{f}(\mathbb{Q}_\ell, T/\pi^m T)$ for every prime $\ell$.
We write
\[
\xymatrix{
\mathrm{Sel}(\mathbb{Q}, T/\pi^mT) = \mathrm{Sel}_{\mathcal{F}_{\mathrm{BK}}}(\mathbb{Q}, T/\pi^m T), & \mathrm{Sel}(\mathbb{Q}, W^*(1)[\pi^m]) = \mathrm{Sel}_{\mathcal{F}^*_{\mathrm{BK}}}(\mathbb{Q}, W^*(1)[\pi^m])
}
\]
since the dual Bloch--Kato Selmer structure is also the Bloch--Kato Selmer structure due to the local Tate duality \cite[Prop. 1.4.3]{rubin-book}.

The \textbf{canonical Selmer structure $\mathcal{F}_{\mathrm{can}}$}
is defined by 
$\mathrm{H}^1_{\mathcal{F}_{\mathrm{can}}}(\mathbb{Q}_\ell, T/\pi^m T) = \mathrm{H}^1_{f}(\mathbb{Q}_\ell, T/\pi^m T)$ for every prime $\ell \neq p$
and 
$\mathrm{H}^1_{\mathcal{F}_{\mathrm{can}}}(\mathbb{Q}_p, T/\pi^m T) = \mathrm{H}^1(\mathbb{Q}_p, T/\pi^m T)$.
We write
\[
\xymatrix{
\mathrm{Sel}_{\mathrm{rel}}(\mathbb{Q}, T/\pi^m T) = \mathrm{Sel}_{\mathcal{F}_{\mathrm{can}}}(\mathbb{Q}, T/\pi^m T), & \mathrm{Sel}_0(\mathbb{Q}, W^*(1)[\pi^m ]) = \mathrm{Sel}_{\mathcal{F}^*_{\mathrm{can}}}(\mathbb{Q}, W^*(1)[\pi^m]).
}
\]

The compact Selmer groups of $T$ are defined by taking the projective limit of Selmer groups of $T/\pi^m T$.
The discrete Selmer groups of $W^*(1)$ are also defined by taking the direct limit of Selmer groups of $W^*(1)[\pi^m]$.

\subsubsection{}
For a given Selmer structure $\mathcal{F}$ on $T/\pi^mT$ and $n \in \mathcal{N}_m$, 
the Selmer structure $\mathcal{F}(n)$ is defined by
\begin{itemize}
\item $\mathrm{H}^1_{\mathcal{F}(n)}(\mathbb{Q}_\ell, T/\pi^m T) = \mathrm{H}^1_{\mathcal{F}}(\mathbb{Q}_\ell, T/\pi^m T)$ for $\ell$ not dividing $n$, and
\item $\mathrm{H}^1_{\mathcal{F}(n)}(\mathbb{Q}_\ell, T/\pi^m T) = \mathrm{H}^1_{\mathrm{tr}}(\mathbb{Q}_\ell, T/\pi^m T)$ for $\ell$ dividing $n$.
\end{itemize}
The dual Selmer structure $\mathcal{F}(n)^*$ on $W^*(1)[\pi^m]$ becomes $\mathcal{F}^*(n)$ thanks to the local Tate duality again \cite[Prop. 1.3.2.(ii)]{mazur-rubin-book}.
For $n \in \mathcal{N}_m$, we write
%\begin{align*}
%\mathrm{Sel}_{\mathrm{rel}, n}(\mathbb{Q}, T/I_nT)  & = \mathrm{Sel}_{\mathcal{F}_{\mathrm{can}}(n)}(\mathbb{Q}, T/I_nT) , \\
%\mathrm{Sel}_{n}(\mathbb{Q}, T/I_nT)  & = \mathrm{Sel}_{\mathcal{F}_{\mathrm{BK}}(n)}(\mathbb{Q}, T/I_nT) , \\
%\mathrm{Sel}_{n}(\mathbb{Q}, W^*(1)[I_n])  & = \mathrm{Sel}_{\mathcal{F}^*_{\mathrm{BK}}(n)}(\mathbb{Q}, W^*(1)[I_n]), \\
%\mathrm{Sel}_{0,n}(\mathbb{Q}, W^*(1)[I_n]) & = \mathrm{Sel}_{\mathcal{F}^*_{\mathrm{can}}(n)}(\mathbb{Q}, W^*(1)[I_n]).
%\end{align*}
\[
\xymatrix@R=0em{
\mathrm{Sel}_{\mathrm{rel}, n}(\mathbb{Q}, T/I_nT)  = \mathrm{Sel}_{\mathcal{F}_{\mathrm{can}}(n)}(\mathbb{Q}, T/I_nT) , & \mathrm{Sel}_{n}(\mathbb{Q}, T/I_nT)  = \mathrm{Sel}_{\mathcal{F}_{\mathrm{BK}}(n)}(\mathbb{Q}, T/I_nT) , \\
\mathrm{Sel}_{n}(\mathbb{Q}, W^*(1)[I_n])  = \mathrm{Sel}_{\mathcal{F}^*_{\mathrm{BK}}(n)}(\mathbb{Q}, W^*(1)[I_n]), & \mathrm{Sel}_{0,n}(\mathbb{Q}, W^*(1)[I_n]) = \mathrm{Sel}_{\mathcal{F}^*_{\mathrm{can}}(n)}(\mathbb{Q}, W^*(1)[I_n]).
}
\]

%We define
%\begin{align*}
%\mathrm{Sel}_{\mathrm{rel}, n}(\mathbb{Q}, T/I_nT)  & := \mathrm{ker} \left( \mathrm{H}^1(\mathbb{Q}_{\Sigma_n}/\mathbb{Q}, T/I_nT) 
%\to
% \bigoplus_{\ell \in \Sigma \setminus \lbrace p\rbrace} \dfrac{ \mathrm{H}^1(\mathbb{Q}_{\ell}, T/I_nT) }{ \mathrm{H}^1_f(\mathbb{Q}_{\ell}, T/I_nT) }
%\oplus
% \bigoplus_{\ell \mid n} \dfrac{ \mathrm{H}^1(\mathbb{Q}_{\ell}, T/I_nT) }{ \mathrm{H}^1_{\mathrm{tr}}(\mathbb{Q}_{\ell}, T/I_nT) }
%\right) , \\
%\mathrm{Sel}_{n}(\mathbb{Q}, T/I_nT)  & := \mathrm{ker} \left( \mathrm{Sel}_{\mathrm{rel}, n}(\mathbb{Q}, T/I_nT) 
%\to
% \dfrac{ \mathrm{H}^1(\mathbb{Q}_{p}, T/I_nT) }{ \mathrm{H}^1_f(\mathbb{Q}_{p}, T/I_nT) } \right) , \\
%\mathrm{Sel}_{n}(\mathbb{Q}, W^*(1)[I_n])  & := \mathrm{ker} \left( \mathrm{H}^1(\mathbb{Q}_{\Sigma_n}/\mathbb{Q}, W^*(1)[I_n]) 
%\to
% \bigoplus_{\ell \in \Sigma } \dfrac{ \mathrm{H}^1(\mathbb{Q}_{\ell}, W^*(1)[I_n]) }{ \mathrm{H}^1_f(\mathbb{Q}_{\ell}, W^*(1)[I_n]) }
%\oplus
% \bigoplus_{\ell \mid n} \dfrac{ \mathrm{H}^1(\mathbb{Q}_{\ell}, W^*(1)[I_n]) }{ \mathrm{H}^1_{\mathrm{tr}}(\mathbb{Q}_{\ell}, W^*(1)[I_n]) }
%\right) , \\
%\mathrm{Sel}_{0,n}(\mathbb{Q}, W^*(1)[I_n]) & := \mathrm{ker} \left( \mathrm{Sel}_n(\mathbb{Q}, W^*(1)[I_n]) \to \mathrm{H}^1(\mathbb{Q}_p, W^*(1)[I_n]) \right), 
%\end{align*}
%respectively. 
%
%

\subsection{Modular symbols and Kurihara numbers} \label{subsec:modular-symbols-kurihara-numbers}
We review modular symbols closely following \cite{bellaiche-book} and introduce Kurihara numbers for higher weight modular forms.
\subsubsection{Period integrals}
For an integer $r$ with $1 \leq r \leq k-1$, a positive integer $a > 0$, and an integer $n$, the period integral is defined by
$$\lambda_{\overline{f}}( z^{r-1} ; a,n ) = (\sqrt{-1})^{r}  \cdot \int^{\infty}_{0} \overline{f}(iy+ a/n) \cdot y^{r-1} dy .
$$
%\begin{align*}
%\lambda_{\overline{f}}( z^{r-1} ; a,n ) & = \int^{i \infty}_{a/n} \overline{f}(z) \cdot (z - a/n)^{r-1} dz  \\
%& = (\sqrt{-1})^{r}  \cdot \int^{\infty}_{0} \overline{f}(iy+ a/n) \cdot y^{r-1} dy .
%\end{align*}
Following \cite[$\S$5.3.3]{bellaiche-book}, we define
$$\lambda^{\pm}_{\overline{f}}( z^{r-1}; a,n ) =
\dfrac{1}{2} \cdot \left( 
\lambda_{\overline{f}}( z^{r-1}; a,n )
\pm
\lambda_{\overline{f}}( z^{r-1}; -a,n ) \right) .$$
In \cite[(8.6)]{mtt}, the convention of period integrals is given by
$( - 2 \pi \sqrt{-1} ) \cdot n^{r-1} \cdot \lambda_{\overline{f}}( z^{r-1} ; a,n )$.
%\begin{align*}
%\lambda^{\mathrm{MTT}}_{\overline{f}}( z^{r-1} ; a,n ) & = - 2 \pi \sqrt{-1} \cdot (\sqrt{-1})^{r}  \cdot n^{r-1} \cdot \int^{\infty}_{0} \overline{f}(iy+ a/n) \cdot y^{r-1} dy \\
%& = 2 \pi \cdot (\sqrt{-1})^{r-1}  \cdot n^{r-1} \cdot \int^{\infty}_{0} \overline{f}(iy+ a/n) \cdot y^{r-1} dy .
%\end{align*}

%This is more optimal for the integrality of modular symbols, but it is less compatible with the construction of $p$-adic $L$-functions.

%We also recall the functional equation
%\begin{thm}
%Let $f \in S_k(\Gamma, \mathbb{C})$ be a newform.
%$$\dfrac{\Gamma(s)}{(2 \pi)^s} \cdot W(f) \cdot \overline{L(f, \overline{s})}  = \dfrac{(-1)^{k-2} \cdot N^s}{(\sqrt{-1})^k} \cdot \dfrac{\Gamma(k-s)}{(2\pi)^{k-s}} \cdot L(f, k-s)$$
%If moreover $f$ has trivial nebentypus, then $W(f) = \pm N^{(k-2)/2}$ and the becomes
%$$\dfrac{\Gamma(s)}{(2\pi)^s} \cdot L(f,s) = \pm 1 \cdot \dfrac{(-1)^{k-2} \cdot N^{s - (k-2)/2}}{(\sqrt{-1})^k} \cdot \dfrac{\Gamma(k-s)}{(2\pi)^{k-s}} \cdot L(f, k-s)$$
%At the central critical point $s = k/2$, the function $L(f, k-s)$ has a zero of even order if $\pm = +$, and a zero of odd order if $\pm  = -$.
%\end{thm}
%\begin{proof}
%See \cite[Theorem 5.4.3]{bellaiche-book}
%\end{proof}
\subsubsection{Minimal integral periods and modular symbols}
\begin{defn} \label{defn:minimal-periods}
The \textbf{minimal integral periods $\Omega^{\pm}_{\overline{f}, \mathrm{min}}$ of $\overline{f}$}
are defined uniquely up to $\mathcal{O}^\times$
such that
\begin{equation} \label{eqn:period-integrality}
\lambda^{\pm, \mathrm{min}}_{\overline{f}}(z^{r-1}; a, n) :=   \dfrac{\lambda^{\pm}_{\overline{f}}(z^{r-1}; a, n)}{\Omega^{\pm}_{\overline{f}, \mathrm{min}}} \in \mathcal{O}
\end{equation}
for an integer $r$ with $1 \leq r \leq k-1$, a positive integer $a > 0$, and an integer $n$, and
\begin{equation} \label{eqn:period-minimality}
\lambda^{\pm, \mathrm{min}}_{\overline{f}}(z^{r-1}; a, n) \in \mathcal{O}^\times
\end{equation}
for some integer $r$ with $1 \leq r \leq k-1$, some positive integer $a > 0$, and some integer $n$.
The value $\lambda^{\pm, \mathrm{min}}_{\overline{f}}(z^{r-1}; a, n)$ is called the \textbf{minimally normalized modular symbol}.
\end{defn}
\begin{rem} \label{rem:minimal-periods}
%The existence of the periods up to $F^\times$ follows from the strong multiplicity one (e.g. \cite[Lemma 5.4.8]{bellaiche-book}).
The existence of the minimal integral periods up to $\mathcal{O}^\times$ follows from \cite[Prop., p. 7]{mtt}.
See also \cite[Rem. 5.4.10]{bellaiche-book}.
If we require (\ref{eqn:period-integrality}) only, we just call them (non-minimal) integral periods.
We can also choose the minimal integral periods in the same way but with fixed $r$.
Any choice does not affect the main result of this paper.
\end{rem}
%\begin{thm}[Manin, Shimura]
%For any Dirichlet character $\chi$ of conductor $n$ and  any integer $r$ with $1 \leq r \leq k-1$, we have
%$$\dfrac{  (r-1)! \cdot \tau(\chi) }{ (-2 \pi \sqrt{-1})^r} \cdot \dfrac{L(\overline{f}, \chi^{-1}, r)}{\Omega^{\pm}_{\overline{f}, \mathrm{min}}} \in \mathcal{O}[\chi] $$
%where $\tau(\chi)$ is the Gauss sum for $\chi$ and the sign of the period coincides with that of $(-1)^{r-1} \cdot \chi(-1)$.
%\end{thm}
%\begin{proof}
%By Birch's lemma \cite[Lem. 5.4.5]{bellaiche-book}, we have equality
%\begin{equation} \label{eqn:birchs-lemma}
%\sum_{a \in (\mathbb{Z}/n\mathbb{Z})^\times} \left[ \dfrac{a}{n} \right]^{\pm}_{\overline{f}, r} \cdot \chi(a)  = 
%\dfrac{  (r-1)! \cdot \tau(\chi) }{ (-2 \pi \sqrt{-1})^r} \cdot \dfrac{L(\overline{f}, \chi^{-1}, r)}{\Omega^{\pm}_{\overline{f}, \mathrm{min}}} 
%\end{equation}
%for any Dirichlet character  $\chi : (\mathbb{Z}/n\mathbb{Z})^\times \to \mathbb{C}^\times$.
%See \cite{manin-periods, shimura-periods, shokurov-shimura-integrals, bellaiche-book} for more details.
%\end{proof}

\subsubsection{Kurihara numbers}  \label{subsubsec:kurihara-numbers}
For each prime $\ell \in \mathcal{N}_m$, we fix a primitive root $\eta_\ell$ mod $\ell$ and define
 $\mathrm{log}_{\eta_\ell}(a) \in \mathbb{Z}/(\ell-1)\mathbb{Z}$ by $\eta^{ \mathrm{log}_{\eta_\ell}(a)}_\ell \equiv a \pmod{\ell}$.
\begin{defn} \label{defn:kurihara-numbers}
Let $n \in \mathcal{N}_1$.
The \textbf{Kurihara number for $\overline{f}$ at $s = r$} is defined by
$$\widedelta^{\mathrm{min}, r}_n := \sum_{a \in (\mathbb{Z}/n\mathbb{Z})^\times}  \overline{ \lambda^{\pm, \mathrm{min}}_{\overline{f}}(z^{r-1}; a, n) } \cdot \left( \prod_{\ell \vert n} \overline{ \mathrm{log}_{\eta_\ell} (a) } \right)  \in \mathcal{O}/I_n \mathcal{O}$$
where $\overline{(-)}$ is the mod $I_n$ reduction of $(-)$ and the sign of the modular symbol coincides with that of $(-1)^{r-1}$.
The \textbf{collection of Kurihara numbers  for $\overline{f}$ at $s = r$} is defined by
$$\kn^{\mathrm{min}, r} = \left \lbrace 
\widedelta^{\mathrm{min}, r}_n \in \mathcal{O}/I_n \mathcal{O} : n \in \mathcal{N}_1
\right \rbrace .$$
When $\psi = \mathbf{1}$ and $r =k/2$, we replace $k/2$ by $\dagger$ in the above notation as in Theorem \ref{thm:main-central-critical}.
\end{defn}

%To sum up, we have the following diagram
%\[
%\xymatrix{
%V_F(f) \otimes F_\pi \ar[d]_-{\simeq} & V_F(f) \ar[l] \ar[r] & V_F(f) \otimes \mathbb{C} \\
%V_{F_\pi}(f) & & S(f) \ar[u]^-{\mathrm{per}_f} .
%}
%\]

\subsubsection{The functional equation for Kurihara numbers}
We assume that $\psi = \mathbf{1}$ and consider the $r = k/2$ case here.
By using the functional equation for modular symbols (e.g. \cite[$\S$6]{mtt}, \cite[Prop. 2.5]{ota-rank-part}), we observe that
\begin{equation} \label{eqn:functional-equation-delta_n}
w(f) \cdot (-1)^{\nu(n)} \cdot \widetilde{\delta}^{\mathrm{min},\dagger}_n = \widetilde{\delta}^{\mathrm{min},\dagger}_n \in \mathcal{O}/I_n \mathcal{O} .
\end{equation}
where $w(f)$ is the sign of the functional equation associated to $L(f,s)$. See also \cite[p. 220]{kurihara-munster} and \cite[Lem. 4, p. 347]{kurihara-iwasawa-2012}. The following statement is straightforward from (\ref{eqn:functional-equation-delta_n}).
\begin{prop} \label{prop:vanishing-delta-n}
If $(-1)^{\nu(n)} \neq w(f)$, then $\widetilde{\delta}^{\mathrm{min},\dagger}_n  = 0$.
In particular, if $\widedelta^{\mathrm{min},\dagger}_n \neq 0$, then $\widedelta^{\mathrm{min},\dagger}_{n\ell} = 0$ for every $\ell \in \mathcal{N}_1$ with $(n, \ell) =1$.
\end{prop}
This vanishing is independent of the normalization of modular symbols.

\section{Kato's zeta elements and Kolyvagin systems (including the $p=3$ case)} \label{sec:kato-zeta-elts-kolyvagin}
\subsection{Canonical Kato's Euler systems} \label{subsec:canonical-Kato-Euler-systems}
We quickly review the notion of the canonical Kato's Euler system following \cite{kim-kato}.
\subsubsection{}
Let $L$ be a finite abelian extension of $\mathbb{Q}$.
Write
$$\mathrm{H}^1_{\mathcal{I}w}(L, V_f(k-r)) = \varprojlim_s \mathrm{H}^1(L(\zeta_{p^s}), V_f(k-r))$$
where the inverse limit is taken with respect to the corestriction maps.
The notation `$\mathrm{H}^1_{\mathrm{Iw}}$' is reserved for the cyclotomic $\mathbb{Z}_p$-extension.

By using the comparison between \'{e}tale and Betti cohomologies, the complex conjugation $c$ acts on $V_f$.
Let $V^\pm_f = (V_f)^{c= \pm 1} \subseteq V_f$ be the $c$-eigenspace with eigenvalue $\pm 1$, respectively.
For $\gamma \in V_f$, we write $\gamma = \gamma^+ + \gamma^-$ with $\gamma^{\pm} \in V^{\pm}_f$, respectively.
\begin{thm}[Kato] \label{thm:zeta-morphism}
Suppose that $\rho_f$ has large image.
For each finite abelian extension $L$ of $\mathbb{Q}$, there exists a canonical $\mathcal{O}$-linear morphism (integral zeta morphism)
$$T_f \to \mathrm{H}^1_{\mathcal{I}w}(L, T_f)$$
defined by $\gamma \mapsto \mathbf{z}^{(p)}_{L,\gamma}$
satisfying the following properties.

For every integers $n \geq 0$ and $r$ with $1 \leq r \leq k-1$, denote by 
$z^{\mathrm{Kato}, k-r}_{L(\zeta_{p^n}), \gamma}$ the image of
$\mathbf{z}^{(p)}_{L, \gamma} \otimes (\zeta_{p^s})^{\otimes k-r}_s \in \mathrm{H}^1_{\mathcal{I}w}(L, T_f(k-r))$
in $\mathrm{H}^1(L(\zeta_{p^n}), T_f(k-r))$.
\begin{enumerate}
\item $\left\lbrace z^{\mathrm{Kato}, k-r}_{L(\zeta_{p^n}), \gamma} \right\rbrace_{L(\zeta_{p^n})}$ satisfies the axioms of the Euler system for $T_f(k-r)$ in the sense of Rubin \cite{rubin-book}.
\item  Write
$ \mathrm{exp}^*\circ \mathrm{loc}_p\left( z^{\mathrm{Kato}, k-r}_{L(\zeta_{p^n}), \gamma} \right) = c_{k-r,L,p^n, \gamma} \cdot \omega_{f}$
where $c_{k-r,L,p^n, \gamma} \in L(\zeta_{p^n}) \otimes F$ and $\omega_f \in \mathrm{Fil}^0\mathbf{D}_{\mathrm{dR}}(V_f(k-r))$ is a non-zero element.
For any character $\chi : \mathrm{Gal}(L(\zeta_{p^n})/\mathbb{Q}) \to \mathbb{C}^\times$,
we have
$$ \sum_{\sigma \in \mathrm{Gal}(L(\zeta_{p^n})/\mathbb{Q})} \sigma \left(  c_{k-r,L,p^n, \gamma}  \right) \cdot \chi(\sigma) \cdot \mathrm{per}_f \left( \omega_{f} \right) =  (2\pi \sqrt{-1})^{k-r-1} \cdot L^{(p)}(\overline{f}, \chi, r) \cdot \gamma^{\pm}$$
where $\mathrm{per}_f$ is Kato's period map in \cite[$\S$6.3]{kato-euler-systems},
$L^{(p)}(\overline{f}, \chi, r)$ is the $\chi$-twisted $L$-value of $\overline{f}$ at $s=r$  with the Euler factor at $p$ removed,
 and the sign of $\gamma^{\pm}$ coincides with that of $(-1)^{k-r-1} \cdot \chi(-1)$.
\end{enumerate}
\end{thm}
\begin{proof}
When $L = \mathbb{Q}$, this is \cite[Thm. 12.5 and Thm. 16.6]{kato-euler-systems}.
We may assume that $L$ is unramified at $p$ without loss of generality. 
If the coefficient module is $V_f(k-r)$ instead of $T_f(k-r)$, the explicit recipe of $z^{\mathrm{Kato}, k-r}_{L(\zeta_{p^n}), \gamma}$ can be found in \cite[Appendix A]{delbourgo-book} and \cite{kim-kato}, and the Euler system relation along the tame direction follows easily  from \cite[Prop. 8.12]{kato-euler-systems}.
Therefore, it suffices to check the integrality of the zeta morphism.
The large image assumption implies $\mathrm{H}^0( L(\zeta_{p^\infty}), \overline{\rho}_f ) = 0$.
Thus, $\mathrm{H}^1_{\mathcal{I}w}(L, T_f) \simeq \mathrm{H}^1_{\mathcal{I}w}(\mathbb{Q}, T_f \otimes_{\mathcal{O}} \mathcal{O}[\mathrm{Gal}(L/\mathbb{Q})])$ is free over $\Lambda_{\mathcal{I}w}$ (e.g. \cite[Rem. 6.5]{sakamoto-stark-systems}.)
The integrality of the zeta morphism over $L$ now follows from \cite[$\S$13.14]{kato-euler-systems}.
See also \cite[Thm. 6.1]{kataoka-thesis} and \cite{kim-kato}.
\end{proof}
\begin{defn}
Suppose that $\rho_f$ has large image.
For each $\gamma \in T_f$, 
the family of cohomology classes
$$\mathbf{z}^{\mathrm{Kato}, k-r}_{\gamma} = \left\lbrace z^{\mathrm{Kato}, k-r}_{L(\zeta_{p^n}), \gamma} \in \mathrm{H}^1(L(\zeta_{p^n}), T_f(k-r)) : L/\mathbb{Q}, \textrm{ finite abelian}, n \geq 0 \right\rbrace$$
forms an Euler system for $T_f(k-r)$ in the sense of Rubin \cite{rubin-book} as in Theorem \ref{thm:zeta-morphism}.
When $\gamma = \gamma^+ + \gamma^-$ with $T^\pm_f = \mathcal{O}_F \cdot \gamma^\pm$, we omit $\gamma$ and 
$\mathbf{z}^{\mathrm{Kato}, k-r}$ is called the \textbf{canonical Kato's Euler system for $T_f(k-r)$}.
\end{defn}

\subsubsection{}
By applying the Euler-to-Kolyvagin system map \cite[Thm. 3.2.4]{mazur-rubin-book} to $\mathbf{z}^{\mathrm{Kato}, k-r}$, we obtain Kato's Kolyvagin system
$$\ks^{\mathrm{Kato}, k-r}  = \left\lbrace \kappa^{\mathrm{Kato}, k-r}_n \in \mathrm{Sel}_{\mathrm{rel},n}(\mathbb{Q}, T_f(k-r)/I_n) \right\rbrace_{ n \in \mathcal{N}_1 },$$
which is a family of cohomology classes satisfying the axiom (\ref{eqn:axiom-kolyvagin-systems}) given below.

\subsubsection{}
For each $n \in \mathcal{N}_1$, the Kolyvagin derivative operator at $n$ is defined by
$D_{\mathbb{Q}(\zeta_n)}  = \prod_{\ell \vert n} \sum_{i=1}^{\ell-2} i \cdot \sigma^i_{\eta_\ell}$
where $\eta_\ell$ is a fixed primitive root mod $\ell \in \mathcal{P}_1$ as in $\S$\ref{subsubsec:kurihara-numbers}.
Then $\kappa^{\mathrm{Kato},k-r}_n$ can also be defined by the image of $D_{\mathbb{Q}(\zeta_n)} z^{\mathrm{Kato},k-r}_{\mathbb{Q}(\zeta_n)} \in \mathrm{H}^1(\mathbb{Q}(\zeta_n), T_f(k-r))$ in $\mathrm{H}^1(\mathbb{Q}, T_f(k-r)/I_n)$.

\subsection{Kolyvagin systems for modular forms} \label{subsec:kolyvagin-systems}
When we work with Kolyvagin systems, we always assume that $\rho_f$ has large image.

\subsubsection{} \label{subsubsec:kolyvagin-systems}
A \textbf{Kolyvagin system $\ks^{k-r}$ for $(T_f(k-r), \mathcal{F}_{\mathrm{can}}, \mathcal{P}_1)$}
is a collection of cohomology classes
$\ks^{k-r} = \left\lbrace  \kappa^{k-r}_n \in \mathrm{Sel}_{\mathcal{F}(n)}(\mathbb{Q}, T_f(k-r)/I_n) : n \in \mathcal{N}_1  \right\rbrace$
such that
\begin{equation} \label{eqn:axiom-kolyvagin-systems}
\mathrm{loc}^s_\ell ( \kappa^{k-r}_{n\ell} ) = \phi^{\mathrm{fs}}_\ell \circ \mathrm{loc}_\ell (\kappa^{k-r}_n) \in \mathrm{H}^1_{/f}(\mathbb{Q}_\ell, T_f(k-r)/I_{n\ell})
\end{equation}
 for $\ell \in \mathcal{N}_1$ with $(n,\ell) =1$.
Let $\KS(T_f(k-r)) = \KS(T_f(k-r), \mathcal{F}_{\mathrm{can}}, \mathcal{P}_1)$ denote the module of Kolyvagin systems and $\KSbar(T_f(k-r)) = \KSbar(T_f(k-r), \mathcal{F}_{\mathrm{can}}, \mathcal{P}_1)$ denote the generalized module of Kolyvagin systems defined by the completion \cite[$\S$3.1]{mazur-rubin-book}.
We do not distinguish $\KS$ and $\KSbar$ due to \cite[Cor. 4.5.3]{mazur-rubin-book} with the core rank one property (Theorem \ref{thm:core-rank-modular-forms-char-zero}) and \cite[Rem. 5.3.11]{mazur-rubin-book}.

Our convention of Kolyvagin systems depends on the choice of generators of $\mathrm{Gal}(\mathbb{Q}(\zeta_\ell)/\mathbb{Q})$ for each prime $\ell$ dividing $n$, and it corresponds to the choice of the primitive roots in $\S$\ref{subsubsec:kurihara-numbers}.

\subsubsection{}
When $p > 3$, the large image assumption is strong enough to satisfy all the working hypotheses for Kolyvagin systems \cite[$\S$3.5 and Lem. 6.2.3]{mazur-rubin-book}. 
When $p = 3$ and $\mathrm{Hom}(\overline{\rho}^\dagger_f, \mathrm{Hom}(\overline{\rho}^\dagger_f, \mathbb{F}(1) ) ) \neq 0$, we need a more argument following 
\cite{sakamoto-p-3}.
In \cite{mazur-rubin-book}, the condition $p > 3$ is used only in the refined Chebotarev density result below and its consequences.
\begin{prop}[Mazur--Rubin] \label{prop:chebotarev-mazur-rubin}
Let $p > 3$ be a prime and assume that $\rho_f$ has large image.
Let $c_1, c_2 \in \mathrm{H}^1(\mathbb{Q}, T_f(k-r)/\pi^m)$ and $c_3, c_4 \in \mathrm{H}^1(\mathbb{Q}, W_{\overline{f}}(r)[\pi^m])$.
For every $m' \in \mathbb{Z}_{>0}$, there exists a set $S \subseteq \mathcal{P}_{m'}$ of positive density such that for every $\ell \in S$, 
the localizations $\mathrm{loc}_{\ell} (c_i)$ are all non-zero.
\end{prop}
\begin{proof}
See \cite[Prop. 3.6.1]{mazur-rubin-book}.
\end{proof}
When $p=3$ and the Galois representation is residually self-dual, Sakamoto obtained a slightly weaker result.
\begin{prop}[Sakamoto] \label{prop:chebotarev-sakamoto}
Let $p = 3$ and assume that $\psi = \mathbf{1}$ and $\rho_f$ has large image.
Let $c_1, c_2, c_3, c_4  \in \mathrm{H}^1(\mathbb{Q}, \overline{\rho}^\dagger_f)$ be non-zero elements such that
$$\mathrm{dim}_{\mathbb{F}} (\mathbb{F}c_1 + \mathbb{F}c_2 + \mathbb{F}c_3 + \mathbb{F}c_4 ) \geq 3.$$
For every $m' \in \mathbb{Z}_{>0}$, there are infinitely many primes $\ell \in \mathcal{P}_{m'}$ such that 
the localizations $\mathrm{loc}_{\ell} (c_i)$ are all non-zero in $\mathrm{H}^1(\mathbb{Q}_\ell, \overline{\rho}^\dagger_f)$.
\end{prop}
\begin{proof}
See \cite[Lem. 5.2]{sakamoto-p-3}.
\end{proof}

\subsubsection{}
A prime $\ell \in \mathcal{P}_1$ is said to be \textbf{useful for (non-zero) $\kappa^{k-r}_n$} with $n \in \mathcal{N}_1$ if
 $(\ell ,n)=1$ and $\mathrm{loc}_\ell(\kappa^{k-r}_n) \neq 0$.
 By (\ref{eqn:axiom-kolyvagin-systems}), if  $\ell$ is a useful prime for a non-zero $\kappa^{k-r}_n$, then $\kappa^{k-r}_{n\ell} \neq 0$. 
\begin{prop} \label{prop:chebotarev}
There are infinitely many useful primes for a non-zero $\kappa^{k-r}_n$.
\end{prop}
\begin{proof}
It is immediate from Propositions \ref{prop:chebotarev-mazur-rubin} and \ref{prop:chebotarev-sakamoto}.
\end{proof}
\subsection{Kolyvagin systems over principal Artinian rings}
We fix a positive integer $m$ in this subsection.
\subsubsection{}
The following lemma plays an important role in our proof.
\begin{lem} \label{lem:surjectivity-at-ell}
Let $n \in \mathcal{N}_m$ and $\ell \in \mathcal{P}_m$ with $(n, \ell) = 1$, and  `$\ell\textrm{-str}$' denotes the strict local condition at $\ell$.
Let $\mathcal{F} = \mathcal{F}_{\mathrm{can}}$ or $\mathcal{F}_{\mathrm{BK}}$ in $\S$\ref{subsec:selmer-structure}.
If the restriction map 
$$\mathrm{Sel}_{\mathcal{F}(n)}(\mathbb{Q}, T_f(k-r)/\pi^m) \to \mathrm{H}^1_f(\mathbb{Q}_\ell, T_f(k-r)/\pi^m)$$
 is surjective, then
$\mathrm{Sel}_{\mathcal{F}^*(n\ell)}(\mathbb{Q}, W_{\overline{f}}(r)[\pi^m]) = \mathrm{Sel}_{\mathcal{F}^*(n),\ell\textrm{-}\mathrm{str}}(\mathbb{Q}, W_{\overline{f}}(r)[\pi^m])$.
\end{lem}
\begin{proof}
See \cite[Lem. 4.1.7]{mazur-rubin-book}.
\end{proof}
\subsubsection{}
Following \cite[$\S$4.1]{mazur-rubin-book}, 
for $n \in \mathcal{N}_m$,
we write
\begin{align*}
 \lambda(n, W_{\overline{f}}(r)[\pi^m]) & = \mathrm{length}_{\mathcal{O}} \mathrm{Sel}_{0, n}(\mathbb{Q}, W_{\overline{f}}(r)[\pi^m]), \\
 \lambda(n, T_{f}(k-r)/\pi^m) & = \mathrm{length}_{\mathcal{O}} \mathrm{Sel}_{\mathrm{rel}, n}(\mathbb{Q}, T_{f}(k-r)/\pi^m).
\end{align*}
We say $n \in \mathcal{N}_m$ is a \textbf{core vertex} if $\lambda(n, W_{\overline{f}}(r)[\pi^m])$ or $\lambda(n, T_{f}(k-r)/\pi^m)$ is zero \cite[Def. 4.1.8]{mazur-rubin-book}.
It is known that 
$\mathrm{Sel}_{\mathrm{rel}, n}(\mathbb{Q}, T_{f}(k-r)/\pi^m)$ and $\mathrm{Sel}_{0, n}(\mathbb{Q}, W_{\overline{f}}(r)[\pi^m])$ are free $\mathcal{O}/\pi^m\mathcal{O}$-modules if $n$ is a core vertex, and the ranks of these modules are independent of $n$ and one of them is zero \cite[Thm. 4.1.10]{mazur-rubin-book}.
The \textbf{core rank} of $T_f(k-r)/\pi^m$ with respect to the canonical Selmer structure is defined by $\chi(T_f(k-r)/\pi^m) = \mathrm{rk}_{ \mathcal{O}/\pi^m\mathcal{O} }\mathrm{Sel}_{\mathrm{rel}, n}(\mathbb{Q}, T_{f}(k-r)/\pi^m)$ for any core vertex $n$ \cite[Def. 4.1.11]{mazur-rubin-book}.

Since the core rank with respect to the Bloch--Kato Selmer structure is zero and the core rank with respect to the canonical Selmer structure is one \cite[Prop. 6.2.2]{mazur-rubin-book}, we have that $\KS(T_f(k-r)/\pi^m, \mathcal{F}_{\mathrm{BK}}, \mathcal{P}_m) = 0$ and $\KS(T_f(k-r)/\pi^m, \mathcal{F}_{\mathrm{can}}, \mathcal{P}_m)$ is free of rank one over $\mathcal{O}/\pi^m\mathcal{O}$ \cite[Thm. 4.2.2 and Cor. 4.5.2]{mazur-rubin-book}.

\begin{thm} \label{thm:splitting-mazur-rubin}
For every $m \geq 1$ and $n \in \mathcal{N}_m$, there exists a non-canonical isomorphism
\begin{equation} \label{eqn:sel-sel0-decomposition}
\mathrm{Sel}_{\mathrm{rel},n}(\mathbb{Q}, T_f(k-r)/\pi^m ) \simeq \mathcal{O}/I_n\mathcal{O} \oplus \mathrm{Sel}_{0,n}(\mathbb{Q}, W_{\overline{f}}(r)[\pi^m] ) .
\end{equation}
\end{thm}
\begin{proof}
See \cite[Thm. 4.1.13.(i)]{mazur-rubin-book} and \cite[Thm. 5.2.5]{mazur-rubin-book} with help of the core rank one property.
\end{proof}

\subsubsection{}
Let $\mathcal{X}$ be the graph whose set of vertices is $\mathcal{N}_m$ and whenever $n, n\ell \in \mathcal{N}_m$ with $\ell$ prime, we join $n$ and $n\ell$ by an edge \cite[Def. 3.1.2]{mazur-rubin-book}, and we write
\[
\xymatrix{
 \mathcal{H}(n)  = \mathrm{Sel}_{\mathrm{rel}, n}(\mathbb{Q},  T_f(k-r)/\pi^m), & \mathcal{H}'(n)  =  \pi^{ \lambda(n, W_{\overline{f}}(r)[\pi^m])}\mathrm{Sel}_{\mathrm{rel}, n}(\mathbb{Q},  T_f(k-r)/\pi^m),
}
\]
and $\mathcal{H}'(n)$ is said to be the \textbf{stub Selmer submodule at $n$} \cite[Def. 4.3.1]{mazur-rubin-book}.
See \cite[Def. 3.1.2 and 4.3.1]{mazur-rubin-book} for the notion of sheaves of Selmer modules $\mathcal{H}$ and stub Selmer modules $\mathcal{H}'$ on $\mathcal{X}$. Their stalks at $n \in \mathcal{X}$ are  $\mathcal{H}(n)$ and $\mathcal{H}'(n)$, respectively.

Define a subgraph $\mathcal{X}^0$ of $\mathcal{X}$ as follows.
The vertices of $\mathcal{X}^0$ are the core vertices of $\mathcal{X}$, i.e. the $n \in \mathcal{N}_m$ with $\lambda(n, W_{\overline{f}}(r)[\pi]) = 0$. We join $n$ and $n\ell$ by an edge in $\mathcal{X}^0$ if and only if the restriction map $\mathrm{Sel}_{\mathrm{rel}, n} (\mathbb{Q}, T_f(k-r)/\pi) \to \mathrm{H}^1_{f} (\mathbb{Q}_\ell, T_f(k-r)/\pi)$ is non-zero.
If we restrict sheaves $\mathcal{H}$ and $\mathcal{H}'$ to $\mathcal{X}^0$, they coincide by definition \cite[Def. 4.3.6]{mazur-rubin-book}.

The following theorem is fundamental in the theory of Kolyvagin systems.
\begin{thm} \label{thm:connected-graph}
The graph $\mathcal{X}^0$ is connected.
\end{thm}
\begin{proof}
See \cite[Thm. 4.3.12]{mazur-rubin-book} and \cite[Thm. 6.7]{sakamoto-p-3}.
\end{proof}
Theorem \ref{thm:connected-graph} is the only statement requiring the full strength of Propositions \ref{prop:chebotarev-mazur-rubin} and \ref{prop:chebotarev-sakamoto}. For details, see \cite[Prop. 4.3.11]{mazur-rubin-book} and \cite[Lem. 6.4 and Cor. 6.6]{sakamoto-p-3}.

\subsubsection{} 
Let 
$\ks^{k-r,(m)} = \left\lbrace \kappa^{k-r,(m)}_n \in \mathrm{Sel}_{\mathrm{rel},n}(\mathbb{Q}, T_f(k-r)/\pi^m) : n \in \mathcal{N}_m \right\rbrace$
 be a mod $\pi^m$ Kolyvagin system for $T_f(k-r)/\pi^m $.
\begin{thm} \label{thm:kolyvagin-system-location}
\begin{enumerate}
\item For every $n \in \mathcal{N}_m$, $\kappa^{k-r, (m)}_n \in \mathcal{H}'(n)$.
\item If $\ks^{k-r, (m)}$ is non-trivial, then there exists an integer $j \geq 0$ such that
$\left\langle \kappa^{k-r, (m)}_n \right\rangle  = \pi^j\mathcal{H}'(n)  = \pi^{j+\lambda(n, W_{\overline{f}}(r)[\pi^m])}\mathcal{H}(n) $
for every $n \in \mathcal{N}_m$.
\end{enumerate}
\end{thm}
\begin{proof}
\begin{enumerate}
\item See \cite[Thm. 4.4.1]{mazur-rubin-book} since  the core rank is one.
The argument depends heavily on Theorem \ref{thm:connected-graph}.
\item See \cite[Cor. 4.5.2.(ii)]{mazur-rubin-book}.
\end{enumerate}
\end{proof}
We say that a Kolyvagin system \textbf{$\ks^{k-r,(m)}$ is primitive} if $\ks^{k-r,(1)}  = \left( \kappa^{k-r, (m)}_n \pmod{\pi} \right)_{n \in \mathcal{N}_1}$ is non-zero. This is equivalent to $j=0$ in Theorem \ref{thm:kolyvagin-system-location}. 
See \cite[Cor. 4.5.4 and Def. 4.5.5]{mazur-rubin-book}.
For $m \gg 0$, this integer $j$ becomes $\partial^{(\infty)} \left( \ks^{k-r} \right)$ in $\S$\ref{subsubsec:numerical-invariants-kappa} later.

\subsubsection{} \label{subsubsec:structure-fine-selmer-p^k}
Suppose that $\kappa^{k-r,(m)}_n \neq 0$ for some $n \in \mathcal{N}_m$.
Following \cite[Ex. 3.1.12]{mazur-rubin-book}, we obtain another Kolyvagin system $\ks^{k-r,n, (m)}$ for $T_f(k-r)/\pi^m$
defined by 
$\kappa^{k-r,n, (m)}_{n'} = \kappa^{k-r, (m)}_{n \cdot n'}$
where $n' \in \mathcal{N}_m$ with $(n',n)=1$.
%Then Theorem \ref{thm:structure-fine-selmer-p^k} also applies in this case since $\kappa^{n,(k)}_1 =\kappa^{(k)}_n \neq 0$.
%Thus, we have
%\[
%\xymatrix{
%\partial^{(0)}(\ks^{n, (k)})  \geq \partial^{(1)}(\ks^{n, (k)}) \geq \cdots , & e_0(\ks^{n, (k)})  \geq e_1(\ks^{n, (k)}) \geq \cdots \geq 0,
%}
%\]
%and
%$$\mathrm{Sel}_{0,n}(\mathbb{Q}, E[p^k]) \simeq \bigoplus_{i \geq 0} \dfrac{ \mathbb{Z} /p^k\mathbb{Z} }{ p^{e_i(\ks^{n, (k)})} \mathbb{Z} /p^k\mathbb{Z} } $$
%where $\partial^{(r)}(\ks^{n, (k)})$ and $e_i(\ks^{n, (k)})$ are similarly defined.
\subsubsection{}
Let
$\partial^{(i)} ( \ks^{k-r,(m)} ) = \mathrm{min} \left\lbrace m - \mathrm{length}( \mathcal{O}/\pi^m \mathcal{O} \cdot \kappa^{k-r, (m)}_n ) : n \in \mathcal{N}_m, \nu(n) = i\right\rbrace $.
The following structure theorem provides the foundation of our proof of Theorems  \ref{thm:main-central-critical} and \ref{thm:main-all-critical}.
\begin{thm} \label{thm:structure-fine-selmer-p^k}
Suppose that $\kappa^{k-r,(m)}_n \neq 0$ for some $n \in \mathcal{N}_m$.
Write
$$\mathrm{Sel}_{0,n}(\mathbb{Q}, W_{\overline{f}}(r)[\pi^m]) \simeq \bigoplus_{i \geq 1}  \mathcal{O} / \pi^{d_i} \mathcal{O}$$
with non-negative integers
$d_1 \geq d_2 \geq \cdots $, and fix $j \geq 0$ such that
$\kappa^{k-r,(m)}_n$ generates $\pi^j \mathcal{H}'(n)$ (Theorem \ref{thm:kolyvagin-system-location}).
Then for every $r \geq 0$, we have
$$\partial^{(r)} (\ks^{k-r, n, (m)}) = \mathrm{min}\lbrace m, j+\sum_{i > r} d_i \rbrace $$
where $\ks^{k-r, n, (m)}$ is the Kolyvagin system defined in $\S$\ref{subsubsec:structure-fine-selmer-p^k}.
\end{thm}
\begin{proof}
It immediately follows from \cite[Prop. 4.5.8]{mazur-rubin-book} and \cite[Thm. 4.5.9]{mazur-rubin-book}.
\end{proof}

\subsection{Kolyvagin systems over discrete valuation rings}

\subsubsection{}
The following theorem reflects the rigidity of Kolyvagin systems.
\begin{thm} \label{thm:core-rank-modular-forms-char-zero}
Under our setting, the following statements hold.
\begin{enumerate}
\item $\KS(T_f(k-r), \mathcal{F}_{\mathrm{BK}}, \mathcal{P}_1) = 0$.
\item $\KS(T_f(k-r), \mathcal{F}_{\mathrm{can}}, \mathcal{P}_1)$ is free of rank one over $\mathcal{O}$.
\end{enumerate}
\end{thm}
\begin{proof}
See \cite[Thm. 5.2.10]{mazur-rubin-book}.
\end{proof}

\subsubsection{} \label{subsubsec:numerical-invariants-kappa}
For a given Kolyvagin system $\ks^{k-r}$ for $T_f(k-r)$, the numerical invariants associated to $\ks^{k-r}$ are defined as follows.
\begin{align*}
\mathrm{ord} \left( \ks^{k-r} \right) & = \mathrm{min} \left\lbrace \nu(n) : n \in \mathcal{N}_1 , \kappa^{k-r}_n \neq 0 \right\rbrace , \\
\partial^{(0)} \left( \ks^{k-r} \right) & = \mathrm{max} \left\lbrace j : \kappa^{k-r}_1 \in \pi^j \mathrm{Sel}_{\mathrm{rel}}(\mathbb{Q}, T_f(k-r)) \right\rbrace , \\
\mathrm{ord}_\pi \left( \kappa^{k-r}_n \right) & = \mathrm{max} \left\lbrace j_n : \kappa^{k-r}_n \in \pi^{j_n} \mathrm{Sel}_{\mathrm{rel},n}(\mathbb{Q}, T_f(k-r)/I_n)  \right\rbrace , \\
\partial^{(i)}(\ks^{k-r})  & = \mathrm{min} \left\lbrace \mathrm{ord}_\pi \left( \kappa^{k-r}_n \right) : n \in \mathcal{N}_1 \textrm{ with } \nu(n) = i \right\rbrace , \\
\partial^{(\infty)}(\ks^{k-r})  & = \mathrm{min} \left\lbrace \partial^{(i)}(\ks^{k-r}) : i \geq 0 \right\rbrace .
\end{align*}
We say $\partial^{(0)} (\ks^{k-r})  = \infty$ when $\kappa^{k-r}_1=0$.
We also define the elementary divisors of $\ks^{k-r}$ by
$$e_i(\ks^{k-r}) = \partial^{(i)}(\ks^{k-r}) - \partial^{(i+1)}(\ks^{k-r})$$
where $i \geq \mathrm{ord}(\ks^{k-r})$.

\begin{thm} \label{thm:kato-kolyvagin-main}
Let $\ks^{k-r}$ be  a non-trivial Kolyvagin system for $T_f(k-r)$.
Then the following statements hold.
\begin{enumerate}
\item For every $s \geq 0$,
$\partial^{(s)}(\ks^{k-r}) = \lim_{m \to \infty} \partial^{(s)}(\ks^{k-r,(m)}) $.
\item The sequence $\partial^{(s)}(\ks^{k-r})$ is non-increasing, and finite for $s \geq \mathrm{ord}(\ks^{k-r})$.
\item The sequence $e_i(\ks^{k-r})$ is non-increasing, non-negative, and finite for $i  \geq \mathrm{ord}(\ks^{k-r})$.
\item 
For any other non-trivial Kolyvagin system $\ks^{\circ,k-r}$ in $\KS(T_f(k-r))$,
$\mathrm{ord}(\ks^{k-r}) = \mathrm{ord}(\ks^{\circ, k-r})$ and $e_i(\ks^{k-r}) = e_i(\ks^{\circ, k-r})$.
\item $\mathrm{cork}_{\mathcal{O}} \mathrm{Sel}_0(\mathbb{Q},W_{\overline{f}}(r)) = \mathrm{ord}(\ks^{k-r})$.
\item 
$\mathrm{Sel}_0(\mathbb{Q}, W_{\overline{f}}(r))_{/\mathrm{div}} \simeq \bigoplus_{i \geq \mathrm{ord}(\ks^{k-r})} \dfrac{\mathcal{O}}{\pi^{e_i(\ks^{k-r})}\mathcal{O}}$
\item  $\mathrm{length}_{\mathcal{O}} \mathrm{Sel}_0(\mathbb{Q}, W_{\overline{f}}(r))_{/\mathrm{div}} =
\partial^{(\mathrm{ord}(\ks^{k-r}))}(\ks^{k-r}) - \partial^{(\infty)}(\ks^{k-r})$
\item $\ks^{k-r}$ is primitive if and only if $\partial^{(\infty)}(\ks^{k-r}) = 0$.
\end{enumerate}
\end{thm}
\begin{proof}
See \cite[Thm. 5.2.12]{mazur-rubin-book}.
\end{proof}
\begin{cor} \label{cor:kato-kolyvagin-main}
Let $\ks^{k-r}$ be  a non-trivial Kolyvagin system for $T_f(k-r)$.
\begin{enumerate}
\item $\mathrm{length}_{\mathcal{O}} \left( \mathrm{Sel}_0(\mathbb{Q}, W_{\overline{f}}(r)) \right)$ is finite if and only if $\kappa^{ k-r}_1 \neq 0$.
\item $\mathrm{length}_{\mathcal{O}} \left( \mathrm{Sel}_0(\mathbb{Q}, W_{\overline{f}}(r)) \right) \leq \partial^{(0)} (\ks^{k-r})$, with equality if and only if $\ks^{k-r}$ is primitive. 
\end{enumerate}
\end{cor}
\begin{proof}
See \cite[Cor. 5.2.13]{mazur-rubin-book}.
\end{proof}

\subsubsection{}
The following ``Theorem \ref{thm:main-all-critical} before taking $\mathrm{exp}^*$" slightly refines Kato's one-sided divisibility \cite[Thm. 14.5]{kato-euler-systems}  of the Tamagawa number conjecture for modular motives formulated by Kato and Fontaine--Perrin-Riou when the analytic rank is zero.
See \cite[$\S$8.2]{kato-euler-systems} for the notation of \'{e}tale cohomology group $\mathrm{H}^i(\mathbb{Z}[1/p], T_f(k-r) )$.
\begin{cor}[Mazur--Rubin] \label{cor:main-all-critical-before-dual-exp}
Assume that $\rho_f$ has large image.
For an integer $r$ with $1 \leq r \leq k-1$, if
$(1 - a_p(\overline{f}) p^{-r} + \psi^{-1}(p) p^{k-1-2r} ) \cdot L(\overline{f}, r) \neq 0 $, 
then
\[
\mathrm{length}_{\mathcal{O}} \left( \mathrm{Sel}_0(\mathbb{Q}, W_{\overline{f}}(r) ) \right)  = 
\mathrm{length}_{\mathcal{O}} \left( \dfrac{  \mathrm{Sel}_{\mathrm{rel}}(\mathbb{Q}, T_{f}(k-r) )  }{ \kappa^{\mathrm{Kato},k-r}_1 } \right) - \partial^{(\infty)} \left( \ks^{\mathrm{Kato}, k-r} \right)  .
\]
If we further assume that $\mathrm{H}^0(\mathbb{Q}_p, W_{\overline{f}}(r)) = 0$ and  $\overline{\rho}_f$ is ramified at all primes dividing $N$, then
then Kato's formulation of the Tamagawa number conjecture \cite[Conj. 2.3.3, Chap. I]{kato-lecture-1}
$$\mathrm{length}_{\mathcal{O}} \left( \mathrm{H}^2(\mathbb{Z}[1/p], T_f(k-r) ) \right)  = 
\mathrm{length}_{\mathcal{O}} \left( \dfrac{  \mathrm{H}^1(\mathbb{Z}[1/p], T_{f}(k-r) )  }{ \kappa^{\mathrm{Kato},k-r}_1 } \right)$$
is equivalent to $\partial^{(\infty)} \left( \ks^{\mathrm{Kato}, k-r} \right) = 0$, i.e. Kato's Kolyvagin system for $T_f(k-r)$ is primitive.
\end{cor}
\begin{proof}
The first equality follows from Theorem \ref{thm:kato-kolyvagin-main}.(7). The non-triviality of $\ks^{\mathrm{Kato}, k-r}$ follows from Kato's explicit reciprocity law 
(Theorem \ref{thm:zeta-morphism}) and the non-vanishing assumption. 
The second statement follows easily from the comparison among local conditions of Selmer groups and $\mathrm{H}^2(\mathbb{Z}[1/p], T_f(k-r) )$ under our assumptions (e.g. \cite[p. 217]{kurihara-invent} and \cite[$\S$14.8]{kato-euler-systems}).
See also \cite[Thm. 5.2.14]{mazur-rubin-book}.
\end{proof}
The explicit comparison between $\mathrm{Sel}_0(\mathbb{Q}, W_{\overline{f}}(r) )$ and $\mathrm{H}^2(\mathbb{Z}[1/p], T_f(k-r))$ should reveal a precise relation between Kato's formulation of the Tamagawa number conjecture and the arithmetic interpretation of $\partial^{(\infty)} \left( \ks^{\mathrm{Kato}, k-r} \right)$.
See $\S$\ref{sec:formulation-refined-conjecture} for the precise formulation of the general conjecture.
See also \cite[Conj. 5.15]{bloch-kato}, \cite[4.5.2.$C_{\mathrm{BK}, \ell}(M)$ (p. 702)]{fontaine-perrin-riou}, \cite[Conj. 4 (p. 100)]{burns-flach-annalen}, \cite[Conj. 4 (p. 535)]{burns-flach-documenta}, and \cite[Conj. 3]{flach-survey}.

%\begin{rem}
%In Corollary \ref{cor:main-all-critical-before-dual-exp}, if we further assume that $\mathrm{H}^0(\mathbb{Q}_p, W_{\overline{f}}(r)) = 0$, 
%by using the localization exact sequence and the global Poitou--Tate duality, Kato's reformulation of the Tamagawa number conjecture for $T_f(k-r)$ \cite[Conj. 2.3.3]{kato-lecture-1}, \cite[Thm. 14.5]{kato-euler-systems} is reduced to
%$$\sum_{\ell \vert N}  \mathrm{length}_{\mathcal{O}}\left(  \mathrm{H}^0(\mathbb{F}_\ell, \mathrm{H}^1(\mathbb{Q}^{\mathrm{ur}}_\ell, T_f(k-r)))_{\mathrm{tor}} \right)
%= \partial^{(\infty)} \left( \ks^{\mathrm{Kato}, k-r} \right) .$$
%\end{rem}

\subsection{Iwasawa theory}
\subsubsection{}
Let $\mathbb{Q}_\infty$ be the cyclotomic $\mathbb{Z}_p$-extension of $\mathbb{Q}$ and $\mathbb{Q}_s \subseteq \mathbb{Q}_\infty$ the cyclic subextension of $\mathbb{Q}$ of degree $p^s$ in $\mathbb{Q}_\infty$.
Let $\Lambda = \Lambda_{\mathrm{Iw}} = \mathbb{Z}_p\llbracket  \mathrm{Gal}(\mathbb{Q}_\infty/\mathbb{Q}) \rrbracket = \varprojlim_s \mathbb{Z}_p[  \mathrm{Gal}(\mathbb{Q}_s/\mathbb{Q}) ]$ be the Iwasawa algebra.

\subsubsection{}
Applying the $\Lambda$-adic version of Euler-to-Kolyvagin system map \cite[Thm. 5.3.3]{mazur-rubin-book} to Kato's Euler system $\mathbf{z}^{\mathrm{Kato},k-r}$, we obtain the $\Lambda$-adic Kato's Kolyvagin system $\ks^{\mathrm{Kato},k-r,\infty}$.
\begin{thm}[Kato--Rohrlich] \label{thm:kato-rohrlich}
Kato's zeta element $\kappa^{\mathrm{Kato},k-r,\infty}_{1} = z^{\mathrm{Kato},k-r}_{\mathbb{Q}_\infty}$ over $\mathbb{Q}_\infty$ is non-trivial.
\end{thm}
\begin{proof}
By using Kato's explicit reciprocity law (Theorem \ref{thm:zeta-morphism}),
the conclusion follows from the generic non-vanishing of twisted $L$-values \cite{rohrlich-nonvanishing, rohrlich-nonvanishing-2}.
\end{proof}
For a co-finitely generated module $M$, write $M^\vee = \mathrm{Hom}(M , F/\mathcal{O})$.
\begin{thm}[Kato] \label{thm:kato-12-4}
Suppose that $\overline{\rho}_f$ is irreducible. Then:
\begin{enumerate}
\item 
 $\mathrm{H}^1_{\mathrm{Iw}}(\mathbb{Q}, T_f(k-r))$ is free of rank one over $\Lambda$.
\item 
$\mathrm{Sel}_0(\mathbb{Q}_{\infty}, W_{\overline{f}}(r) )^\vee$ is a finitely generated torsion $\Lambda$-module.
\end{enumerate}
\end{thm}
\begin{proof}
See \cite[Thm. 12.4]{kato-euler-systems}. See also \cite[Lem. 5.3.5 and Thm. 5.3.6]{mazur-rubin-book}.
\end{proof}
\subsubsection{}
We recall the Iwasawa main conjecture without $p$-adic $L$-functions \cite[Conj. 12.10]{kato-euler-systems} and the known results.
\begin{conj}[Iwasawa main conjecture] \label{conj:main-conjecture}
Assume that $\rho_f$ has large image.
For a height one prime $\mathfrak{P}$ of $\Lambda$,
the Iwasawa main conjecture  for $T_f(k-r)$ localized at $\mathfrak{P}$ means the equality
$$ \mathrm{ord}_{\mathfrak{P}} \left( \mathrm{char}_{\Lambda} \left( \dfrac{\mathrm{H}^1_{\mathrm{Iw}}(\mathbb{Q}, T_f(k-r))}{\Lambda \kappa^{\mathrm{Kato}, k-r, \infty}_1 }  \right)  \right) = 
\mathrm{ord}_{\mathfrak{P}} \left( \mathrm{char}_{\Lambda} \left( \mathrm{Sel}_0(\mathbb{Q}_{\infty}, W_{\overline{f}}(r) )^\vee  \right) \right) .$$
The Iwasawa main conjecture for $T_f(k-r)$ means the Iwasawa main conjecture  for $T_f(k-r)$ localized at every height one prime ideal.
\end{conj}
We say that \textbf{$\ks^{\mathrm{Kato}, k-r,\infty}$ is $\Lambda$-primitive} if $\ks^{\mathrm{Kato}, k-r,\infty} \pmod{\mathfrak{P}}$ does not vanish for every height one prime $\mathfrak{P}$ of $\Lambda$ \cite[Def. 5.3.9]{mazur-rubin-book}.
\begin{thm}[Kato, Mazur--Rubin] \label{thm:mazur-rubin-main-conjecture}
Assume that $\rho_f$ has large image.
Let $\ks^{\mathrm{Kato}, k-r,\infty} \in \KS(T_f(k-r) \otimes \Lambda)$ be the $\Lambda$-adic Kato's Kolyvagin system for $T_f(k-r)$.
\begin{enumerate}
\item The one-sided divisibility of the Iwasawa main conjecture holds
$$ \mathrm{char}_{\Lambda} \left( \dfrac{\mathrm{H}^1_{\mathrm{Iw}}(\mathbb{Q}, T_f(k-r))}{\Lambda \kappa^{\mathrm{Kato}, k-r, \infty}_1 }   \right) \subseteq 
\mathrm{char}_{\Lambda} \left( \mathrm{Sel}_0(\mathbb{Q}_{\infty}, W_{\overline{f}}(r) )^\vee  \right) .$$
\item The specialization of $\ks^{\mathrm{Kato}, k-r,\infty}$ at a height one prime $\mathfrak{P}$ of $\Lambda$ is non-trivial if and only if the Iwasawa main conjecture  for $T_f(k-r)$ localized at $\mathfrak{P}$ holds, i.e.
$$ \mathrm{ord}_{\mathfrak{P}} \left( \mathrm{char}_{\Lambda} \left( \dfrac{\mathrm{H}^1_{\mathrm{Iw}}(\mathbb{Q}, T_f(k-r))}{\Lambda \kappa^{\mathrm{Kato}, k-r, \infty}_1 }  \right)  \right) = 
\mathrm{ord}_{\mathfrak{P}} \left( \mathrm{char}_{\Lambda} \left( \mathrm{Sel}_0(\mathbb{Q}_{\infty}, W_{\overline{f}}(r) )^\vee  \right) \right) .$$
\item The $\Lambda$-adic Kato's Kolyvagin system $\ks^{\mathrm{Kato}, k-r,\infty}$  is $\Lambda$-primitive if and only if
$$ \mathrm{char}_{\Lambda} \left( \dfrac{\mathrm{H}^1_{\mathrm{Iw}}(\mathbb{Q}, T_f(k-r))}{\Lambda \kappa^{\mathrm{Kato}, k-r, \infty}_1 }  \right)  = 
 \mathrm{char}_{\Lambda} \left( \mathrm{Sel}_0(\mathbb{Q}_{\infty}, W_{\overline{f}}(r) )^\vee  \right) .$$
\end{enumerate}
\end{thm}
\begin{proof}
See \cite[Thm. 12.5.(4)]{kato-euler-systems} and \cite[Thm. 5.3.10.(i)]{mazur-rubin-book} for the first statement.
Although \cite[Thm. 5.3.10.(ii) and (iii)]{mazur-rubin-book} claims one direction of the second and the third statements, the proof implies the equivalence.
Theorem \ref{thm:kato-rohrlich} is essential in the proof.
\end{proof}
\begin{thm}[Kato, Skinner--Urban, Wan] \label{thm:skinner-urban-wan}
Suppose that $\psi = \mathbf{1}$ and $f$ is good ordinary at $p$ and $p$-distinguished.
\begin{enumerate}
\item[(SU)]
If $\rho_f$ has large image, $k \equiv 2 \pmod{p-1}$, and there exists a prime $\ell$ exactly dividing $N$ where $\overline{\rho}_f$ is ramified, then
$$\mathrm{ord}_{\mathfrak{P}} \left( \mathrm{char}_{\Lambda} \left( \dfrac{\mathrm{H}^1_{\mathrm{Iw}}(\mathbb{Q}, T_f(k-1))}{\Lambda \kappa^{\mathrm{Kato}, k-1, \infty}_1 }  \right)  \right) = 
\mathrm{ord}_{\mathfrak{P}} \left(  \mathrm{char}_{\Lambda} \left( \mathrm{Sel}_0(\mathbb{Q}_{\infty}, W_{f}(1) )^\vee  \right)  \right) $$
for every height one prime $\mathfrak{P}$ of $\Lambda$.
\item[(Wan)] If $p \geq 5$ and $\overline{\rho}_f$ is irreducible, then
$$\mathrm{ord}_{\mathfrak{P}} \left( \mathrm{char}_{\Lambda} \left( \dfrac{\mathrm{H}^1_{\mathrm{Iw}}(\mathbb{Q}, T^{\dagger}_f)}{\Lambda \kappa^{\mathrm{Kato}, \dagger, \infty}_1 }  \right)  \right) = 
\mathrm{ord}_{\mathfrak{P}} \left(  \mathrm{char}_{\Lambda} \left( \mathrm{Sel}_0(\mathbb{Q}_{\infty}, W^{\dagger}_{f} )^\vee  \right)  \right) $$
for every height one prime $\mathfrak{P}$ except $\pi\Lambda$.
\end{enumerate}
\end{thm}
\begin{proof}
See \cite[Thm. 12.5]{kato-euler-systems}, \cite[Thm. 3.29]{skinner-urban}, and \cite[Thm. 4 (= Thm. 103)]{wan_hilbert}.
The latter two results concern the Iwasawa main conjecture with $p$-adic $L$-functions, which is equivalent to Conjecture \ref{conj:main-conjecture} via  the global Poitou--Tate duality \cite[$\S$17.13]{kato-euler-systems}.
\end{proof}
\section{The explicit reciprocity law with torsion coefficients} \label{sec:explicit-reciprocity-law}
In order to develop the precise connection between Kato's Kolyvagin systems and Kurihara numbers, it is natural to apply Kato's explicit reciprocity law to Kato's Kolyvagin system. However, it does not make sense literally since Kolyvagin systems are cohomology classes with torsion coefficients.
We first need to construct the dual exponential map with torsion coefficients.
The best way would be to compute the integral image of the dual exponential map precisely and take the mod $\pi^m$ reduction.
Computing this integral image is more or less equivalent to the computation of Fontaine--Perrin-Riou's local Tamagawa ideals at $p$ \cite{fontaine-perrin-riou, perrin-riou-book}, which seems out of reach in general. 
The computation requires a significant development of integral $p$-adic Hodge theory which we do not have yet.
See $\S$\ref{sec:formulation-refined-conjecture} for the ``idealistic" version we want.

In order to avoid this computation,  we just take the integral lattice in the target of the dual exponential map by the integral image of the dual exponential map itself. 
Then the dual exponential map with torsion coefficients can be defined by taking the mod $\pi^m$ reduction of this lattice.
As a consequence,  we will only have the ``equality" between the image of Kato's Kolyvagin system under the dual exponential map and the collection of Kurihara numbers \emph{up to a non-zero uniform constant}. 
Thus, we are not able to know the exact valuation of the image of \emph{each} Kato's Kolyvagin system class under the dual exponential map (with torsion coefficients). 
%since we are not able to compute the integral image of the dual exponential map.

We ignore this uniform constant according to our design since it is cancelled out when we compute the difference between the valuations of the images of two different Kato's Kolyvagin system classes under the dual exponential map.  The difference is only important when we describe the structure of Selmer groups.
To sum up, the \emph{integral} nature of the explicit reciprocity law is extremely crucial for us, but \emph{no} integral $p$-adic Hodge theory is needed for our purpose.

\subsection{The integral lattice in the Bloch--Kato dual exponential map} \label{subsec:integrality-bloch-kato}
For $n \in \mathcal{N}_1$, consider exact sequence
\[
\xymatrix{
\dfrac{\mathrm{H}^1(\mathbb{Q}_p, T_f(k-r))}{I_n \mathrm{H}^1(\mathbb{Q}_p, T_f(k-r))} \ar@{^{(}->}[r] &
\mathrm{H}^1(\mathbb{Q}_p, T_f(k-r)/I_n) \ar@{->>}[r] &
\mathrm{H}^2(\mathbb{Q}_p, T_f(k-r))[I_n]  .
}
\]
Since $\mathrm{H}^1_f(\mathbb{Q}_p, T_f(k-r)/I_n)$ is defined by the image of 
$\dfrac{\mathrm{H}^1_f(\mathbb{Q}_p, T_f(k-r))}{ I_n\mathrm{H}^1_f(\mathbb{Q}_p, T_f(k-r))}$ in $\mathrm{H}^1(\mathbb{Q}_p, T_f(k-r)/I_n)$,
we obtain the following exact sequence
\begin{equation} \label{eqn:extending-dual-exp}
\xymatrix{
 \dfrac{\mathrm{H}^1_{/f}(\mathbb{Q}_p, T_f(k-r))}{I_n \mathrm{H}^1_{/f}(\mathbb{Q}_p, T_f(k-r)) } \ar@{^{(}->}[r] &
\mathrm{H}^1_{/f}(\mathbb{Q}_p, T_f(k-r)/I_n)  \ar@{->>}[r] &
\mathrm{H}^2(\mathbb{Q}_p, T_f(k-r))[I_n] 
}
\end{equation}
where
$\dfrac{\mathrm{H}^1_{/f}(\mathbb{Q}_p, T_f(k-r))}{I_n \mathrm{H}^1_{/f}(\mathbb{Q}_p, T_f(k-r))  } = \dfrac{\mathrm{H}^1(\mathbb{Q}_p, T_f(k-r))}{I_n \mathrm{H}^1(\mathbb{Q}_p, T_f(k-r))  + \mathrm{H}^1_f(\mathbb{Q}_p, T_f(k-r)) }$.

Thanks to Theorem \ref{thm:bloch-kato-exponential} and the local Tate duality, we have an isomorphism of one-dimensional $F$-vector spaces
$\mathrm{exp}^{*} : \mathrm{H}^1_{/f}(\mathbb{Q}_p, V_f(k-r)) \simeq \mathrm{Fil}^0\mathbf{D}_{\mathrm{dR}}(V_f(k-r)) $.
Denote the image of $\mathrm{H}^1_{/f}(\mathbb{Q}_p, T_f(k-r))$ under $\mathrm{exp}^{*} $ by
$$\mathcal{L} := \mathrm{exp}^{*} \left(  \mathrm{H}^1_{/f}(\mathbb{Q}_p, T_f(k-r)) \right) \subseteq \mathrm{Fil}^0\mathbf{D}_{\mathrm{dR}}(V_f(k-r)).$$
Since $\mathcal{L}$ is a (non-zero) $\mathcal{O}$-lattice in $\mathrm{Fil}^0\mathbf{D}_{\mathrm{dR}}(V_f(k-r))$, $\mathcal{L}$ is a free $\mathcal{O}$-module of rank one, and we fix a generator $\omega_{f,r}$ of $\mathcal{L}$.
We expect that the ``local torsion" $\mathrm{length}_{\mathcal{O}}\mathrm{H}^0(\mathbb{Q}_p, W_{\overline{f}}(r))$ and the Euler factor at $p$ are involved in the choice of $\omega_{f,r}$.

The mod $I_n$ reduction of $\mathcal{L}$ yields an isomorphism of free $\mathcal{O}/I_n\mathcal{O}$-modules of rank one induced from the dual exponential map
$$\dfrac{\mathrm{H}^1_{/f}(\mathbb{Q}_p, T_f(k-r))}{I_n \mathrm{H}^1_{/f}(\mathbb{Q}_p, T_f(k-r))  } \simeq  \mathcal{L} / I_n \mathcal{L} \simeq \mathcal{O}/I_n\mathcal{O} .$$
Thus, the sequence (\ref{eqn:extending-dual-exp}) splits as $\mathcal{O}/I_n \mathcal{O}$-modules.
By using this splitting, the dual exponential map on $\mathrm{H}^1(\mathbb{Q}_p, T_f(k-r)/I_n)$ is defined by
$$\mathrm{exp}^{*}: \mathrm{H}^1(\mathbb{Q}_p, T_f(k-r)/I_n) \to \mathrm{H}^1_{/f}(\mathbb{Q}_p, T_f(k-r)/I_n)
\to
\dfrac{\mathrm{H}^1_{/f}(\mathbb{Q}_p, T_f(k-r))}{I_n \mathrm{H}^1_{/f}(\mathbb{Q}_p, T_f(k-r)) } 
\simeq
\mathcal{L} / I_n \mathcal{L} .$$

\subsection{From Kato's Kolyvagin systems to Kurihara numbers} \label{subsec:from-kato-to-kurihara}
Recall Birch's lemma \cite[Lem. 5.4.5]{bellaiche-book}
$$\sum_{a \in (\mathbb{Z}/n\mathbb{Z})^\times} \lambda^{\pm, \mathrm{min}}_{\overline{f}}(z^{r-1}; a, n) \cdot \chi(a)  = 
\dfrac{  (r-1)! \cdot \tau(\chi) }{ (-2 \pi \sqrt{-1})^r} \cdot \dfrac{L(\overline{f}, \chi^{-1}, r)}{\Omega^{\pm}_{\overline{f}, \mathrm{min}}} .$$
By factorizing Gauss sum $\tau(\chi)$ with $\chi(-1) \cdot \dfrac{\tau(\chi^{-1})}{n} =  \tau(\chi)^{-1}$, we have
$$
\dfrac{  (r-1)!  }{ (-2 \pi \sqrt{-1})^r} \cdot \dfrac{L(\overline{f}, \chi^{-1}, r)}{\Omega^{\pm}_{\overline{f}, \mathrm{min}}}  = \sum_{c \in (\mathbb{Z}/n\mathbb{Z})^\times} \sigma_c \left( \dfrac{1}{n} \cdot  \sum_{a \in (\mathbb{Z}/n\mathbb{Z})^\times} \zeta^{-a}_n  \cdot \lambda^{\pm, \mathrm{min}}_{\overline{f}}(z^{r-1}; a, n) \right) \cdot \chi^{-1}(c)  
$$
where $\sigma_c  \in \mathrm{Gal}(\mathbb{Q}(\zeta_n)/\mathbb{Q})$ sends $\zeta_n$ to $\zeta^c_n$.
%Since $\chi(-1) \cdot \dfrac{\tau(\chi^{-1})}{n} =  \tau(\chi)^{-1}$, we have
%\begin{align*}
%\dfrac{  (r-1)!  }{ (-2 \pi \sqrt{-1})^r} \cdot \dfrac{L(\overline{f}, \chi^{-1}, r)}{\Omega^{\pm}_{\overline{f}, \mathrm{min}}} 
%& = \chi(-1) \cdot \dfrac{\tau(\chi^{-1})}{n}  \cdot \sum_{a \in (\mathbb{Z}/n\mathbb{Z})^\times} \left[ \dfrac{a}{n} \right]^{\pm}_{\overline{f}, r} \cdot \chi(a)  \\
%& = \dfrac{\chi(-1)}{n} \cdot \left( \sum_{b \in (\mathbb{Z}/n\mathbb{Z})^\times} \chi^{-1}(b) \cdot \zeta^b_n \right)  \cdot \sum_{a \in (\mathbb{Z}/n\mathbb{Z})^\times} \left[ \dfrac{a}{n} \right]^{\pm}_{\overline{f}, r} \cdot \chi(a)  \\
%& = \dfrac{1}{n} \cdot  \sum_{b \in (\mathbb{Z}/n\mathbb{Z})^\times} \sum_{a \in (\mathbb{Z}/n\mathbb{Z})^\times} \zeta^b_n  \cdot  \left[ \dfrac{a}{n} \right]^{\pm}_{\overline{f}, r} \cdot \chi^{-1}(b) \cdot  \chi(-a) \\
%& = \dfrac{1}{n} \cdot  \sum_{c \in (\mathbb{Z}/n\mathbb{Z})^\times} \sum_{a \in (\mathbb{Z}/n\mathbb{Z})^\times} \zeta^{-ac}_n  \cdot  \left[ \dfrac{a}{n} \right]^{\pm}_{\overline{f}, r} \cdot \chi^{-1}(-ac) \cdot  \chi(-a)  & (b= -ac) \\
%& = \dfrac{1}{n} \cdot  \sum_{c \in (\mathbb{Z}/n\mathbb{Z})^\times} \sum_{a \in (\mathbb{Z}/n\mathbb{Z})^\times} \zeta^{-ac}_n  \cdot  \left[ \dfrac{a}{n} \right]^{\pm}_{\overline{f}, r} \cdot \chi^{-1}(c) \\
%& = \sum_{c \in (\mathbb{Z}/n\mathbb{Z})^\times} \sigma_c \left( \dfrac{1}{n} \cdot  \sum_{a \in (\mathbb{Z}/n\mathbb{Z})^\times} \zeta^{-a}_n  \cdot  \left[ \dfrac{a}{n} \right]^{\pm}_{\overline{f}, r} \right) \cdot \chi^{-1}(c)  .
%\end{align*}
Write
$ \mathrm{exp}^*\circ \mathrm{loc}_p\left( z^{\mathrm{Kato}, k-r}_{\mathbb{Q}(\zeta_n)} \right) = c_{k-r,n} \cdot \omega_{f,r}$
with $c_{k-r,n} \in \mathbb{Q}(\zeta_n) \otimes F$.
Then Theorem \ref{thm:zeta-morphism} says
$$
 \sum_{c \in (\mathbb{Z}/n\mathbb{Z})^\times} \sigma_c \left(  c_{k-r,n}  \right) \cdot \chi^{-1}(c) \cdot \mathrm{per}_f \left( \omega_{f,r} \right) 
=  (2\pi \sqrt{-1})^{k-r-1} \cdot L^{(p)}(\overline{f}, \chi^{-1}, r) \cdot \gamma^{\pm}
$$
where $\mathrm{per}_f \left( \omega_{f,r} \right) = \Omega^{+}_{\gamma, \omega_{f,r}} \cdot \gamma^+ +  \Omega^{-}_{\gamma, \omega_{f,r}} \cdot \gamma^-$.
In \cite[$\S$7.6]{kato-euler-systems}, the convention of the period integral is given by 
\begin{equation}
\label{eqn:period-kato-mtt}
\Omega(f,r)  = (2\pi \sqrt{-1})^{k-1} \cdot(-2\pi \sqrt{-1})^{-r}  \cdot (r-1)! \cdot L(f, r) .
\end{equation}
See also \cite[(7.13.6) and Thm. 16.2]{kato-euler-systems}.
Inverting Kato's period map $\mathrm{per}_f$ with Kato's convention (\ref{eqn:period-kato-mtt}), we have
\begin{align*}
\begin{split}
& \sum_{c \in (\mathbb{Z}/n\mathbb{Z})^\times} \sigma_c \left(  c_{k-r,n}  \right) \cdot \chi^{-1}(c) \cdot  \omega_{f,r}  \\
= \ & \chi^{-1}(E_p(\overline{f}, r))  \cdot \sum_{c \in (\mathbb{Z}/n\mathbb{Z})^\times} \sigma_c \left( \dfrac{1}{n} \cdot  \sum_{a \in (\mathbb{Z}/n\mathbb{Z})^\times} \zeta^{-a}_n  \cdot \lambda^{\pm, \mathrm{min}}_{\overline{f}}(z^{r-1}; a, n) \right) \cdot \chi^{-1}(c) \cdot C^{\pm}_{\mathrm{per}} \cdot \omega_{f,r}
\end{split}
\end{align*}
%\begin{align*}
%\begin{split}
%& \sum_{c \in (\mathbb{Z}/n\mathbb{Z})^\times} \sigma_c \left(  c_{k-r,n}  \right) \cdot \chi^{-1}(c) \cdot  \omega_{f,r}  \\
%= \ &  (2\pi \sqrt{-1})^{k-r-1}  \cdot \dfrac{L^{(p)}(\overline{f}, \chi^{-1}, r)}{\Omega^{\pm}_{\gamma, \omega_{f,r}} } \cdot \omega_{f,r} \\
%= \ &   \dfrac{(r-1)!}{(2\pi \sqrt{-1})^{r} } \cdot \dfrac{L^{(p)}(\overline{f}, \chi^{-1}, r)}{\Omega^{\pm}_{\overline{f}, \mathrm{min}} } \cdot C_{\mathrm{per}} \cdot \omega_{f,r} \\
%= \ & \chi^{-1}(E_p(\overline{f}, r))  \cdot \sum_{c \in (\mathbb{Z}/n\mathbb{Z})^\times} \sigma_c \left( \dfrac{1}{n} \cdot  \sum_{a \in (\mathbb{Z}/n\mathbb{Z})^\times} \zeta^{-a}_n  \cdot \lambda^{\pm, \mathrm{min}}_{\overline{f}}(z^{r-1}; a, n) \right) \cdot \chi^{-1}(c) \cdot C_{\mathrm{per}} \cdot \omega_{f,r}
%\end{split}
%\end{align*}
where 
\begin{itemize}
\item $\chi^{-1}(E_p(\overline{f}, r)) =  1 - a_p(\overline{f})  \cdot p^{-r} \cdot \chi^{-1}(p) - \psi^{-1}(p)  \cdot p^{k-1-2r} \cdot (\chi^{-1}(p) )^2$, 
\item the sign of $\Omega^{\pm}_{\gamma, \omega_{f,r}}$ is that of $(-1)^{k-r-1} \cdot \chi(-1)$,
\item the sign of $\Omega^{\pm}_{\overline{f}, \mathrm{min}}$ is that of $(-1)^{r-1} \cdot \chi(-1) $, and
\item  $C^{\pm}_{\mathrm{per}}$ is an algebraic integer such that
the sign of $C^{\pm}_{\mathrm{per}}$ is  that of $(-1)^{r-1} \cdot \chi(-1) $
and it satisfies
$$(2\pi \sqrt{-1})^{k-1} \cdot \Omega^{\pm}_{\overline{f}, \mathrm{min}} =  C^{\pm}_{\mathrm{per}}  \cdot (r-1)! \cdot \Omega^{\pm}_{\gamma, \omega_{f,r}} $$
up to $\mathcal{O}^\times$ with the relevant signs of periods. See also \cite[$\S$4.5.6]{fukaya-kato-sharificonj} for the sign convention of periods.
\end{itemize}
\begin{rem}
The constant $C^{\pm}_{\mathrm{per}}$ essentially measures the difference between Kato's period integrals with respect to $(\gamma, \omega_{f,r})$ and the minimal integral periods. The explicit computation of $C^{\pm}_{\mathrm{per}}$ seems out of reach in general.
\end{rem}
Since the above equality holds for every character $\chi$ on $(\mathbb{Z}/n\mathbb{Z})^\times$,
the equality lifts to ($p$-inverted!) group ring $F[\mathrm{Gal}(\mathbb{Q}(\zeta_n)/\mathbb{Q})] \otimes_{F} \mathrm{Fil}^0\mathbf{D}_{\mathrm{dR}}(V_f(k-r))$ (e.g. \cite[Cor. 5.13]{ota-thesis} and \cite[$\S$7.2]{kks})
\begin{align} \label{eqn:equality-in-group-ring}
\begin{split}
& \sum_{c \in (\mathbb{Z}/n\mathbb{Z})^\times} \sigma_c \left(  c_{k-r,n}  \right) \cdot \sigma_c \cdot  \omega_{f,r}  \\
= \ & 
E_p(\overline{f}, r) \cdot \sum_{c \in (\mathbb{Z}/n\mathbb{Z})^\times} \sigma_c \left( \dfrac{1}{n} \cdot  \sum_{a \in (\mathbb{Z}/n\mathbb{Z})^\times} \zeta^{-a}_n  \cdot \lambda^{\pm, \mathrm{min}}_{\overline{f}}(z^{r-1}; a, n) \right) \cdot \sigma_c \cdot C^{\pm}_{\mathrm{per}} \cdot \omega_{f,r}
\end{split}
\end{align}
where $E_p(\overline{f}, r) = 1 - a_p(\overline{f})  \cdot p^{-r} \cdot \sigma_p - \psi^{-1}(p)  \cdot p^{k-1-2r} \cdot (\sigma_p )^2$ and
 $\sigma_p$ is the arithmetic Frobenius at $p$ in $\mathrm{Gal}(\mathbb{Q}(\zeta_n)/\mathbb{Q})$.
Since the modular symbols are integrally normalized, the equality indeed holds in $\mathcal{O}[\mathrm{Gal}(\mathbb{Q}(\zeta_n)/\mathbb{Q})] \otimes_{\mathcal{O}} \mathcal{L}$.
By comparing the coefficients of $\sigma_c$ in (\ref{eqn:equality-in-group-ring}), we have
$$c_{k-r,n}   = E_p(\overline{f}, r) \cdot \dfrac{1}{n} \cdot  \sum_{a \in (\mathbb{Z}/n\mathbb{Z})^\times} \zeta^{-a}_n  \cdot  \lambda^{\pm, \mathrm{min}}_{\overline{f}}(z^{r-1}; a, n) \cdot C^{\pm}_{\mathrm{per}} .$$
Since $ D_{\mathbb{Q}(\zeta_n)}\mathrm{exp}^*\circ \mathrm{loc}_p\left( z^{\mathrm{Kato}, k-r}_{\mathbb{Q}(\zeta_n)} \right) =  D_{\mathbb{Q}(\zeta_n)} c_{k-r,n} \cdot \omega_{f,r}$,
we have
\begin{align*}
& E_p(\overline{f}, r)  \cdot D_{\mathbb{Q}(\zeta_n)} \left( \dfrac{1}{n} \cdot  \sum_{a \in (\mathbb{Z}/n\mathbb{Z})^\times} \zeta^{-a}_n  \cdot  \lambda^{\pm, \mathrm{min}}_{\overline{f}}(z^{r-1}; a, n) \right) \cdot C^{\pm}_{\mathrm{per}}  \\
\equiv  \ & E_p(\overline{f}, r) \cdot  \left( \dfrac{1}{n} \cdot  \sum_{a \in (\mathbb{Z}/n\mathbb{Z})^\times}   \lambda^{\pm, \mathrm{min}}_{\overline{f}}(z^{r-1}; a, n) \cdot \prod_{\ell \vert n} \mathrm{log}_{\eta_\ell}(a)   \right) \cdot C^{\pm}_{\mathrm{per}}  \pmod{I_n}  \\ 
= \ & \mathbf{1}(E_p(\overline{f}, r)) \cdot 
C^{\pm}_{\mathrm{per}} \cdot \widedelta^{\mathrm{min}, r}_n \in   \mathcal{O}/I_n\mathcal{O}
\end{align*}
where the congruence modulo $I_n$ follows from \cite[Lem. 4.4]{kurihara-documenta} (cf. \cite[Thm. 7.5]{kks}), and the last equality follows from the trivial action of $\mathrm{Gal}(\mathbb{Q}(\zeta_n)/\mathbb{Q})$ on $\widedelta^{\mathrm{min}, r}_n$ and  $n \in \mathcal{N}_1$.
Thus, we have
$$\mathrm{exp}^* \circ \mathrm{loc}^s_p ( \kappa^{\mathrm{Kato}, k-r}_n ) = \left( 1 - a_p(\overline{f})  \cdot p^{-r}  - \psi^{-1}(p)  \cdot p^{k-1-2r}  \right) \cdot  
C^{\pm}_{\mathrm{per}} \cdot \widedelta^{\mathrm{min}, r}_n \cdot \omega_{f,r} $$
and 
$\left( 1 - a_p(\overline{f})  \cdot p^{-r}  - \psi^{-1}(p)  \cdot p^{k-1-2r}  \right) \cdot C^{\pm}_{\mathrm{per}}$ is
independent of $n$.
%The above construction yields the family of explicit reciprocity laws with $\mathcal{O}/I_n \mathcal{O}$-coefficients 
%$$\mathrm{exp}^* \circ \mathrm{loc}^s_p : \mathcal{H}(n) \to \mathcal{L}/I_n\mathcal{L}$$
%for every $n \in \mathcal{N}_1$, and these maps send $\ks^{\mathrm{Kato}, k-r}$ to $\kn^{\mathrm{min}, r}$ up to the  uniform constant 
%$$\left( 1 - a_p(\overline{f})  \cdot p^{-r}  - \psi^{-1}(p)  \cdot p^{k-1-2r}  \right) \cdot C_{\mathrm{per}}$$
%independent of $n$.
\begin{prop} \label{prop:equivalence-nonvanishing}
Assume that $\rho_f$ has large image and $1 - a_p(\overline{f})  \cdot p^{-r}  - \psi^{-1}(p)  \cdot p^{k-1-2r}  \neq 0$.
Then
$$\ks^{\mathrm{Kato}, k-r} \neq \left\lbrace 0 \right\rbrace \Leftrightarrow \kn^{\mathrm{min}, r} \neq \left\lbrace 0 \right\rbrace .$$
\end{prop}
\begin{proof}
The $\Rightarrow$ direction follows from Theorem \ref{thm:core-rank-modular-forms-char-zero}, and the opposite direction is straightforward.
\end{proof}

\section{A ``relative" Kolyvagin system argument} \label{sec:proof-main-formulas}
\subsection{Overview of the strategy}
We first explain the idea of proof of Theorem \ref{thm:main-central-critical}.  (Theorem \ref{thm:main-all-critical} is easier.)
What we need to do is to upgrade Mazur--Rubin's structure theorem for $p$-strict Selmer groups (Theorem \ref{thm:kato-kolyvagin-main}) to that for Bloch--Kato Selmer groups.
Then it is natural to compare $\mathrm{Sel}_{0,n}(\mathbb{Q}, W^\dagger_f[I_n])$ with $\mathrm{Sel}_n(\mathbb{Q}, W^\dagger_f[I_n])$  for every $n  \in \mathcal{N}_1$,
and it can be done by applying the global duality argument for every $n  \in \mathcal{N}_1$. 
In the corresponding zeta element side,  we consider the difference between $\pi$-valuations of $\kappa^{\mathrm{Kato}, \dagger}_n$ and
$\mathrm{loc}^s_p \kappa^{\mathrm{Kato}, \dagger}_n$. Although each object is very difficult to compute, it is not terribly difficult to observe the following equality
$$\mathrm{length}_{\mathcal{O}}\mathrm{Sel}_n(\mathbb{Q}, W^\dagger_f[I_n]) - \mathrm{length}_{\mathcal{O}}\mathrm{Sel}_{0,n}(\mathbb{Q}, W^\dagger_f[I_n]) = \mathrm{ord}_\pi ( \mathrm{loc}^s_p \kappa^{\mathrm{Kato}, \dagger}_n ) - \mathrm{ord}_\pi ( \kappa^{\mathrm{Kato}, \dagger}_n) .$$
However, this equality is not strong enough to transplant the structure theorem. 
In particular, $\mathrm{loc}^s_p \kappa^{\mathrm{Kato}, \dagger}_n = 0$ if $\nu(n) \not\equiv \mathrm{cork}_{\mathcal{O}} \mathrm{Sel}(\mathbb{Q}, W^\dagger_f) \pmod{2}$ since $\widedelta^{\mathrm{min}, \dagger}_{n} = \mathrm{exp}^* \circ \mathrm{loc}^s_p \kappa^{\mathrm{Kato}, \dagger}_n$ up to a non-zero uniform constant and it vanishes by functional equation.  Thus, we do not expect to control $\mathrm{Sel}_n(\mathbb{Q}, W^\dagger_f[I_n])$ by using $\widedelta^{\mathrm{min}, \dagger}_{n}$ when the parity of $n$ does not match up with the parity of the root number of $f$.
However, this lack of control can be filled by using Flach's generalized Cassels--Tate pairing.
Furthermore, the vanishing of $\widedelta^{\mathrm{min}, \dagger}_{n}$ implies $\kappa^{\mathrm{Kato}, \dagger}_n \in \mathrm{Sel}_{n}(\mathbb{Q}, T^\dagger_f / I_n)$.
Since we have the self-duality $T^\dagger_f / I_n \simeq W^\dagger_f[I_n]$, the Chebotarev density type argument with choice of \emph{two} primes in $\mathcal{N}_1$ and some tricks yields the conclusion.

This argument should be viewed as a structural refinement of the four term exact sequence argument in Iwasawa theory,  which is used to obtain the equivalence between two different main conjectures (e.g. \cite[$\S$17.13]{kato-euler-systems}). 

Mazur--Rubin's Kolyvagin system argument gives us a filtration of $p$-strict Selmer groups by annihilating generators (Theorem \ref{thm:structure-fine-selmer-p^k}). Our refined four term exact sequence argument also yields a filtration of Bloch--Kato Selmer groups ``with parity constraint". As we emphasized, this parity constraint is natural due to the generalized Cassels--Tate pairing and the functional equation.

\subsection{The toolbox: global duality and generalized Cassels--Tate pairings} \label{subsec:toolbox}
\subsubsection{}
From now on, write
\begin{align*}
\mathrm{coker} ( \mathrm{loc}^s_p(\mathrm{Sel}^{k-r}_{\mathrm{rel},n}) )^\vee & = \left( \dfrac{  \mathrm{H}^1_{/f}(\mathbb{Q}_p, T_f(k-r)/I_n)  }{ \mathrm{loc}^s_p \left( \mathrm{Sel}_{\mathrm{rel},n}(\mathbb{Q}, T_f(k-r)/I_n) \right) } \right)^\vee , \\
\mathrm{coker} ( \mathrm{loc}^s_p(\mathrm{Sel}^\dagger_{\mathrm{rel},n}) )^\vee & = \left( \dfrac{  \mathrm{H}^1_{/f}(\mathbb{Q}_p, T^\dagger_f/I_n)  }{ \mathrm{loc}^s_p \left( \mathrm{Sel}_{\mathrm{rel},n}(\mathbb{Q}, T^\dagger_f/I_n) \right) } \right)^\vee
\end{align*}
for convenience, and these are cyclic modules. When $n=1$, we omit $n$.

\subsubsection{Global duality}
By the global Poitou--Tate duality, we have the exact sequences
\begin{equation} \label{eqn:kappa_n-widedelta_n-sequence}
\begin{split}
\xymatrix@R=0em@C=2em{
		 \mathrm{Sel}_{0}(\mathbb{Q}, W_{\overline{f}}(r)) 
		 \ar@{^{(}->}[r]  & \mathrm{Sel}(\mathbb{Q}, W_{\overline{f}}(r)) 
       	\ar@{->>}[r]^-{\mathrm{loc}_p} & \mathrm{coker} ( \mathrm{loc}^s_p(\mathrm{Sel}^{k-r}_{\mathrm{rel}}) )^\vee   , \\
       			\mathrm{Sel}_{0,n}(\mathbb{Q}, W_{\overline{f}}(r)[I_n]) 
		 \ar@{^{(}->}[r]  & \mathrm{Sel}_n(\mathbb{Q}, W_{\overline{f}}(r)[I_n]) 
       	\ar@{->>}[r]^-{\mathrm{loc}_p} & \mathrm{coker} ( \mathrm{loc}^s_p(\mathrm{Sel}^{k-r}_{\mathrm{rel},n}) )^\vee  .
}
\end{split}
\end{equation}
Fix a \emph{non-trivial} Kolyvagin system 
$\ks^{k-r} = \left\lbrace \kappa^{k-r}_n : n \in \mathcal{N}_1 \right\rbrace \in \KS(T_f(k-r))$,
and denote by $\mathrm{loc}^s_p \kappa^{k-r}_n$ the image of $\kappa^{k-r}_n$ under the map
$\mathrm{loc}^s_p : \mathrm{Sel}_{\mathrm{rel},n}(\mathbb{Q}, T_f(k-r)/I_n) \to \mathrm{H}^1_{/f}(\mathbb{Q}_p, T_f(k-r)/I_n)$.
For notational convenience, we write
\begin{align*}
\mathrm{ord}_{\pi}(\kappa^{k-r}_n) & = \mathrm{min} \left\lbrace j : \kappa^{k-r}_n \in \pi^j  \mathrm{Sel}_{\mathrm{rel},n}(\mathbb{Q}, T_f(k-r)/I_n) \right\rbrace  , \\
\mathrm{ord}_{\pi}( \mathrm{loc}^s_p \kappa^{k-r}_n ) & = \mathrm{min} \left\lbrace j : \mathrm{loc}^s_p \kappa^{k-r}_n \in \pi^j  \mathrm{H}^1_{/f}(\mathbb{Q}_p, T_f(k-r)/I_n)  \right\rbrace .
\end{align*}
By Theorem \ref{thm:kolyvagin-system-location}, we have
\begin{align} \label{eqn:kappa_n-widedelta_n}
\begin{split}
& \mathrm{ord}_\pi(\kappa^{k-r}_1)
+ \mathrm{length}_{\mathcal{O}} ( \mathrm{coker} ( \mathrm{loc}^s_p(\mathrm{Sel}^{k-r}_{\mathrm{rel}}) )^\vee )
= \mathrm{ord}_\pi(\mathrm{loc}^s_p \kappa^{k-r}_1) , \\
& \mathrm{ord}_\pi(\kappa^{k-r}_n)
+ \mathrm{length}_{\mathcal{O}} ( \mathrm{coker} ( \mathrm{loc}^s_p(\mathrm{Sel}^{k-r}_{\mathrm{rel},n}) )^\vee )
= \mathrm{ord}_\pi(\mathrm{loc}^s_p \kappa^{k-r}_n)  .
\end{split} 
\end{align}
From (\ref{eqn:kappa_n-widedelta_n}), we have the following statement immediately.
\begin{prop} \label{prop:behavior-kappa-n-localization}
Let $n \in \mathcal{N}_m$.
The followings are equivalent.
\begin{enumerate}
\item $\kappa^{k-r}_n \in \mathrm{Sel}_{n}(\mathbb{Q}, T_f(k-r)/I_n)$.
\item $\mathrm{loc}^s_p \kappa^{k-r}_n =0$.
\item $\mathrm{ord}_\pi(\kappa^{k-r}_n)
+ \mathrm{length}_{\mathcal{O}} ( \mathrm{coker} ( \mathrm{loc}^s_p(\mathrm{Sel}^{k-r}_{\mathrm{rel},n}) )^\vee ) \geq  \mathrm{length}_{\mathcal{O}}( \mathcal{O} /I_n\mathcal{O} )$.
\end{enumerate}
\end{prop}
%\begin{proof}
%$(1) \Leftrightarrow (2)$: 
%It follows from that 
%$$\mathrm{Sel}_{n}(\mathbb{Q}, T_f(k-r)/I_n) = \mathrm{ker} \left( \mathrm{loc}^s_p:  \mathrm{Sel}_{n}(\mathbb{Q}, T_f(k-r)/I_n) \to \mathrm{H}^1_{/f}(\mathbb{Q}_p, T_f(k-r)/I_n) \right).$$
%$(1) \Leftrightarrow (3)$: It follows . 
%\end{proof}

\begin{lem} \label{lem:vanishing-localization}
Let $\ks^{k-r}$ is a non-trivial Kolyvagin system.
Then
$$\partial^{(\infty)}(\ks^{k-r}) = \partial^{(\infty)}(\mathrm{loc}^s_p\ks^{k-r}) $$
where
$\partial^{(\infty)}(\mathrm{loc}^s_p\ks^{k-r}) = \mathrm{min} \left\lbrace \mathrm{ord}_\pi(\mathrm{loc}^s_p\kappa^{k-r}_n) : n \in \mathcal{N}_1 \right\rbrace $.
\end{lem}
\begin{proof}
We first suppose that $\mathrm{ord}_\pi(\kappa^{k-r}_n) = \mathrm{ord}_\pi(\mathrm{loc}^s_p\kappa^{k-r}_n) $ for some $n \in \mathcal{N}_1$.
Then 
$$ \mathrm{length}_{\mathcal{O}}  \mathrm{Sel}_{n}(\mathbb{Q}, W_{\overline{f}}(r) [I_{n}]) =  \mathrm{length}_{\mathcal{O}}  \mathrm{Sel}_{0,n}(\mathbb{Q}, W_{\overline{f}}(r) [I_n]) $$
by (\ref{eqn:kappa_n-widedelta_n-sequence}) and (\ref{eqn:kappa_n-widedelta_n}), so
$\mathrm{Sel}_{n}(\mathbb{Q}, W_{\overline{f}}(r) [I_{n}]) \simeq  \mathrm{Sel}_{0,n}(\mathbb{Q}, W_{\overline{f}}(r) [I_n])$.
By the Kolyvagin system argument (Theorem \ref{thm:structure-fine-selmer-p^k}), there exists $m \in \mathcal{N}_1$ dividing $n$ such that
\begin{itemize}
\item $\mathrm{ord}_\pi \kappa^{k-r}_m = \partial^{(\infty)}(\ks^{k-r}) < \infty$, so 
\item $\mathrm{length}_{\mathcal{O}}  \mathrm{Sel}_{m}(\mathbb{Q}, W_{\overline{f}}(r) [I_{m}]) =   \mathrm{length}_{\mathcal{O}}  \mathrm{Sel}_{0,m}(\mathbb{Q}, W_{\overline{f}}(r) [I_m]) = 0  $.
\end{itemize}
This implies $\partial^{(\infty)}(\ks^{k-r}) = \partial^{(\infty)}(\mathrm{loc}^s_p\ks^{k-r}) $.

We suppose that  $\partial^{(\infty)}(\ks^{k-r}) < \partial^{(\infty)}(\mathrm{loc}^s_p\ks^{k-r}) $.
Let $n \in \mathcal{N}_1$ satisfying $\mathrm{ord}_\pi(\kappa^{k-r}_n) = \partial^{(\infty)}(\ks^{k-r})$.
By Theorem \ref{thm:structure-fine-selmer-p^k}, we have 
$\mathrm{Sel}_{0,n}(\mathbb{Q}, W_{\overline{f}}(r) [I_n]) = 0$,
but $\mathrm{Sel}_{n}(\mathbb{Q}, W_{\overline{f}}(r) [I_n])$ is non-trivial and cyclic by (\ref{eqn:kappa_n-widedelta_n-sequence})  and (\ref{eqn:kappa_n-widedelta_n}).
This implies that
$\mathrm{Sel}_{n'}(\mathbb{Q}, W_{\overline{f}}(r) [I_{n'}])$ is non-trivial and cyclic for every $n' \in \mathcal{N}_1$ dividing $n$.
By using the core rank zero property of the self-dual Selmer structure \cite[Thm. 4.1.13.(ii)]{mazur-rubin-book},
we have a non-canonical isomorphism of cyclic modules
$$ \mathrm{Sel}_{n}(\mathbb{Q}, W_{\overline{f}}(r) [I_n]) \simeq \mathrm{Sel}_{n}(\mathbb{Q}, T_f(k-r) / I_n T_f(k-r)).$$
By using the Chebotarev density argument (Propositions  \ref{prop:chebotarev-mazur-rubin} and \ref{prop:chebotarev-sakamoto}), 
there exists a prime $\ell \in \mathcal{P}_1$ such that the following restriction maps
\begin{align*}
& \mathrm{loc}^{(k-r)}_\ell : \mathrm{Sel}_n(\mathbb{Q}, T_f(k-r) / \pi T_f(k-r)  ) \simeq \mathrm{H}^1_f(\mathbb{Q}_\ell, T_f(k-r) / \pi T_f(k-r) ), \\
& \mathrm{loc}^{(r)}_\ell : \mathrm{Sel}_n(\mathbb{Q}, W_{\overline{f}}(r)[\pi] ) \simeq \mathrm{H}^1_f(\mathbb{Q}_\ell,  W_{\overline{f}}(r)[\pi] )
\end{align*}
are isomorphisms.
Then Lemma \ref{lem:surjectivity-at-ell} with the surjectivity of $\mathrm{loc}^{(k-r)}_\ell$ implies
$\mathrm{Sel}_{n\ell}(\mathbb{Q}, W_{\overline{f}}(r)[\pi]) = \mathrm{Sel}_{n, \ell\textrm{-str}}(\mathbb{Q}, W_{\overline{f}}(r)[\pi]) $,
and the surjectivity of $\mathrm{loc}^{(r)}_\ell$ implies that $\mathrm{Sel}_{n\ell}(\mathbb{Q}, W_{\overline{f}}(r) [\pi])$ is trivial, so we get contradiction.
\end{proof}

\subsection{Proof of Theorems \ref{thm:main-central-critical} I: the corank part} \label{subsec:proof-theorem-structure-1}
In this subsection, we prove the corank part $\mathrm{cork}_{\mathcal{O}}\left( \mathrm{Sel}(\mathbb{Q}, W^{\dagger}_f) \right)  = \mathrm{ord}(\kn^{\mathrm{min},\dagger} )$ of the structure theorem.

Fix a non-trivial Kolyvagin system 
$\ks^{\dagger} = \left\lbrace \kappa^{\dagger}_n : n \in \mathcal{N}_1 \right\rbrace \in \KS(T^\dagger_f)$.
We use $\ks^{\mathrm{Kato},\dagger}$ only when the connection with $\kn^{\mathrm{min}, \dagger}$ is needed.
\subsubsection{}
We first consider the corank zero case.
\begin{prop} \label{prop:corank-zero}
The following statements are equivalent.
\begin{enumerate}
\item $\widedelta^{\mathrm{min},\dagger}_1 \neq 0$.
\item $\mathrm{Sel}(\mathbb{Q}, W^\dagger_{f})$ is finite.
\end{enumerate}
\end{prop}
\begin{proof}
(1) implies (2): We have
\begin{align*}
& \mathrm{length}_{\mathcal{O}} \mathrm{Sel}(\mathbb{Q}, W^\dagger_{f}) \\
& =\mathrm{length}_{\mathcal{O}} \mathrm{Sel}_0(\mathbb{Q}, W^\dagger_{f}) +\mathrm{length}_{\mathcal{O}} (\mathrm{coker} ( \mathrm{loc}^s_p(\mathrm{Sel}^{\dagger}_{\mathrm{rel}}) )^\vee) & (\ref{eqn:kappa_n-widedelta_n-sequence}) \\
& = \mathrm{ord}_\pi(\kappa^{\dagger}_1) - \partial^{(\infty)}(\ks^{\dagger})
+ \mathrm{length}_{\mathcal{O}} ( \mathrm{coker} ( \mathrm{loc}^s_p(\mathrm{Sel}^{\dagger}_{\mathrm{rel}}) )^\vee ) & \textrm{(Thm. \ref{thm:kato-kolyvagin-main})} \\
& = \mathrm{ord}_\pi(\mathrm{loc}^s_p \kappa^{\dagger}_1) - \partial^{(\infty)}(\ks^{\dagger}) &  (\ref{eqn:kappa_n-widedelta_n})\\
& = \mathrm{ord}_\pi(\mathrm{loc}^s_p \kappa^{\dagger}_1) - \partial^{(\infty)}(\mathrm{loc}^s_p\ks^{\dagger}) & \textrm{(Lem. \ref{lem:vanishing-localization})} \\
& = \mathrm{ord}_\pi(\mathrm{loc}^s_p \kappa^{\mathrm{Kato},\dagger}_1) - \partial^{(\infty)}(\mathrm{loc}^s_p\ks^{\mathrm{Kato},\dagger}) & \textrm{(Thm. \ref{thm:core-rank-modular-forms-char-zero})} \\
& = \mathrm{ord}_\pi(\widedelta^{\mathrm{min},\dagger}_1) - \partial^{(\infty)}(\kn^{\mathrm{min},\dagger}) . & 
\end{align*}
The non-triviality of $\ks^{\mathrm{Kato},\dagger}$ implies the non-triviality of $\kn^{\mathrm{min},\dagger}$ by  Proposition \ref{prop:equivalence-nonvanishing}. Thus, $\partial^{(\infty)}(\kn^{\mathrm{min},\dagger}) < \infty$ and the conclusion holds.

(2) implies (1):
Assume that $\mathrm{Sel}(\mathbb{Q}, W^\dagger_{f})$ is finite.
Then, by (\ref{eqn:kappa_n-widedelta_n-sequence}), we have
$\mathrm{length}_{\mathcal{O}} \left( \mathrm{Sel}_0(\mathbb{Q}, W^\dagger_{f}) \right) < \infty$ and
$\mathrm{length}_{\mathcal{O}} ( \mathrm{coker} ( \mathrm{loc}^s_p(\mathrm{Sel}^{\dagger}_{\mathrm{rel}}) )^\vee ) < \infty$.
Corollary \ref{cor:kato-kolyvagin-main} says that the  finiteness of $\mathrm{Sel}_0(\mathbb{Q}, W^\dagger_{f})$ is equivalent to $\mathrm{ord}_\pi \kappa^{\dagger}_1 < \infty$.
Combining this with the finiteness of $\mathrm{coker} ( \mathrm{loc}^s_p(\mathrm{Sel}^{\dagger}_{\mathrm{rel}}) )^\vee$, (\ref{eqn:kappa_n-widedelta_n}) implies that  $\mathrm{loc}^s_p \kappa^{\dagger}_1 \neq 0$, which is equivalent to $\widedelta^{\mathrm{min},\dagger}_1 \neq 0$.
\end{proof}
\begin{rem} \label{rem:corank-zero}
The proof of Theorem \ref{thm:main-all-critical} is identical with the (1) $\Rightarrow$ (2) direction of Proposition \ref{prop:corank-zero}.
This argument is independent of the self-duality.
\end{rem}
\subsubsection{}
We discuss the higher Selmer corank case, which generalizes Proposition \ref{prop:corank-zero}. 
The self-duality is used in this part.

Fix an integer $m \gg 0$.
Let $n \in \mathcal{N}_m$ satisfying
\begin{itemize}
\item $I_n = \pi^m\mathcal{O}$
\item $\mathrm{ord}(\ks^\dagger) = \mathrm{cork}_{\mathcal{O}} \mathrm{Sel}_0(\mathbb{Q}, W^\dagger_f)  = \nu(n)$, and
\item $\kappa^{\dagger}_n \neq 0$.
\end{itemize}
\begin{rem} \label{rem:choice-of-n-via-kolyvagin}
Due to the second condition, we may assume that our choice of the prime divisors of $n$ \emph{is} the choice of primes in the standard Kolyvagin system argument \cite[Prop. 4.5.8]{mazur-rubin-book}. 
\end{rem}
Theorem \ref{thm:kolyvagin-system-location} says that
$$\langle \kappa^{\dagger}_n \rangle = \pi^j \mathcal{H}'(n) = \pi^{j+ \lambda(n, W^\dagger_f[I_n])}\mathrm{Sel}_{\mathrm{rel}, n}(\mathbb{Q},  T^\dagger_f/I_n)$$ where $\lambda(n, W^\dagger_f[I_n]) = \mathrm{length}_{\mathcal{O}} \mathrm{Sel}_{0,n}(\mathbb{Q}, W^\dagger_f[I_n])$.
Due to Remark \ref{rem:choice-of-n-via-kolyvagin}, we have 
\begin{equation} \label{eqn:sel_0_div}
\mathrm{Sel}_{0,n}(\mathbb{Q}, W^\dagger_f[I_n]) \simeq \mathrm{Sel}_{0}(\mathbb{Q}, W^\dagger_f)_{/\mathrm{div}},
\end{equation}
 so $\lambda(n, W^\dagger_f[I_n])$ is actually bounded independent of the choice of $n$.
 In addition, $j = \partial^{(\infty)}( \ks^{\dagger})$ since $m \gg 0$; thus, $j$ is also independent of the choice of $n$.

With our choice of $n$, we have the surjective map and the isomorphism
\begin{align} \label{eqn:application-chebotarev-structure}
\begin{split}
\mathrm{Sel}(\mathbb{Q}, T^\dagger_{f}/I_n) & \twoheadrightarrow \oplus_{\ell \vert n} \mathrm{H}^1_f(\mathbb{Q}_\ell, T^\dagger_{f}/I_n) , \\
\mathrm{Sel}_{n\textrm{-str}}(\mathbb{Q}, W^\dagger_{f}[I_n]) & \simeq \mathrm{Sel}_{n}(\mathbb{Q}, W^\dagger_{f}[I_n])  ,
\end{split}
\end{align}
and the latter isomorphism follows from the first surjective map and  Lemma \ref{lem:surjectivity-at-ell}.
By using the self-duality
\begin{equation} \label{eqn:self-duality}
T^\dagger_f/I_n \simeq W^\dagger_f[I_n] ,
\end{equation}
(\ref{eqn:application-chebotarev-structure}) becomes
\begin{equation} \label{eqn:chebotarev-density-argument}
\begin{split}
\xymatrix{
0 \ar[r] &\mathrm{Sel}_{n}(\mathbb{Q}, W^\dagger_{f}[I_n]) \ar[r]  &
\mathrm{Sel}(\mathbb{Q}, W^\dagger_{f}[I_n]) \ar[r] & \oplus_{\ell \vert n} \mathrm{H}^1_f(\mathbb{Q}_\ell, W^\dagger_{f}[I_n]) \ar[r] & 0 
}
\end{split}
\end{equation}
Following Proposition \ref{prop:behavior-kappa-n-localization}, $\mathrm{loc}^s_p \kappa^{\dagger}_n = 0$ if and only if $\kappa^{\dagger}_n \in \mathrm{Sel}_n(\mathbb{Q}, T^\dagger_f/I_n)$. We consider two possible cases separately.
\begin{enumerate}
\item $\mathrm{loc}^s_p \kappa^{\dagger}_n \neq 0$.
\item $\mathrm{loc}^s_p \kappa^{\dagger}_n = 0$.
\end{enumerate}

\subsubsection{}
\underline{Suppose that  $\mathrm{loc}^s_p \kappa^{\dagger}_n \neq 0$.}
Then
\begin{align*}
& \mathrm{length}_{\mathcal{O}} ( \mathrm{Sel}_{n}(\mathbb{Q}, W^{\dagger}_f[I_n]) ) \\
& = \mathrm{length}_{\mathcal{O}} ( \mathrm{Sel}_{0,n}(\mathbb{Q}, W^{\dagger}_f[I_n]) )
+ \mathrm{length}_{\mathcal{O}} ( \mathrm{coker} ( \mathrm{loc}^s_p(\mathrm{Sel}^\dagger_{\mathrm{rel},n}) )^\vee ) \\
& =\mathrm{ord}_\pi ( \kappa^{\dagger}_n) - \partial^{(\infty)}( \ks^{\dagger})
+ \mathrm{length}_{\mathcal{O}} ( \mathrm{coker} ( \mathrm{loc}^s_p(\mathrm{Sel}^\dagger_{\mathrm{rel},n}) )^\vee ) \\
& = \mathrm{ord}_\pi (\mathrm{loc}^s_p \kappa^{\dagger}_n) - \partial^{(\infty)}(\mathrm{loc}^s_p \ks^{\dagger}) < m-j .
\end{align*}
This computation with (\ref{eqn:chebotarev-density-argument}) implies that 
$\mathrm{cork}_{\mathcal{O}} \mathrm{Sel}(\mathbb{Q}, W^{\dagger}_f) = \nu(n)$.
Therefore, we have
$$\mathrm{ord} ( \mathrm{loc}^s_p \ks^{\dagger} ) =
\mathrm{ord} (  \ks^{\dagger} ) = \nu(n) =
\mathrm{cork}_{\mathcal{O}} \mathrm{Sel}_{0}(\mathbb{Q}, W^{\dagger}_{f}) =  
\mathrm{cork}_{\mathcal{O}} \mathrm{Sel}(\mathbb{Q}, W^{\dagger}_{f}) .$$
If we replace $\ks^{\dagger}$ by $\ks^{\mathrm{Kato},\dagger}$, then we also have $\widedelta^{\mathrm{min}, \dagger}_n \neq 0$ since $m \gg 0$, so
$\mathrm{ord} ( \mathrm{loc}^s_p \ks^{\mathrm{Kato}, \dagger} ) = \mathrm{ord} ( \kn^{\mathrm{min}, \dagger} ) = \nu(n)$.

\subsubsection{}
\underline{Suppose that  $\mathrm{loc}^s_p \kappa^{\dagger}_n = 0$.}
Thus, we have
$\kappa^{\dagger}_n \in  \mathrm{Sel}_n(\mathbb{Q}, T^\dagger_f/I_n) \simeq \mathrm{Sel}_n(\mathbb{Q}, W^\dagger_f[I_n])$ by using (\ref{eqn:self-duality}).
Theorem \ref{thm:kolyvagin-system-location} and (\ref{eqn:sel_0_div}) imply
\begin{align} \label{eqn:kappa_n-order}
\begin{split}
\mathrm{ord}_\pi \kappa^{\dagger}_n & = m - j + \lambda(n, W^\dagger_f[I_n]) \\
& = m- \partial^{(\infty)}( \ks^{\dagger}) - \mathrm{length}_{\mathcal{O}}  \mathrm{Sel}_{0}(\mathbb{Q}, W^\dagger_f)_{/\mathrm{div}} .
\end{split}
\end{align}
Suppose that $\mathrm{cork}_{\mathcal{O}} \mathrm{Sel}(\mathbb{Q}, W^\dagger_f) = \nu(n)$.
Then (\ref{eqn:chebotarev-density-argument}) implies that
$\mathrm{Sel}_{n}(\mathbb{Q}, W^{\dagger}_f[I_n]) \simeq \mathrm{Sel}(\mathbb{Q}, W^\dagger_f)_{/\mathrm{div}}$,
so its length is bounded independent of the choice of $n$. However, it is impossible because $m$ can be arbitrarily large in (\ref{eqn:kappa_n-order}) and $\kappa^{\dagger}_n \in \mathrm{Sel}_{n}(\mathbb{Q}, W^{\dagger}_f[I_n])$ via (\ref{eqn:self-duality}).
Thus, we have $\mathrm{cork}_{\mathcal{O}} \mathrm{Sel}(\mathbb{Q}, W^\dagger_f) = \nu(n) +1 $ by (\ref{eqn:kappa_n-widedelta_n-sequence}).
Considering (\ref{eqn:chebotarev-density-argument}) again, $\mathrm{Sel}_n(\mathbb{Q}, W^\dagger_f[I_n])$ has rank one over $\mathcal{O}/I_n\mathcal{O}$.

Since $m \gg 0$ and $\mathrm{Sel}_{0,n}(\mathbb{Q}, W^\dagger_f[I_n]) \simeq \mathrm{Sel}_{0}(\mathbb{Q}, W^\dagger_f)_{/\mathrm{div}}$ with our choice of $n$, we may assume that
$\mathrm{length}_{\mathcal{O}} \mathrm{Sel}_{0,n}(\mathbb{Q}, W^\dagger_f[I_n]) < m$.
We also have
$$\mathrm{Sel}_{0,n}(\mathbb{Q}, W^\dagger_f[I_n]) 
\subseteq 
\mathrm{Sel}_{n}(\mathbb{Q}, W^\dagger_f[I_n]) 
\subseteq
\mathrm{Sel}_{\mathrm{rel},n}(\mathbb{Q}, W^\dagger_f[I_n]) \simeq \mathcal{O}/I_n\mathcal{O} \oplus \mathrm{Sel}_{0,n}(\mathbb{Q}, W^\dagger_f[I_n])$$
where the last isomorphism follows from Theorem \ref{thm:splitting-mazur-rubin} and (\ref{eqn:self-duality}),
so the $\mathcal{O}/I_n \mathcal{O}$-component in $\mathrm{Sel}_{\mathrm{rel},n}(\mathbb{Q}, W^\dagger_f[I_n])$
is also contained in $\mathrm{Sel}_{n}(\mathbb{Q}, W^\dagger_f[I_n])$.

By using the Chebotarev density argument (Propositions  \ref{prop:chebotarev-mazur-rubin} and \ref{prop:chebotarev-sakamoto}), we choose a useful prime $\ell \in \mathcal{P}_m$ for $\kappa^\dagger_n$ such that
\begin{itemize}
\item $I_{n\ell} = I_n = \pi^m$,
\item $\kappa^\dagger_{n\ell} \neq 0$, and
\item $\mathrm{loc}_{\ell} : \mathrm{Sel}_{n}(\mathbb{Q}, T^\dagger_f/I_n) \to \mathrm{H}^1_f(\mathbb{Q}_\ell, T^\dagger_f/I_n)$ is surjective. 
%More precisely the $\mathcal{O}/I_n\mathcal{O}$-component in $\mathrm{Sel}_{n}(\mathbb{Q}, W^\dagger_f[I_n])$ maps to $\mathrm{H}^1_f(\mathbb{Q}_\ell, W^\dagger_f[I_n])$ isomorphically under $\mathrm{loc}_{\ell}$.
\end{itemize}
The last condition implies that
$\mathrm{Sel}_{n\ell }(\mathbb{Q}, W^\dagger_f[I_n]) = \mathrm{Sel}_{n, \ell\textrm{-str}}(\mathbb{Q}, W^\dagger_f[I_n])$ by Lemma \ref{lem:surjectivity-at-ell}.
By using (\ref{eqn:self-duality}) again,  $\mathrm{Sel}_{n\ell }(\mathbb{Q}, W^\dagger_f[I_n])$ has rank zero over $\mathcal{O}/I_n\mathcal{O}$,.
Since $m \gg 0$, we have
$\mathrm{Sel}_{n\ell }(\mathbb{Q}, W^\dagger_f[I_n]) \subseteq \mathrm{Sel}(\mathbb{Q}, W^\dagger_f)_{/\mathrm{div}}$.
Here, we regard $\mathrm{Sel}(\mathbb{Q}, W^\dagger_f)_{/\mathrm{div}}$ as a submodule of $\mathrm{Sel}(\mathbb{Q}, W^\dagger_f)$.
This shows that 
$$\kappa^\dagger_{n\ell}  \not\in \mathrm{Sel}_{n\ell }(\mathbb{Q}, T^\dagger_f/I_n) \simeq \mathrm{Sel}_{n\ell }(\mathbb{Q}, W^\dagger_f[I_n]),$$
so
$\mathrm{loc}^s_p \kappa^\dagger_{n\ell} \neq 0.$
Therefore, we have
$$\mathrm{ord} ( \mathrm{loc}^s_p \ks^{\dagger} ) =
\mathrm{ord} (  \ks^{\dagger} ) +1 = \nu(n) +1 =
\mathrm{cork}_{\mathcal{O}} \mathrm{Sel}_{0}(\mathbb{Q}, W^{\dagger}_{f}) +1 =  
\mathrm{cork}_{\mathcal{O}} \mathrm{Sel}(\mathbb{Q}, W^{\dagger}_{f}) .$$
If we replace $\ks^{\dagger}$ by $\ks^{\mathrm{Kato},\dagger}$, then we also have $\widedelta^{\mathrm{min}, \dagger}_{n\ell} \neq 0$ since $m \gg 0$, so
$\mathrm{ord} ( \mathrm{loc}^s_p \ks^{\mathrm{Kato}, \dagger} ) = \mathrm{ord} ( \kn^{\mathrm{min}, \dagger} ) = \nu(n)+1$.

\subsection{Proof of Theorems \ref{thm:main-central-critical}  II: the structure part} \label{subsec:proof-theorem-structure-2}
In this subsection, we describe the structure of $\mathrm{Sel}(\mathbb{Q}, W^{\dagger}_f)_{/\mathrm{div}}$ in terms of $\kn^{\mathrm{min}, \dagger}$.

Let $m \gg 0$ be an integer.
We fix $n \in \mathcal{N}_m$ such that
\begin{itemize}
\item $\mathcal{O}/I_n\mathcal{O} = \pi^m \mathcal{O}$
\item $\mathrm{loc}^s_p \kappa^{\dagger}_n \neq 0$,
\item $\nu(n)  = r = \mathrm{cork}_{\mathcal{O}}\mathrm{Sel}(\mathbb{Q}, W^\dagger_f)$, and
\item $\mathrm{length}_{\mathcal{O}} \mathrm{Sel}_n(\mathbb{Q}, W^\dagger_f[I_n]) < m = \mathrm{length}_{\mathcal{O}} (\mathcal{O}/I_n\mathcal{O})$
\end{itemize}
We investigate the structure of $\mathrm{Sel}_n(\mathbb{Q}, W^{\dagger}_{f}[I_n])$.
From now on, we explicitly use Kato's Kolyvagin system $\ks^{\mathrm{Kato}, \dagger}$ in order to invoke the functional equation for $\kn^{\mathrm{min}, \dagger}$ (Proposition \ref{prop:vanishing-delta-n}).
\subsubsection{} \label{subsubsec:determine-str-n}
By using Flach's generalized Cassels--Tate pairing \cite{flach-cassels-tate, howard-kolyvagin}, 
we have
\begin{equation} \label{eqn:cassels-tate-pairing-M_n}
\mathrm{Sel}_n(\mathbb{Q}, W^\dagger_{f}[I_n]) \simeq M_n \oplus M_n
\end{equation}
 for some finite abelian $p$-group $M_n$.
We first recall (\ref{eqn:kappa_n-widedelta_n-sequence}) for $r = k/2$
\[
\xymatrix{
\mathrm{Sel}_{0,n}(\mathbb{Q}, W^\dagger_{f}[I_n])
\ar@{^{(}->}[r] &
\mathrm{Sel}_{n}(\mathbb{Q}, W^\dagger_{f}[I_n])
\ar@{->>}[r] &
\mathrm{coker} ( \mathrm{loc}^s_p(\mathrm{Sel}^{k-r}_{\mathrm{rel},n}) )^\vee .
}
\]
This sequence determines the structure of $\mathrm{Sel}_{n}(\mathbb{Q}, W^\dagger_{f}[I_n])$
 uniquely 
thanks to the structure theorem for $\mathrm{Sel}_{0,n}(\mathbb{Q}, W^\dagger_{f}[I_n])$ (Theorem \ref{thm:structure-fine-selmer-p^k}),
(\ref{eqn:cassels-tate-pairing-M_n}), and the cyclicity of $\mathrm{coker} ( \mathrm{loc}^s_p(\mathrm{Sel}^{k-r}_{\mathrm{rel},n}) )^\vee$.

\subsubsection{}
For convenience, we write
$$\mathrm{Sel}_{0,n}(\mathbb{Q}, W^\dagger_{f}[I_n]) = \mathrm{Sel}_{0,n}(\mathbb{Q}, W^\dagger_{f}[\pi^m]) \simeq \bigoplus_{i \geq 1}  \mathcal{O} / \pi^{d_i} \mathcal{O}$$
with non-negative integers $m > d_1 \geq d_2 \geq \cdots \geq 0$
following Theorem \ref{thm:structure-fine-selmer-p^k}.

By using the Chebotarev density argument (Propositions  \ref{prop:chebotarev-mazur-rubin} and \ref{prop:chebotarev-sakamoto}), choose  $\ell_1 \in \mathcal{P}_m$ such that
\begin{itemize}
\item $I_{n\ell_1} \mathcal{O} = \pi^m \mathcal{O}$,
\item $\kappa^{\mathrm{Kato}, \dagger}_{n\ell_1} \neq 0$, i.e. useful for $\kappa^{\mathrm{Kato}, \dagger}_{n}$, 
\item the restriction map $\mathrm{loc}_{\ell_1} : \pi^{m-1}\mathrm{Sel}_{\mathrm{rel}, n}(\mathbb{Q}, T^\dagger_f/I_n) \to \mathrm{H}^1_f(\mathbb{Q}_{\ell_1}, T^\dagger_f/I_n)$ is non-zero, and
\item the restriction map $\mathrm{loc}_{\ell_1} : \pi^{d_1 - 1}\mathrm{Sel}_{0, n}(\mathbb{Q}, W^\dagger_f[I_n]) \to \mathrm{H}^1_f(\mathbb{Q}_{\ell_1}, W^\dagger_f[I_n])$ is non-zero.
\end{itemize}
Since this choice of $\ell_1$ is compatible with the Kolyvagin system argument \cite[Prop. 4.5.8]{mazur-rubin-book}, we have
$$\mathrm{Sel}_{0, n\ell_1}(\mathbb{Q}, W^\dagger_f[\pi^m])  \simeq \bigoplus_{i \geq 2} \mathcal{O} / \pi^{d_i} \mathcal{O}. $$
The third condition implies that 
$$\mathrm{loc}_{\ell_1} : \mathrm{Sel}_{\mathrm{rel}, n}(\mathbb{Q}, T^\dagger_f/I_n) \to \mathrm{H}^1_f(\mathbb{Q}_{\ell_1}, T^\dagger_f/I_n)$$
is surjective, so we have
$\mathrm{Sel}_{\mathrm{rel}, n, \ell_1\textrm{-}\mathrm{str}}(\mathbb{Q}, T^\dagger_f/I_n) = \mathrm{Sel}_{\mathrm{rel}, n\ell_1}(\mathbb{Q}, T^\dagger_f/I_n)$
thanks to Lemma \ref{lem:surjectivity-at-ell}.
\subsubsection{}
Consider inclusions
\begin{align*}
\mathrm{Sel}_{0, n \ell_1} (\mathbb{Q}, W^\dagger_{f}[I_n])
\subseteq
\mathrm{Sel}_{n \ell_1} (\mathbb{Q}, W^\dagger_{f}[I_n])
& \subseteq
\mathrm{Sel}_{\mathrm{rel}, n \ell_1} (\mathbb{Q}, W^\dagger_{f}[I_n]) \\
& \simeq 
\mathrm{Sel}_{0, n \ell_1} (\mathbb{Q}, W^\dagger_{f}[I_n]) \oplus \mathcal{O}/I_{n}\mathcal{O}
\end{align*}
where the last isomorphism follows from Theorem \ref{thm:splitting-mazur-rubin} with the self-duality.

We explicitly characterize the ``free component of $\mathrm{Sel}_{n \ell_1} (\mathbb{Q}, W^\dagger_{f}[I_n])$".
We have $\widedelta^{\mathrm{min},\dagger}_n \neq 0$ since $\mathrm{loc}^s_p(\kappa^{\mathrm{Kato}, \dagger}_n) \neq 0$ and $m \gg 0$.
However, $\widedelta^{\mathrm{min},\dagger}_{n\ell_1} = 0$ due to Proposition \ref{prop:vanishing-delta-n}.
Thus, we have
\begin{align} \label{eqn:kappa_ell_1_local_condition}
\begin{split}
\kappa^{\mathrm{Kato}, \dagger}_{n\ell_1} & \in \mathrm{Sel}_{n\ell_1}(\mathbb{Q}, T^\dagger_f/I_{n\ell_1}) \\
& \simeq \mathrm{Sel}_{n\ell_1}(\mathbb{Q}, W^\dagger_f[I_{n\ell_1}]) \\
& = \mathrm{Sel}_{n\ell_1}(\mathbb{Q}, W^\dagger_f[\pi^m]) .
\end{split}
\end{align}
Theorem \ref{thm:kolyvagin-system-location} implies that
\begin{align*}
\langle \kappa^{\mathrm{Kato}, \dagger}_{n\ell_1} \rangle & = \pi^{\lambda(n\ell_1, W^\dagger_f[I_{n\ell_1}]) + j} \mathrm{Sel}_{\mathrm{rel}, n\ell_1}(\mathbb{Q}, T^\dagger_f / I_{n\ell_1} ) \\
& \simeq \mathcal{O} / \pi^{m - \lambda(n\ell_1, W^\dagger_f[I_{n\ell_1}]) - j } \mathcal{O} .
\end{align*}
\subsubsection{}
Write 
$$m_2 = m - \lambda(n\ell_1, W^\dagger_f[I_{n\ell_1}]) - j$$
for convenience.
Then we have
\begin{align*}
\begin{split}
 \langle \kappa^{\mathrm{Kato}, \dagger}_{n\ell_1} \rangle & =  \pi^{\lambda(n\ell_1, E[I_{n\ell_1}]) + j} \mathrm{Sel}_{\mathrm{rel}, n\ell_1}(\mathbb{Q}, T^\dagger_f / I_{n\ell_1} )  \\
&\simeq \pi^{\lambda(n\ell_1, W^\dagger_f[I_{n\ell_1}]) + j} \mathrm{Sel}_{\mathrm{rel}, n\ell_1}(\mathbb{Q}, W^\dagger_f[I_{n\ell_1}])  \\
& \subseteq \mathrm{Sel}_{\mathrm{rel}, n\ell_1}(\mathbb{Q}, W^\dagger_f[I_{n\ell_1}]) [\pi^{m - \lambda(n\ell_1, E[I_{n\ell_1}]) + j}] \\
& = \mathrm{Sel}_{\mathrm{rel}, n\ell_1}(\mathbb{Q}, W^\dagger_f[\pi^{m_2}]) .
\end{split}
\end{align*}
By (\ref{eqn:kappa_ell_1_local_condition}), we know $ \left\langle \kappa^{\mathrm{Kato}, \dagger}_{n\ell_1} \right\rangle \subseteq \mathrm{Sel}_{n\ell_1}(\mathbb{Q}, W^\dagger_f[\pi^{m_2}]) $.
Since $m$ is sufficiently large and $\lambda(n\ell_1, W^\dagger_f[I_{n\ell_1}]) + j$ is bounded independent of the choice of $n \ell_1$,
 we may assume that $d_2 < m_2 = m - \lambda(n\ell_1, W^\dagger_f[I_{n\ell_1}]) - j$.
In particular, $\mathrm{Sel}_{n\ell_1}(\mathbb{Q}, W^\dagger_f[\pi^{m_2}])$ has rank one over $\mathcal{O}/\pi^{m_2}\mathcal{O}$, and the rank one component is generated by $\kappa^{\mathrm{Kato}, \dagger}_{n\ell_1}$.

\subsubsection{}
By using the Chebotarev density argument (Propositions  \ref{prop:chebotarev-mazur-rubin} and \ref{prop:chebotarev-sakamoto}) again, we are able to choose a useful prime $\ell_2 \in \mathcal{P}_{m_2}$ for $\kappa^{\mathrm{Kato}, \dagger}_{n\ell_1}$ such that
\begin{itemize}
\item  $I_{n\ell_1 \ell_2}  = \pi^{m_2} \mathcal{O}$,
\item $\kappa^{\mathrm{Kato}, \dagger}_{n\ell_1 \ell_2} \neq 0$, i.e. useful for $\kappa^{\mathrm{Kato}, \dagger}_{n\ell_1}$, 
\item $\mathrm{loc}_{\ell_2} : \pi^{m_2-1} \mathrm{Sel}_{\mathrm{rel}, n\ell_1}(\mathbb{Q}, T^\dagger_f/I_{n\ell_1\ell_2}) \to \mathrm{H}^1_f( \mathbb{Q}_{\ell_2}, T^\dagger_f/I_{n\ell_1\ell_2})$ is non-zero, and
\item $\mathrm{loc}_{\ell_2} : \pi^{d_2 - 1}\mathrm{Sel}_{0, n\ell_1}(\mathbb{Q}, W^\dagger_f[\pi^{m_2}]) \to \mathrm{H}^1_f( \mathbb{Q}_{\ell_2}, W^\dagger_f[\pi^{m_2}])$ is non-zero.
\end{itemize}
Since the $\mathcal{O}/\pi^{m_2}\mathcal{O}$-component of $\mathrm{Sel}_{\mathrm{rel}, n\ell_1}(\mathbb{Q}, T^\dagger_f/\pi^{m_2})$
lies in $\mathrm{Sel}_{n\ell_1}(\mathbb{Q}, W^\dagger_f[\pi^{m_2}])$,
the third condition on $\ell_2$ also implies that we have exact sequence
\[
\xymatrix{
0 \ar[r] & \mathrm{Sel}_{n\ell_1, \ell_2\textrm{-}\mathrm{str}}(\mathbb{Q}, W^\dagger_f[\pi^{m_2}]) \ar[r] & \mathrm{Sel}_{n\ell_1}(\mathbb{Q}, W^\dagger_f[\pi^{m_2}]) \ar[r]^-{\mathrm{loc}_{\ell_2}} & \mathrm{H}^1_f(\mathbb{Q}_{\ell_2}, W^\dagger_f[\pi^{m_2}]) \ar[r] & 0 
}
\]
and isomorphism
 $$\mathrm{Sel}_{n\ell_1, \ell_2\textrm{-}\mathrm{str}}(\mathbb{Q}, W^\dagger_f[\pi^{m_2}]) \simeq \mathrm{Sel}_{n\ell_1 \ell_2}(\mathbb{Q}, W^\dagger_f[\pi^{m_2}]) $$ by Lemma \ref{lem:surjectivity-at-ell}.
In particular, $\mathrm{Sel}_{n \ell_1 \ell_2} (\mathbb{Q}, W^\dagger_{f}[\pi^{m_2}])$ has rank zero over $\mathcal{O}/\pi^{m_2}\mathcal{O}$.
By applying the generalized Cassels--Tate pairing again, we have
\begin{equation} \label{eqn:cassels-tate-pairing-M_nell1ell2} 
\mathrm{Sel}_{n\ell_1 \ell_2}(\mathbb{Q}, W^\dagger_f[\pi^{m_2}]) \simeq M_{n\ell_1 \ell_2}  \oplus M_{n\ell_1 \ell_2}
\end{equation}
for an $\mathcal{O}$-module $M_{n\ell_1 \ell_2}$  with finite cardinality.

As in $\S$\ref{subsubsec:determine-str-n}, the structure of $\mathrm{Sel}_{n\ell_1 \ell_2}(\mathbb{Q}, W^\dagger_f[\pi^{m_2}])$ is also determined uniquely since
we have exact sequence
\[
\xymatrix{
\mathrm{Sel}_{0,n\ell_1\ell_2}(\mathbb{Q}, W^\dagger_{f}[I_n])
\ar@{^{(}->}[r] &
\mathrm{Sel}_{n\ell_1\ell_2}(\mathbb{Q}, W^\dagger_{f}[I_n])
\ar@{->>}[r] &
\mathrm{coker} ( \mathrm{loc}^s_p(\mathrm{Sel}^{\dagger}_{\mathrm{rel},\ell_1\ell_2}) )^\vee ,
}
\]
the structure theorem for $\mathrm{Sel}_{0,n\ell_1\ell_2}(\mathbb{Q}, W^\dagger_{f}[\pi^{m_2}])$ (Theorem \ref{thm:structure-fine-selmer-p^k}), (\ref{eqn:cassels-tate-pairing-M_nell1ell2}), and the cyclicity of $\mathrm{coker} ( \mathrm{loc}^s_p(\mathrm{Sel}^{\dagger}_{\mathrm{rel},n\ell_1\ell_2}) )^\vee$.
\subsubsection{}
We now compare
$\mathrm{Sel}_{n\ell_1\ell_2}(\mathbb{Q}, W^\dagger_{f}[\pi^{m_2}]) \simeq M_{n\ell_1 \ell_2}  \oplus M_{n\ell_1 \ell_2}$
with
$\mathrm{Sel}_{n}(\mathbb{Q}, W^\dagger_{f}[\pi^{m}]) \simeq M_{n}  \oplus M_{n}$.
Write $\mathrm{Sel}_{0,n}(\mathbb{Q}, W^\dagger_{f}[\pi^{m}]) \simeq M_{n}  \oplus M'_{n}$
with $M'_n \subseteq M_n$.
Then $M_n/M'_n $ is cyclic since it is isomorphic to $\mathrm{coker} ( \mathrm{loc}^s_p(\mathrm{Sel}^{\dagger}_{\mathrm{rel},n}) )^\vee$.
In the same manner, we also write
 $\mathrm{Sel}_{0,n\ell_1\ell_2}(\mathbb{Q}, W^\dagger_{f}[\pi^{m_2}]) \simeq M_{n\ell_1\ell_2}  \oplus M'_{n\ell_1\ell_2}$
with $M'_{n\ell_1\ell_2} \subseteq M_{n\ell_1\ell_2}$.
\begin{itemize}
\item Suppose that $M_n/M'_n $ is not isomorphic to the largest cyclic factor of $M_n$.
Then, by the Kolyvagin system argument for $\mathrm{Sel}_{0,n}(\mathbb{Q}, W^\dagger_{f}[\pi^{m}]) \simeq M_{n}  \oplus M'_{n}$, we have that $M_{n} /M_{n\ell_1\ell_2}$ is isomorphic to the largest cyclic factor of $M_n$.
\item Suppose that $M_n/M'_n $ is isomorphic to the largest cyclic factor of $M_n$.
Then we have
 $\mathrm{Sel}_{0,n\ell_1\ell_2}(\mathbb{Q}, W^\dagger_{f}[\pi^{m_2}]) \simeq M_{n\ell_1\ell_2}  \oplus M'_{n\ell_1\ell_2} \hookrightarrow M'_{n} \oplus M'_{n}$ with cyclic cokernel.
This implies that 
 $\mathrm{Sel}_{n\ell_1\ell_2}(\mathbb{Q}, W^\dagger_{f}[\pi^{m_2}]) \simeq M'_{n}  \oplus M'_{n}$.
\end{itemize}
In any case,  $\mathrm{Sel}_{n\ell_1\ell_2}(\mathbb{Q}, W^\dagger_{f}[\pi^{m_2}])$
is (non-canonically) isomorphic to 
the quotient of  $\mathrm{Sel}_{n}(\mathbb{Q}, W^\dagger_{f}[\pi^{m}])$ by the two copies of its largest factor.
\subsubsection{}
Putting it all together, we have
\begin{align*}
& \mathrm{length}_{\mathcal{O}} \mathrm{Sel}_n(\mathbb{Q}, W^\dagger_f[I_n]) -  \mathrm{length}_{\mathcal{O}} \mathrm{Sel}_{n\ell_1 \ell_2}(\mathbb{Q}, W^\dagger_f[I_{n\ell_1 \ell_2}]) \\
& = 
\left( \mathrm{length}_{\mathcal{O}} \mathrm{Sel}_{0,n}(\mathbb{Q}, W^\dagger_f[I_n]) + \mathrm{length}_{\mathcal{O}} \mathrm{coker} ( \mathrm{loc}^s_p(\mathrm{Sel}^\dagger_{\mathrm{rel},n}) )^\vee \right) \\
& \ \ \ \ -
\left(
 \mathrm{length}_{\mathcal{O}} \mathrm{Sel}_{0,n\ell_1 \ell_2}(\mathbb{Q}, W^\dagger_f[I_{n\ell_1 \ell_2}])  + 
 \mathrm{length}_{\mathcal{O}} \mathrm{coker} ( \mathrm{loc}^s_p(\mathrm{Sel}^\dagger_{\mathrm{rel},n\ell_1\ell_2}) )^\vee \right) \\
& = 
\left( \mathrm{ord}_\pi ( \kappa^{\mathrm{Kato}, \dagger}_{n} )  + \mathrm{length}_{\mathcal{O}} \mathrm{coker} ( \mathrm{loc}^s_p(\mathrm{Sel}^\dagger_{\mathrm{rel},n}) )^\vee \right) \\
& \ \ \ \ -
\left(
  \mathrm{ord}_\pi ( \kappa^{\mathrm{Kato}, \dagger}_{n\ell_1\ell_2} )  + 
 \mathrm{length}_{\mathcal{O}} \mathrm{coker} ( \mathrm{loc}^s_p(\mathrm{Sel}^\dagger_{\mathrm{rel},n\ell_1\ell_2}) )^\vee \right) \\
& =  \mathrm{ord}_\pi ( \mathrm{loc}^s_p(\kappa^{\mathrm{Kato}, \dagger}_{n}) )
 - \mathrm{ord}_\pi ( \mathrm{loc}^s_p(\kappa^{\mathrm{Kato}, \dagger}_{n\ell_1\ell_2}) ) \\
& =  \mathrm{ord}_\pi ( \widedelta^{\mathrm{min}, \dagger}_{n} ) - \mathrm{ord}_\pi ( \widedelta^{\mathrm{min}, \dagger}_{n\ell_1\ell_2} ) .
\end{align*}
By replacing $n \ell_1 \ell_2$ by $n$, we repeat the same process until we get $\mathrm{Sel}_n(\mathbb{Q}, W^\dagger_f[I_n]) =0$.
Thus, the structure part of Theorem \ref{thm:main-central-critical} is complete.

\section{The rank one $p$-converse to the theorem of Nekov\'{a}\v{r} and S.-W. Zhang} \label{sec:p-converse}
In this section, we keep the notation of $\S$\ref{subsubsec:p-converse-intro}.
We first recall the theorem of Nekov\'{a}\v{r} and S.-W. Zhang, which is a higher weight generalization of Gross--Zagier and Kolyvagin.
\begin{thm}[Nekov\'{a}\v{r}, Zhang] \label{thm:nekovar-zhang}
Let $f \in S_k(\Gamma_0(N))$ be a newform with $k \geq 4$ and $p \geq 3$ a prime such that $p$ does not divide $N$. 
Let $\mathcal{K}$ be an imaginary quadratic field of discriminant $-D_{\mathcal{K}} <0$ such that every prime divisor of $N$ splits in $\mathcal{K}$, and $ L(f, \chi_{\mathcal{K}/\mathbb{Q}}, k/2) \neq 0$.
 If $\Phi_{\mathcal{K}} \otimes \mathbb{Q}_p$ is injective on $\mathrm{CM}_{k/2}(X(N)_{\mathcal{K}})$ (Assumption \ref{assu:abel-jacobi-injectivity}),
then
$$\mathrm{ord}_{s=k/2} L(f, s) = 1 \Rightarrow \mathrm{cork}_{\mathcal{O}}\mathrm{Sel}(\mathbb{Q}, W^\dagger_f) = 1 \textrm{ and } \mathrm{length}_{\mathcal{O}}\sha(f^\dagger/\mathbb{Q})[\pi^\infty] < \infty.$$
\end{thm}

The following theorem provides a partial converse to Theorem \ref{thm:nekovar-zhang}.
This generalizes the rank one $p$-converse results of Skinner and the author \cite{skinner-converse, kim-p-converse} to higher weight modular forms.
See \cite{haining-heegner-cycles, longo-vigni-tamagawa} for Heegner cycle approaches.

One difference from the theorem of Nekov\'{a}\v{r}--Zhang (Theorem \ref{thm:nekovar-zhang}) is that Theorem \ref{thm:p-converse-higher-weight} is independent of the injectivity of the $p$-adic Abel--Jacobi map (Assumption \ref{assu:abel-jacobi-injectivity}.(AJ)).
Note that (IMC at $\mathbf{1}$) and (cork1) imply $w(f) = -1$ by Corollary \ref{cor:discrete-bsd}.
As we mentioned in Remark \ref{rem:abel-jacobi-injectivity},  the finiteness of the $\pi$-primary part of Tate--Shafarevich groups implies the injectivity of the restriction map when $k = 2$. Although it is unclear for us whether the same implication is an easy statement for higher weight modular forms, we still expect that the restriction map is always injective under this setting.

\begin{thm} \label{thm:p-converse-higher-weight}
Let $f \in S_k(\Gamma_0(N))$ be a newform with $k \geq 4$ and $p \geq 3$ a prime such that
 $\rho_f$ has large image,
  $p^2 \nmid N$, and
 $\alpha_p \neq \beta_p$ (when $p \nmid N$).
If
\begin{itemize}
\item[(IMC at $\mathbf{1}$)]   the Iwasawa main conjecture for $T^\dagger_f$ at the augmentation ideal holds,
\item[(cork1)] $\mathrm{cork}_{\mathcal{O}}\mathrm{Sel}(\mathbb{Q}, W^\dagger_f) = 1$, and 
 \item[({$\mathrm{loc}_p$})] the restriction map
$\mathrm{loc}_p : \mathrm{Sel}(\mathbb{Q}, V^\dagger_f) \to \mathrm{H}^1(\mathbb{Q}_p, V^\dagger_f)$ is injective (Assumption \ref{assu:abel-jacobi-injectivity}), 
\end{itemize}
then
$\kappa^{\mathrm{Hg},\dagger,  \mathcal{K}}_1 \neq 0 $
for every imaginary quadratic field $\mathcal{K}$ such that every prime divisor of $Np$ splits in $\mathcal{K}$,  and  $L(f, \chi_{\mathcal{K}/\mathbb{Q}}, k/2) \neq 0$.
For such an imaginary quadratic field $\mathcal{K}$, we have
\begin{align*}
\mathrm{cork}_{\mathcal{O}} \left( \mathrm{Sel}(\mathbb{Q}, W^\dagger_f) \right) & = \mathrm{dim}_{F} \left( \mathrm{im}(\Phi_{f, \mathbb{Q}} \otimes F) \right) = \mathrm{dim}_{F} \left( \mathrm{im}(\Phi_{f, \mathcal{K}} \otimes F)^{+} \right) = 1,\\
 \mathrm{dim}_{F} \left( \mathrm{im}(\Phi_{f, \mathcal{K}} \otimes F)^{-} \right) & = 0, \textrm{and}\\
\mathrm{length}_{\mathcal{O}} \sha(f^\dagger/\mathbb{Q})[\pi^\infty] & < \infty .
\end{align*}
If we further 
assume 
\begin{itemize}
\item[(non-deg)] the height pairing\footnote{Here we use the notion of the height pairing defined by Gillet--Soul\'{e} via arithmetic intersection theory following \cite{shou-wu-zhang-heegner-cycles}. Its non-degeneracy is conjectured by Beilinson \cite[Conj. 5.4]{beilinson-height-pairings} and Bloch \cite[Conj.]{bloch-height-pairings}, independently.} is non-degenerate on $\mathrm{CM}_{k/2}(X(N)_{\mathcal{K}})$ for the same $\mathcal{K}$,
\end{itemize}
then the rank one case of the Beilinson--Bloch--Kato conjecture for the motive associated to $f$ holds, i.e.
$$\mathrm{cork}_{\mathcal{O}}\mathrm{Sel}(\mathbb{Q}, W^\dagger_f) = \mathrm{dim}_{F} \left( \mathrm{im}(\Phi_{f, \mathbb{Q}} \otimes F) \right) = \mathrm{ord}_{s=k/2}L(f, s) = 1 .$$
\end{thm}
\begin{proof}
We apply the same strategy of \cite{kim-p-converse} to the higher weight setting.
We have
\begin{align*}
\mathrm{cork}_{\mathcal{O}} \mathrm{Sel}(\mathbb{Q}, W^\dagger_f) = 1  & \Rightarrow \mathrm{cork}_{\mathcal{O}} \mathrm{Sel}_0(\mathbb{Q}, W^\dagger_f) = 0  & \textrm{({$\mathrm{loc}_p$})} \\
&\Leftrightarrow \kappa^{\mathrm{Kato}, \dagger}_1 \neq 0 & \textrm{(Thm. \ref{thm:mazur-rubin-main-conjecture})} \\
&\Rightarrow \mathrm{loc}_p \left( \kappa^{\mathrm{Kato}, \dagger}_1 \right) \neq 0 & \textrm{({$\mathrm{loc}_p$})} \\
&\Rightarrow \mathrm{loc}_p \left( \kappa^{\mathrm{Hg},\dagger,  \mathcal{K}}_1 \right) \neq 0 & \textrm{(PR)} \\
&\Rightarrow  \kappa^{\mathrm{Hg},\dagger,  \mathcal{K}}_1 \neq 0 & \\
&\Rightarrow  \mathrm{length}_{\mathcal{O}}\sha(f^\dagger/\mathbb{Q})[\pi^\infty] < \infty . & \textrm{(Thm.  \ref{thm:nekovar-zhang})}
\end{align*}
and
\begin{align*}
\kappa^{\mathrm{Hg},\dagger,  \mathcal{K}}_1 \neq 0  & \Rightarrow s_{\mathcal{K}} \neq 0 \\
& \Rightarrow \mathrm{ord}_{s=k/2}L(f,s) =1 . & \textrm{(non-deg)}  
\end{align*}
where (PR) means the settlement of Perrin-Riou's conjecture \cite[Thm. B]{bertolini-darmon-venerucci} with \cite[Thm. 14.2]{kato-euler-systems}.

\end{proof}

\section{The quantitative refinement of the non-vanishing conjecture} \label{sec:formulation-refined-conjecture}
%We focus on the self-dual case mainly for notational convenience.
Although we completely determine the structure of Bloch--Kato Selmer groups via $\kn^{\mathrm{min},\dagger}$ in Theorem \ref{thm:main-central-critical}, we do not have any precise control of $\partial^{(\infty)} \left( \kn^{\mathrm{min},\dagger} \right) $ mainly because we could not compute the integral image of the dual exponential map.

Considering the compatibility with the classical Tamagawa number conjecture of Bloch--Kato \cite[Conj. 5.15 and (5.15.1)]{bloch-kato} , it is natural to ask the relation between 
$\partial^{(\infty)} \left( \kn^{\mathrm{min},\dagger} \right) $ and the $\pi$-valuation of the product of the local Tamagawa ideals of $T^\dagger_f$.
This is why we call $\partial^{(\infty)} \left( \kn^{\mathrm{min},\dagger} \right) $ in $\S$\ref{subsubsec:numerical-invariants-kurihara-numbers} the analytic \emph{fudge factor} \cite{rubin-fudge-factors}.

The goal of this section is to formulate the precise conjectures on this question.
In some sense,  our main results can be viewed as the reduction of Kato--Fontaine--Perrin-Riou's formulation of Tamagawa number conjecture for modular motives to the arithmetic interpretation of  $\partial^{(\infty)} \left( \kn^{\mathrm{min},\dagger} \right) $.
Our conjectures also indicate that the information of local Tamagawa ideals cannot be obtained from the Kolyvagin system argument solely.
In order to prove these conjectures,  it is expected that both the Iwasawa main conjecture and integral $p$-adic Hodge theory are needed.

\subsection{The good ordinary or Fontaine--Laffaille case}
When Hida theory or Fontaine--Laffaille theory is applicable, we are able to control the local Tamagawa number at $p$ and have the notion of canonical periods \cite{vatsal-cong}. In this case, it seems reasonable to formulate the following conjecture, which claims equality between two completely different ``fudge factors".
\begin{conj} \label{conj:kn-tamagawa}
Let $f \in S_k(\Gamma_0(N))$ be a newform with $k \geq 2$ and $p \geq 3$ such that $p \nmid N$ and $\rho_f$ has large image.
If $f$ is ordinary at $p$ and $p$-distinguished or $2 \leq k \leq p-1$, then
\begin{equation} \label{eqn:canonical-periods-tamagawa-numbers}
\partial^{(\infty)}(\kn^{\mathrm{can}, \dagger}) \overset{?}{=} \sum_{\ell \vert N} \mathrm{length}_{\mathcal{O}} \left( \dfrac{\mathrm{H}^1_{\mathrm{ur}}(\mathbb{Q}_\ell, W^\dagger_f ) }{\mathrm{H}^1_{f}(\mathbb{Q}_\ell, W^\dagger_f ) }  \right) 
\end{equation}
where 
$\kn^{\mathrm{can},\dagger}$ is the collection of Kurihara numbers for $f$ at $s = k/2$ normalized by the canonical periods.
\end{conj}
When $L(f, k/2) \neq 0$, Conjecture \ref{conj:kn-tamagawa} is equivalent to the classical Tamagawa number conjecture and it follows from the Iwasawa main conjecture for $T^\dagger_f$ \cite{skinner-urban,wan-iwasawa-2017}.
In this sense,  Conjecture \ref{conj:kn-tamagawa} is a refinement of the Tamagawa number conjecture for the central critical twist of $f$. 
Also, Conjecture \ref{conj:kn-tamagawa} is verified numerically for several elliptic curves (of positive rank) in \cite[$\S$8]{kks}.
For elliptic curves with good ordinary reduction,  Conjecture \ref{conj:kn-tamagawa} is confirmed recently \cite{burungale-castella-grossi-skinner-indivisibility}.
See \cite[Conj. 4.5]{wei-zhang-cdm} for the Heegner point version.
\subsection{The general case}
We begin with the following natural question.
\begin{ques} \label{ques:kn-tamagawa}
Without the good ordinary or Fontaine--Laffaille type assumption, what is the correct generalization of Conjecture \ref{conj:kn-tamagawa} including the notion of canonical periods? For example,  can the canonical periods be characterized as the periods satisfying equality (\ref{eqn:canonical-periods-tamagawa-numbers}) in general?
How much are such choices different from the minimal integral periods?
\end{ques}
The precise formulation of Question \ref{ques:kn-tamagawa} should be considered as the refined Tamagawa number conjecture for the central critical twist of $f$ and should be compatible with Kato--Fontaine--Perrin-Riou's formulation of the Tamagawa number conjecture.

\subsubsection{}
% Although Tamagawa numbers are regarded as fudge factors these days, they played an important role when the Birch and Swinnerton-Dyer conjecture was first formulated (cite note ii). In some sense, our formulas say that Tamagawa numbers are the most mysterious factor in the Birch and Swinnerton-Dyer conjecture.
In order to formulate the general conjecture,  we first describe the ``idealistic" integral $p$-adic Hodge theory.
We closely follow \cite[Lem.  14.18]{kato-euler-systems} but removing the low weight assumption. 

Recall that $\mathcal{L} = \mathrm{exp}^*\left(  \mathrm{H}^1_{/f}(\mathbb{Q}_p, T^\dagger_f) \right) \subseteq \mathrm{Fil}^0\mathbf{D}_{\mathrm{dR}}(V^\dagger_f)
$ is the integral image of the dual exponential map in $\S$\ref{sec:explicit-reciprocity-law}.
Suppose that there exists a natural $\varphi$-stable filtered $\mathcal{O}$-lattice $\mathcal{M} \subseteq \mathbf{D}_{\mathrm{dR}}(V^\dagger_f)$
such that the integral Galois representation associated to $\mathcal{M}$ is the restriction of $T^\dagger_f$ to $\mathrm{Gal}(\overline{\mathbb{Q}}_p/\mathbb{Q}_p)$.
Let
$\mathcal{N}^0 = \mathcal{L} + \mathcal{M}^0 \subseteq \mathrm{Fil}^0\mathbf{D}_{\mathrm{dR}}(V^\dagger_f)$
where $ \mathcal{M}^0  = \mathcal{M} \cap  \mathrm{Fil}^0\mathbf{D}_{\mathrm{dR}}(V^\dagger_f)$.
Then define
\begin{align*}
& \mathrm{ord}_\pi \left[ \mathrm{exp}^*:  \mathrm{H}^1_{/f}(\mathbb{Q}_p, T^\dagger_f) \dashrightarrow \mathcal{M}^0 \right] \\
 := & \ \mathrm{length}_{\mathcal{O}} (\mathcal{N}^0/\mathcal{L}) - \mathrm{length}_{\mathcal{O}} (\mathrm{ker}(\mathrm{exp}^*:  \mathrm{H}^1_{/f}(\mathbb{Q}_p, T^\dagger_f)  \to \mathcal{N}^0 )) \\
& - \mathrm{length}_{\mathcal{O}} (\mathcal{M}^0/\mathcal{N}^0) + \mathrm{length}_{\mathcal{O}} (\mathrm{ker}(\mathcal{M}^0 \to \mathcal{N}^0 )) .
\end{align*}
Let $\mathcal{M}' = \mathrm{Hom}_{\mathcal{O}}( \mathcal{M}, \mathcal{O})(1)$. Then define
\begin{align*}
& \mathrm{ord}_\pi  \left[ 1-\varphi:  \mathcal{M}' \dashrightarrow \mathcal{M}' \right] \\
 := & \ \mathrm{length}_{\mathcal{O}} ( \dfrac{ \mathcal{M}' + (1-\varphi)\mathcal{M}'}{(1-\varphi)\mathcal{M}'}  ) - \mathrm{length}_{\mathcal{O}} (\mathrm{ker}(1-\varphi: \mathcal{M}'  \to \mathcal{M}' + (1-\varphi)\mathcal{M}' )) \\
& - \mathrm{length}_{\mathcal{O}} (\dfrac{ \mathcal{M}' + (1-\varphi)\mathcal{M}'}{\mathcal{M}'} ) + \mathrm{length}_{\mathcal{O}} (\mathrm{ker}( \mathcal{M}'  \to \mathcal{M}' + (1-\varphi)\mathcal{M}' )) .
\end{align*}
Following the computation in \cite[Prop. 14.21]{kato-euler-systems}, we \emph{expect} 
$$\mathrm{ord}_\pi  \left[ 1-\varphi:  \mathcal{M}' \dashrightarrow \mathcal{M}' \right] =  \mathrm{ord}_\pi ( 1 - a_p(f)  \cdot p^{-k/2}  - \mathbf{1}(p)  \cdot p^{-1} ) ,$$
so we assume the following statement.
\begin{assu} \label{assu:idealistic-integral-p-adic-hodge-theory}
In addition to the existence of $\mathcal{M}$, we assume the following equality
$$\mathrm{ord}_\pi \left[ \mathrm{exp}^*:  \mathrm{H}^1_{/f}(\mathbb{Q}_p, T^\dagger_f) \dashrightarrow \mathcal{M}^0 \right]  
+ \mathrm{length}_{\mathcal{O}} \mathrm{H}^2(\mathbb{Q}_p, T^\dagger_f) 
=  \mathrm{ord}_\pi ( 1 - a_p(f)  \cdot p^{-k/2}  - \mathbf{1}(p)  \cdot p^{-1} ) .$$
\end{assu}
When the Fontaine--Laffaille case,  Assumption \ref{assu:idealistic-integral-p-adic-hodge-theory} is true and $\mathcal{M}$ comes from the Fontaine--Laffaille modules.
We do not know how to characterize the most general case satisfying Assumption \ref{assu:idealistic-integral-p-adic-hodge-theory}.
We are informed that an important progress on the computation of the local Tamagawa ideals at $p$ is made by Colmez--Wang \cite{colmez-wang}.
\subsubsection{}
We are now ready to state two general refined conjectures. The first one is the ``before taking $\mathrm{exp}^*$" version and the latter is the ``after taking $\mathrm{exp}^*$" version.
\begin{conj} \label{conj:ks-tamagawa-general}
Let $f \in S_k(\Gamma_0(N))$ be a newform with $k \geq 2$ and $p \geq 3$ such that $\rho_f$ has large image.
Then
\[
\partial^{(\infty)}(\ks^{\mathrm{Kato}, \dagger}) \overset{?}{=} \sum_{\ell \vert N_0} \mathrm{length}_{\mathcal{O}} \left( \dfrac{\mathrm{H}^1_{\mathrm{ur}}(\mathbb{Q}_\ell, W^{\dagger}_f ) }{\mathrm{H}^1_{f}(\mathbb{Q}_\ell, W^{\dagger}_f ) }  \right) 
\]
where $N_0$ is the prime-to-$p$ part of $N$.
\end{conj}
\begin{conj} \label{conj:kn-tamagawa-general}
Let $f \in S_k(\Gamma_0(N))$ be a newform with $k \geq 2$ and $p \geq 3$ such that
\begin{enumerate}
\item  $\rho_f$ has large image, 
\item $f$ is not congruent to any anomalous ordinary CM modular form modulo $\pi$,  and
\item the ``idealistic" integral $p$-adic Hodge theory described above works (Assumption \ref{assu:idealistic-integral-p-adic-hodge-theory}). 
\end{enumerate}
Then
\[
\mathrm{ord}_{\pi}(C^{\pm}_{\mathrm{per}}) + \partial^{(\infty)}(\kn^{\mathrm{min}, \dagger}) \overset{?}{=} \sum_{\ell \vert N_0} \mathrm{length}_{\mathcal{O}} \left( \dfrac{\mathrm{H}^1_{\mathrm{ur}}(\mathbb{Q}_\ell, W^{\dagger}_f ) }{\mathrm{H}^1_{f}(\mathbb{Q}_\ell, W^{\dagger}_f ) }  \right) 
\]
where the sign of $C^{\pm}_{\mathrm{per}}$ coincides with that of $(-1)^{k/2 - 1 }$ and $N_0$ is the prime-to-$p$ part of $N$.
\end{conj}
\begin{rem}
\begin{enumerate}
\item If $f$ is not congruent to an anomalous ordinary CM modular form modulo $\pi$,  then $\mathrm{H}^2(\mathbb{Q}_p, T^\dagger_f) = 0$.
\item  If $p \geq 5$, then the residual non-CM condition becomes vacuous since the residual image is already non-solvable.
\item In the setting of Conjecture \ref{conj:kn-tamagawa-general},  we choose $\omega_{f,k/2}$ such that $\mathcal{M}^0 = \mathcal{O} \cdot  \omega_{f,k/2} \subseteq \mathrm{Fil}^0\mathbf{D}_{\mathrm{dR}}(V^\dagger_f)$.
\item  Since minimal integral periods are determined purely by the values of modular symbols, we do not see any arithmetic properties of minimal integral periods like congruence ideals (e.g. \cite{vatsal-cong,kim-ota}) directly in general.
\end{enumerate}
\end{rem}

\section*{Acknowledgement}
We thank Henri Darmon, Karl Rubin, and Christopher Skinner for helpful and detailed discussions.
We are also benefited from discussions with Alberto Angurel Andres, Ashay Burungale, Francesc Castella, and Giada Grossi. 
%We are also benefited from discussions with Francesc Castella and Giada Grossi. 
We thank MSRI/SLMath at Berkeley for providing a wonderful environment for all these discussions.
We also thank Shinichi Kobayashi and Masato Kurihara for their helpful comments and Xin Wan for the detailed discussion on \cite{wan_hilbert}.
%We thank Masato Kurihara for his interest, support and detailed discussions.
%We also thank Xin Wan for the detailed discussion on \cite{wan_hilbert}.
%Ashay Burungale Giada Grossi Francesc Castella Ye Tian Daniel Disegni

We thank Robert Pollack for the numerical computation of higher weight modular forms and for agreeing to write the appendix together.

%\clearpage

\appendix{}
\section{Numerical examples (by Chan-Ho Kim and Robert Pollack)} \label{sec:appendix}
In this appendix, we apply Theorem \ref{thm:main-central-critical} to various elliptic curves and modular forms of higher weight.
We use the LMFDB index \cite{lmfdb-2024} and all the examples satisfy the large image assumption (e.g. \cite{mascot-modular-image}).
\subsection{Elliptic curves}
By applying the algorithm of Alexandru Ghitza \cite[Appendix A]{kim-survey}, we obtain several numerical examples.
The Sage code for this computation is available at:
\begin{center}
\url{https://github.com/aghitza/kurihara_numbers}
\end{center}

\subsubsection{The structure of Tate--Shafarevich groups}
We consider elliptic curves whose rank is zero and the size of the $p$-primary part of their Tate--Shafarevich group is $p^4$ or $p^6$.
We determine the structure of $\sha(E/\mathbb{Q})[p^\infty]$ by computing Kurihara numbers.
This structural information cannot be observed from Birch and Swinnerton-Dyer conjecture.
\begin{center}
{ \scriptsize
\begin{tabular}{ |c|c|c|c|c|c| } 
 \hline
 $p$ & $E$ & reduction at $p$ & $\widedelta_n \pmod{p}$ & $\sha(E/\mathbb{Q})[p^\infty]$ & non-vanishing/vanishing\\ 
 \hline
 3 & \href{https://www.lmfdb.org/EllipticCurve/Q/15675/ba/1}{15675.ba1} & non-split multiplicative &$\widedelta_{7 \cdot 13} \neq 0$ & $(\mathbb{Z}/9\mathbb{Z})^{\oplus 2}$ & 56/64  ($<$ 500) \\ 
 3 & \href{https://www.lmfdb.org/EllipticCurve/Q/19215/x/1}{19215.x1} & additive &$\widedelta_{7 \cdot 19} \neq 0$ & $(\mathbb{Z}/9\mathbb{Z})^{\oplus 2}$ & 48/72 ($<$ 500)\\ 
 3 & \href{https://www.lmfdb.org/EllipticCurve/Q/22077/g/1}{22077.g1} & additive & $\widedelta_{139 \cdot 151} \neq 0$ & $(\mathbb{Z}/9\mathbb{Z})^{\oplus 2}$ & 79/57 ($<$ 500)\\ 
 3 & \href{https://www.lmfdb.org/EllipticCurve/Q/31046/g/1}{31046.g1} & good supersingular & $\widedelta_{31 \cdot 61} \neq 0$ & $(\mathbb{Z}/9\mathbb{Z})^{\oplus 2}$ & 77/76 ($<$ 500)\\ 
3 & \href{http://www.lmfdb.org/EllipticCurve/Q/35090/f/1}{35090.f1} & good ordinary & $\widedelta_{103\cdot 193} \neq 0$ &$(\mathbb{Z}/9\mathbb{Z})^{\oplus 2}$ & 24/42 ($<$ 500) \\
3 & \href{http://www.lmfdb.org/EllipticCurve/Q/40075/p/1}{40075.p1} & good supersingular & $\widedelta_{7 \cdot 19} \neq 0$ & $(\mathbb{Z}/9\mathbb{Z})^{\oplus 2}$ & 77/59 ($<$ 500)\\
3 & \href{http://www.lmfdb.org/EllipticCurve/Q/46731/d/1}{46731.d1} & split multiplicative & $\widedelta_{19 \cdot 43} \neq 0$ & $(\mathbb{Z}/9\mathbb{Z})^{\oplus 2}$ & 39/66 ($<$ 500)\\ 
3 & \href{http://www.lmfdb.org/EllipticCurve/Q/55616/bd/1}{55616.bd1} & good supersingular &  $\widedelta_{31 \cdot 37} \neq 0$ & $(\mathbb{Z}/9\mathbb{Z})^{\oplus 2}$ & 78/75 ($<$ 500)\\
3 & \href{http://www.lmfdb.org/EllipticCurve/Q/55762/i/1}{55762.i1} & good ordinary & $\widedelta_{67 \cdot 193}\neq 0$ & $(\mathbb{Z}/9\mathbb{Z})^{\oplus 2}$ & 92/61 ($<$ 500)\\
3 & \href{http://www.lmfdb.org/EllipticCurve/Q/73206/r/1}{73206.r1} & additive & $\widedelta_{19 \cdot 67} \neq  0$ & $(\mathbb{Z}/9\mathbb{Z})^{\oplus 2}$ & 29/62 ($<$ 500) \\
3 & \href{http://www.lmfdb.org/EllipticCurve/Q/75400/w/1}{75400.w1} & good supersingular & $\widedelta_{7 \cdot 13} \neq  0$ & $(\mathbb{Z}/9\mathbb{Z})^{\oplus 2}$ & 103/68 ($<$ 500) \\
3 & \href{http://www.lmfdb.org/EllipticCurve/Q/83445/e/1}{83445.e1} & non-split multiplicative & $\widedelta_{7 \cdot 31} \neq  0$ & $(\mathbb{Z}/9\mathbb{Z})^{\oplus 2}$ & 68/68 ($<$ 500) \\
3 & \href{http://www.lmfdb.org/EllipticCurve/Q/83790/b/1}{83790.b1} & additive & $\widedelta_{19 \cdot 43} \neq  0$ & $(\mathbb{Z}/9\mathbb{Z})^{\oplus 2}$ & 58/62 ($<$ 500) \\
3 & \href{http://www.lmfdb.org/EllipticCurve/Q/84578/c/1}{84578.c1} &  good supersingular & $\widedelta_{13 \cdot 67} \neq  0$ & $(\mathbb{Z}/9\mathbb{Z})^{\oplus 2}$ & 35/43 ($<$ 500) \\
3 & \href{http://www.lmfdb.org/EllipticCurve/Q/84680/c/1}{84680.c1} & good ordinary & $\widedelta_{7 \cdot 43} \neq  0$ & $(\mathbb{Z}/9\mathbb{Z})^{\oplus 2}$ & 44/61 ($<$ 500) \\
3 & \href{http://www.lmfdb.org/EllipticCurve/Q/84825/y/1}{84825.y1} & additive & $\widedelta_{13 \cdot 19} \neq  0$ & $(\mathbb{Z}/9\mathbb{Z})^{\oplus 2}$ & 62/91 ($<$ 500) \\
3 & \href{http://www.lmfdb.org/EllipticCurve/Q/91605/c/1}{91605.c1} & non-split multiplicative & $\widedelta_{19 \cdot 37} \neq  0$ & $(\mathbb{Z}/9\mathbb{Z})^{\oplus 2}$ & 105/66 ($<$ 500) \\
3 & \href{http://www.lmfdb.org/EllipticCurve/Q/91960/be/1}{91960.be1} & good supersingular & $\widedelta_{7 \cdot 19} \neq  0$ & $(\mathbb{Z}/9\mathbb{Z})^{\oplus 2}$ & 69/84 ($<$ 500) \\
 \hline
 3 & \href{http://www.lmfdb.org/EllipticCurve/Q/441285/e/1}{441285.e1} & non-split multiplicative &$\widedelta_{13 \cdot 61} \neq 0$ & $(\mathbb{Z}/27\mathbb{Z})^{\oplus 2}$ & 13/23  ($<$ 200) \\ 
 3 & \href{http://www.lmfdb.org/EllipticCurve/Q/484080/p/1}{484080.p1} & split multiplicative &$\widedelta_{67 \cdot 151} \neq 0$ & $(\mathbb{Z}/27\mathbb{Z})^{\oplus 2}$ & 2/8 ($<$ 200)\\ 
 \hline
 5 & \href{https://www.lmfdb.org/EllipticCurve/Q/287175/n/2}{287175.n2} & additive & $\widedelta_{131 \cdot 271} \neq 0$ & $(\mathbb{Z}/25\mathbb{Z})^{\oplus 2}$  & 24/4 ($<$ 1,000) \\ 
 5 & \href{https://www.lmfdb.org/EllipticCurve/Q/321398/d/1}{321398.d1} & good ordinary & $\widedelta_{191 \cdot 401} \neq 0$ & $(\mathbb{Z}/25\mathbb{Z})^{\oplus 2}$ &  30/25 ($<$ 1,500) \\ % 6/4 ($<$ 1000) 
 5 & \href{https://www.lmfdb.org/EllipticCurve/Q/366100/y/1}{366100.y1} & additive & $\widedelta_{31 \cdot 101} \neq 0$ & $(\mathbb{Z}/25\mathbb{Z})^{\oplus 2}$ & 39/6 ($<$ 1,500)  \\ 
5 & \href{http://www.lmfdb.org/EllipticCurve/Q/423801/ci/1}{423801.ci1} & good ordinary & $\widedelta_{11 \cdot 41} \neq 0$ & $(\mathbb{Z}/25\mathbb{Z})^{\oplus 2}$ & 12/24 ($<$ 1,500) \\
 5 & \href{https://www.lmfdb.org/EllipticCurve/Q/472834/a/1}{472834.a1} & good ordinary & $\widedelta_{251 \cdot 271} \neq 0$ & $(\mathbb{Z}/25\mathbb{Z})^{\oplus 2}$  & 58/8 ($<$ 1,500) \\ 
 \hline
\end{tabular}
}
\end{center}
In non-vanishing/vanishing, $a/b \ (< c)$ means that
\begin{itemize}
\item $a$ = the number of non-vanishing $\widedelta_n$, and
\item $b$ = the number of vanishing $\widedelta_n$
\end{itemize}
with $\nu(n) = 2$ where each prime factor of $n$ is smaller than $c$.

\subsubsection{Non-trivial Tamagawa defects}
Consider an elliptic curve $E$ of conductor 20787 (\href{http://www.lmfdb.org/EllipticCurve/Q/20787/e/1}{20787.e1}).
Then $E$ has rank zero,  the Tamagawa factor of $E$ at 41 is divisible by 3, and the size of $\sha(E/\mathbb{Q})[3^\infty]$ is $3^4$.
Due to the non-trivial Tamagawa defect, $\widedelta_{n}$ vanishes mod 3 for every $n \in \mathcal{N}_1$.
Since $37, 1783 \in \mathcal{N}_2$ and $\mathrm{ord}_3( \widedelta_{37 \cdot 1783} \pmod{9}) =1$, we still have
$$\mathrm{Sel}(\mathbb{Q}, E[3^\infty]) = \sha(E/\mathbb{Q})[3^\infty] \simeq (\mathbb{Z}/9\mathbb{Z})^{\oplus 2} .$$

\subsubsection{The relation with Cohen--Lenstra}
In all the examples, we have $\sha(E/\mathbb{Q})[p^\infty] \simeq M^{\oplus 2}$ with \emph{cyclic} module $M$.
This is not a strange phenomenon due to Delaunay's analogue of the Cohen--Lenstra heuristics for Tate--Shafarevich groups \cite[\S3]{delaunay-heuristics}.
%There are several elliptic curves $E$ such that $\sha(E/\mathbb{Q})[2^\infty]$ contains $(\mathbb{Z}/2\mathbb{Z})^{\oplus 4}$ as a subgroup \cite[\S10]{fisher-higher-descents}. See also \cite{delaunay-jouhet-cohen-lenstra}.

%67 103
%0
%67 151
%2
%67 163
%0
%67 181
%0
%103 151
%0
%103 163
%0
%103 181
%0
%151 163
%2
%151 181
%0
%163 181
%0
%distribution:  2/8 $<$ 200
\subsubsection{A rank one elliptic curve varying primes}
Consider an elliptic curve $E$ of conductor 43 (\href{http://www.lmfdb.org/EllipticCurve/Q/43/a/1}{43.a1}).
Then $E$ has rank one and $\sha(E/\mathbb{Q})$ is trivial. We obtain the following table regarding the distribution of non-zero $\widedelta_\ell$ with Kolyvagin prime $\ell$ varying $p$.
\begin{center}
{ \scriptsize
\begin{tabular}{ |c|c|c||c|c|c| } 
 \hline
 $p$ & reduction at $p$ &  non-vanishing/vanishing &  $p$ & reduction at $p$ &  non-vanishing/vanishing\\ 
 \hline
 3 & good ordinary & 138/81  ($<$ 10,000) & 13 &  good ordinary &  13/0 ($<$ 20,000) \\ 
 5 & good ordinary & 43/18 ($<$ 10,000) & 17 &  good ordinary &  8/1 ($<$ 30,000)\\ 
 7 &   good supersingular & 27/3 ($<$ 10,000) & 19 & good ordinary &  10/0 ($<$ 30,000)\\ 
 11 & good ordinary  &  17/1 ($<$ 10,000) & 23 & good ordinary & 6/0 ($<$ 30,000)\\ 
 \hline
\end{tabular}
}
\end{center}
In non-vanishing/vanishing, $a/b \ (< c)$ means that
\begin{itemize}
\item $a$ = the number of non-vanishing $\widedelta_\ell$, and
\item $b$ = the number of vanishing $\widedelta_\ell$ 
\end{itemize}
 with Kolyvagin prime $\ell$ is smaller than $c$.
 It will be very interesting if any pattern is recognized in the distribution of non-vanishing $\widedelta_n$.

\subsubsection{Elliptic curves of high rank}
The following computation confirms the triviality of the $p$-primary part of Tate--Shafarevich groups for several elliptic curves of rank $\geq 2$. 
These high rank examples cannot be obtained from any known result on Birch and Swinnerton-Dyer conjecture.
\begin{center}
{ \scriptsize
\begin{tabular}{ |c|c|c|c|c|c| } 
 \hline
 $p$ & $E$ &  rank  & reduction at $p$ & $\widedelta_n \pmod{p}$ &  $\mathrm{Sel}(\mathbb{Q}, E[p^\infty])$ \\ 
 \hline
 3 & \href{https://www.lmfdb.org/EllipticCurve/Q/389/a/1}{389.a1} & 2 & good ordinary &$\widedelta_{79 \cdot 109} \neq 0$ &  $(\mathbb{Q}_3/\mathbb{Z}_3)^{\oplus 2}$ \\ 
 3 & \href{https://www.lmfdb.org/EllipticCurve/Q/433/a/1}{433.a1} & 2 & good ordinary &$\widedelta_{19 \cdot 43} \neq 0$ &  $(\mathbb{Q}_3/\mathbb{Z}_3)^{\oplus 2}$ \\ 
 3 & \href{https://www.lmfdb.org/EllipticCurve/Q/446/a/1}{446.a1} & 2 & good supersingular &$\widedelta_{7 \cdot 31} \neq 0$ &  $(\mathbb{Q}_3/\mathbb{Z}_3)^{\oplus 2}$\\ 
 3 & \href{https://www.lmfdb.org/EllipticCurve/Q/563/a/1}{563.a1} & 2 & good ordinary &$\widedelta_{13 \cdot 61} \neq 0$ &  $(\mathbb{Q}_3/\mathbb{Z}_3)^{\oplus 2}$ \\ 
 3 & \href{https://www.lmfdb.org/EllipticCurve/Q/571/a/1}{571.a1} & 2 & good ordinary &$\widedelta_{7 \cdot 31} \neq 0$ &  $(\mathbb{Q}_3/\mathbb{Z}_3)^{\oplus 2}$ \\ 
 3 & \href{https://www.lmfdb.org/EllipticCurve/Q/643/a/1}{643.a1} & 2 & good ordinary &$\widedelta_{13 \cdot 19} \neq 0$ &  $(\mathbb{Q}_3/\mathbb{Z}_3)^{\oplus 2}$ \\ 

 \hline
 5 & \href{https://www.lmfdb.org/EllipticCurve/Q/655/a/1}{655.a1} & 2 & non-split multiplicative &$\widedelta_{71 \cdot 191} \neq 0$ &  $(\mathbb{Q}_5/\mathbb{Z}_5)^{\oplus 2}$ \\ 
 5 & \href{https://www.lmfdb.org/EllipticCurve/Q/664/a/1}{664.a1} & 2 & good ordinary &$\widedelta_{11 \cdot 191} \neq 0$ &  $(\mathbb{Q}_5/\mathbb{Z}_5)^{\oplus 2}$ \\ 

 5 & \href{https://www.lmfdb.org/EllipticCurve/Q/681/a/1}{681.a1} & 2 & good ordinary &$\widedelta_{61 \cdot 241} \neq 0$ &  $(\mathbb{Q}_5/\mathbb{Z}_5)^{\oplus 2}$ \\ 
 \hline
 7 & \href{https://www.lmfdb.org/EllipticCurve/Q/707/a/1}{707.a1} & 2 & non-split multiplicative &$\widedelta_{281 \cdot 449} \neq 0$ &  $(\mathbb{Q}_7/\mathbb{Z}_7)^{\oplus 2}$ \\ 
 101 & \href{https://www.lmfdb.org/EllipticCurve/Q/707/a/1}{707.a1} & 2 & split multiplicative &$\widedelta_{58379 \cdot 80983} \neq 0$ &  $(\mathbb{Q}_{101}/\mathbb{Z}_{101})^{\oplus 2}$ \\

 \hline
 3 & \href{https://www.lmfdb.org/EllipticCurve/Q/11197/a/1}{11197.a1} & 3 & good supersingular &$\widedelta_{13 \cdot 31 \cdot 73} \neq 0$ &  $(\mathbb{Q}_3/\mathbb{Z}_3)^{\oplus 3}$ \\ 
 3 & \href{https://www.lmfdb.org/EllipticCurve/Q/11642/a/1}{11642.a1} & 3 & good supersingular & $\widedelta_{7 \cdot 97 \cdot 313} \neq 0$ &  $(\mathbb{Q}_3/\mathbb{Z}_3)^{\oplus 3}$ \\ 
 3 & \href{https://www.lmfdb.org/EllipticCurve/Q/12279/a/1}{12279.a1} & 3 & non-split multiplicative & $\widedelta_{7 \cdot 19 \cdot 43} \neq 0$ &  $(\mathbb{Q}_3/\mathbb{Z}_3)^{\oplus 3}$ \\ 
 3 & \href{https://www.lmfdb.org/EllipticCurve/Q/13766/a/1}{13766.a1} & 3 & good ordinary & $\widedelta_{7 \cdot 37 \cdot 79} \neq 0$ &  $(\mathbb{Q}_3/\mathbb{Z}_3)^{\oplus 3}$ \\ 
 3 & \href{https://www.lmfdb.org/EllipticCurve/Q/16811/a/1}{16811.a1} & 3 & good supersingular & $\widedelta_{7 \cdot 19 \cdot 37} \neq 0$ &  $(\mathbb{Q}_3/\mathbb{Z}_3)^{\oplus 3}$ \\ 
 3 & \href{https://www.lmfdb.org/EllipticCurve/Q/18097/a/1}{18097.a1} & 3 & good ordinary & $\widedelta_{7 \cdot 19 \cdot 31} \neq 0$ &  $(\mathbb{Q}_3/\mathbb{Z}_3)^{\oplus 3}$ \\ 

 \hline
 5 & \href{https://www.lmfdb.org/EllipticCurve/Q/18562/a/1}{18562.a1} & 3 & good ordinary & $\widedelta_{41 \cdot 271 \cdot 281} \neq 0$ &  $(\mathbb{Q}_5/\mathbb{Z}_5)^{\oplus 3}$ \\ 
 5 & \href{https://www.lmfdb.org/EllipticCurve/Q/18745/a/1}{18745.a1} & 3 & non-split multiplicative & $\widedelta_{101  \cdot 181  \cdot 541} \neq 0$ &  $(\mathbb{Q}_5/\mathbb{Z}_5)^{\oplus 3}$ \\ 
 5 & \href{https://www.lmfdb.org/EllipticCurve/Q/20888/a/1}{20888.a1} & 3 & good ordinary & $\widedelta_{ 191 \cdot 241 \cdot 331} \neq 0$ &  $(\mathbb{Q}_5/\mathbb{Z}_5)^{\oplus 3}$ \\ 

 \hline

 3 & \href{https://www.lmfdb.org/EllipticCurve/Q/234446/a/1}{234446.a1} & 4 & good supersingular & $\widedelta_{31 \cdot 43 \cdot 61 \cdot 103} \neq 0$  & $(\mathbb{Q}_3/\mathbb{Z}_3)^{\oplus 4}$  \\ 
 3 & \href{https://www.lmfdb.org/EllipticCurve/Q/501029/a/1}{501029.a1} & 4 & good ordinary & $\widedelta_{31 \cdot 43 \cdot 61 \cdot 109} \neq 0$  & $(\mathbb{Q}_3/\mathbb{Z}_3)^{\oplus 4}$  \\ 
 3 & \href{https://www.lmfdb.org/EllipticCurve/Q/545723/a/1}{545723.a1} & 4 & good supersingular & $\widedelta_{19 \cdot 61 \cdot 109 \cdot 139} \neq 0$  & $(\mathbb{Q}_3/\mathbb{Z}_3)^{\oplus 4}$  \\

 \hline
\end{tabular}
}
\end{center}

\subsection{Modular forms}
\subsubsection{Some remarks on the computing algorithm}
The computation of $\widetilde{\delta}^{\mathrm{min},\dagger}_n$ boils down to being able to compute $\lambda^{\pm,\mathrm{min}}_{f}(z^{(k-2)/2},a,n)$ for an eigenform $f$ of weight $k$ and for varying $a$ and $n$.  Standard modular symbols packages in Sage \cite{sage2023} or Magma \cite{magma} allow one to compute $\lambda^{\pm}_f(z^{(k-2)/2},a,n) / \Omega^\pm_f$ where $\Omega^\pm_f \in \mathbb{C}$ are periods such that these normalized period integrals lie in $K_f$, the field generated by the Hecke-eigenvalues of $f$.  That is, these computer packages normalize the period integrals to be algebraic, but do not give finer control of their integrality.  

To fix this normalization problem, we construct the corresponding modular symbols $\varphi_f^\pm \in \mathrm{Hom}_{\Gamma_0(N)}(\mathrm{Div}^0(\mathbb{P}^1(\mathbb{Q})),\mathrm{Hom}({\mathcal P}_{k-2},K_f))$ with respect to the periods $\Omega^\pm_f$ where ${\mathcal P}_{k-2}$ is the space of polynomials of degree at most $k-2$ with coefficients in $\mathbb{Z}$ using the methods of \cite{pollack-stevens}.

Let $\mathfrak{p}$ be the prime lying over $p$ in $K_f$ corresponding to the fixed isomorphism $\mathbb{C} \simeq \overline{\mathbb{Q}}_p$ of $\S$\ref{subsubsec:fixed-isom}.  Then $(K_f)_{\mathfrak p}$ is a finite extension of $\mathbb{Q}_p$ and plays the role of the $p$-adic field $F$ in the main text of this paper.  

If $\pi \in K_f$ is a uniformizer at $\mathfrak{p}$, there is a unique $m \in \mathbb{Z}$ so that  the valuation of $\pi^m \cdot \varphi^\pm_f(D)(Q)$ at $\mathfrak p$ is non-negative for all $D \in \mathrm{Div}^0(\mathbb{P}^1(\mathbb{Q}))$ and all $Q \in \mathcal P_{k-2}$ and has valuation 0 for at least one $D$ and $Q$.  We note that $\varphi_f^\pm$ is determined by its values on finitely many divisors and so determining $m$ involves checking only finitely many divisors $D$.  

We then have that $\pi^{-m} \Omega^\pm_f$ are minimal integral periods in the sense of Definition \ref{defn:minimal-periods}.  With this renormalization in hand, we can compute $\lambda^{\pm,\min}_f(z^{(k-2)/2},a,n) / \Omega^\pm_f$ and thus $\widetilde{\delta}^{\mathrm{min},\dagger}_n$.

Code which carries out these computations is available at:
\begin{center}
\href{https://github.com/rpollack9974/OMS/tree/master/Kurihara_numbers}{https://github.com/rpollack9974/OMS/tree/master/Kurihara\_numbers}
\end{center}
\subsubsection{Modular forms of weight 4}
For a newform $f_0$ and a quadratic Dirichlet character $\chi$,  
write $$f = f_0 \otimes \chi.$$
\begin{center}
{ \scriptsize
\begin{tabular}{ |c|c|c|c|c|c|c|c| } 
 \hline
$p$ &  $f_0$ & $\mathrm{cond}(\chi)$ & reduction at $p$ & central $L$-values &  $\widedelta^{\dagger}_n \pmod{\pi}$ & $\mathrm{Sel}( \mathbb{Q}, W^\dagger_f)$ \\ 
 \hline
3 & \href{https://www.lmfdb.org/ModularForm/GL2/Q/holomorphic/5/4/a/a/}{5.4.a.a} & 61 
   & ordinary & $\mathrm{ord}_3 \left( \frac{L(f,2)}{\Omega_f} \right)  =2$ &   $\widedelta^{\dagger}_{43 \cdot 97} \neq 0$ & $(\mathbb{Z}/3\mathbb{Z})^{\oplus 2}$  \\ 
3 & \href{https://www.lmfdb.org/ModularForm/GL2/Q/holomorphic/5/4/a/a/}{5.4.a.a} & 89 
   & ordinary & $\mathrm{ord}_{s=2}L(f,s) \geq 2$ &   $\widedelta^{\dagger}_{13 \cdot 19} \neq 0$ & $(\mathbb{Q}_3/\mathbb{Z}_3)^{\oplus 2}$  \\ 
3 & \href{https://www.lmfdb.org/ModularForm/GL2/Q/holomorphic/5/4/a/a/}{5.4.a.a} & 457
   & ordinary & odd parity &   $\widedelta^{\dagger}_{13 \cdot 43 \cdot 61} \neq 0$ & $\approx (\mathbb{Q}_3/\mathbb{Z}_3)^{\oplus 3}$  \\ 
 \hline
3 & \href{https://www.lmfdb.org/ModularForm/GL2/Q/holomorphic/7/4/a/a/}{7.4.a.a} & $37$ 
  & ordinary & $\mathrm{ord}_{s=2}L(f,s) \geq 2$   & $\widedelta^{\dagger}_{13 \cdot 19} \neq 0$ &  $(\mathbb{Q}_3/\mathbb{Z}_3)^{\oplus 2}$  \\
3 & \href{https://www.lmfdb.org/ModularForm/GL2/Q/holomorphic/7/4/a/a/}{7.4.a.a} & $88$ 
  & ordinary & $\mathrm{ord}_3 \left( \frac{L(f,2)}{\Omega_f} \right)  =2$  & $\widedelta^{\dagger}_{19 \cdot 79} \neq 0$ &  $(\mathbb{Z}/3\mathbb{Z})^{\oplus 2}$  \\
3 & \href{https://www.lmfdb.org/ModularForm/GL2/Q/holomorphic/7/4/a/a/}{7.4.a.a} & $92$ 
  & ordinary & $\mathrm{ord}_{s=2}L(f,s) \geq 2$ & $\widedelta^{\dagger}_{43 \cdot 61} \neq 0$ &  $(\mathbb{Q}_3/\mathbb{Z}_3)^{\oplus 2}$  \\
 \hline
3 & \href{https://www.lmfdb.org/ModularForm/GL2/Q/holomorphic/11/4/a/a/}{11.4.a.a} & $53$ 
  & ordinary & $\mathrm{ord}_{\pi} \left( \frac{L(f,2)}{\Omega_f} \right)  =2$ & $\widedelta^{\dagger}_{31 \cdot 73} \neq 0$ &  $(\mathcal{O}/\pi\mathcal{O})^{\oplus 2}$  \\
3 & \href{https://www.lmfdb.org/ModularForm/GL2/Q/holomorphic/11/4/a/a/}{11.4.a.a}  & $97$ 
  & ordinary & $\mathrm{ord}_{\pi} \left( \frac{L(f,2)}{\Omega_f} \right)  =2$ & $\widedelta^{\dagger}_{7 \cdot 13} \neq 0$ &  $(\mathcal{O}/\pi\mathcal{O})^{\oplus 2}$ \\
 \hline
3 & \href{https://www.lmfdb.org/ModularForm/GL2/Q/holomorphic/13/4/a/a/}{13.4.a.a} & $37$ 
  & ordinary & $\mathrm{ord}_3 \left( \frac{L(f,2)}{\Omega_f} \right)  =2$ & $\widedelta^{\dagger}_{7 \cdot 73} \neq 0$ &  $(\mathbb{Z}/3\mathbb{Z})^{\oplus 2}$ \\
3 & \href{https://www.lmfdb.org/ModularForm/GL2/Q/holomorphic/13/4/a/a/}{13.4.a.a} &  $44$ 
  & ordinary & $\mathrm{ord}_{s=2}L(f,s) \geq 2$  & $\widedelta^{\dagger}_{7 \cdot 61} \neq 0$ & $(\mathbb{Q}_3/\mathbb{Z}_3)^{\oplus 2}$  \\
 \hline
5 & \href{https://www.lmfdb.org/ModularForm/GL2/Q/holomorphic/17/4/a/a/}{17.4.a.a} & $44$ 
  & ordinary & $\mathrm{ord}_5 \left( \frac{L(f,2)}{\Omega_f} \right)  =2$  & $\widedelta^{\dagger}_{31 \cdot 41} \neq 0$ &  $(\mathbb{Z}/5\mathbb{Z})^{\oplus 2}$  \\
5 & \href{https://www.lmfdb.org/ModularForm/GL2/Q/holomorphic/17/4/a/a/}{17.4.a.a} & $61$ 
  & ordinary & $\mathrm{ord}_5 \left( \frac{L(f,2)}{\Omega_f} \right)  =2$  & $\widedelta^{\dagger}_{31 \cdot 181} \neq 0$ &  $(\mathbb{Z}/5\mathbb{Z})^{\oplus 2}$  \\
 \hline
3 & \href{https://www.lmfdb.org/ModularForm/GL2/Q/holomorphic/17/4/a/b/}{17.4.a.b} &  $13$ 
  & non-ordinary & $\mathrm{ord}_\pi \left( \frac{L(f,2)}{\Omega_f} \right)  =2$  & $\widedelta^{\dagger}_{7 \cdot 31} \neq 0$ & $(\mathcal{O}/\pi\mathcal{O})^{\oplus 2}$  \\
 \hline
5 & \href{https://www.lmfdb.org/ModularForm/GL2/Q/holomorphic/19/4/a/a/}{19.4.a.a}  & $33$ 
  & ordinary & $\mathrm{ord}_{s=2}L(f,s) \geq 2$ & $\widedelta^{\dagger}_{181 \cdot 241} \neq 0$ & $(\mathbb{Q}_5/\mathbb{Z}_5)^{\oplus 2}$  \\
 \hline
3 & \href{https://www.lmfdb.org/ModularForm/GL2/Q/holomorphic/19/4/a/b/}{19.4.a.b} & $17$ 
  & ordinary & $\mathrm{ord}_\pi \left( \frac{L(f,2)}{\Omega_f} \right)  =6$ & $\widedelta^{\dagger}_{7 \cdot 13} \neq 0$ & $(\mathcal{O}/\pi^3\mathcal{O})^{\oplus 2}$  \\
3 & \href{https://www.lmfdb.org/ModularForm/GL2/Q/holomorphic/19/4/a/b/}{19.4.a.b} & $61$ 
  & ordinary & $\mathrm{ord}_\pi \left( \frac{L(f,2)}{\Omega_f} \right)  =4$ & $\widedelta^{\dagger}_{7 \cdot 13} \neq 0$ & $(\mathcal{O}/\pi^2\mathcal{O})^{\oplus 2}$  \\
3 & \href{https://www.lmfdb.org/ModularForm/GL2/Q/holomorphic/19/4/a/b/}{19.4.a.b} & $92$ 
  & ordinary & $\mathrm{ord}_\pi \left( \frac{L(f,2)}{\Omega_f} \right)  =4$ & $\widedelta^{\dagger}_{13 \cdot 109} \neq 0$ & $(\mathcal{O}/\pi^2\mathcal{O})^{\oplus 2}$  \\
 \hline
3 & \href{https://www.lmfdb.org/ModularForm/GL2/Q/holomorphic/23/4/a/a/}{23.4.a.a} & $37$ 
  & ordinary & $\mathrm{ord}_3 \left( \frac{L(f,2)}{\Omega_f} \right)  =2$ & $\widedelta^{\dagger}_{31 \cdot 151} \neq 0$ & $(\mathbb{Z}/3\mathbb{Z})^{\oplus 2}$  \\
3 & \href{https://www.lmfdb.org/ModularForm/GL2/Q/holomorphic/23/4/a/a/}{23.4.a.a} &  $56$ 
  & ordinary & $\mathrm{ord}_3 \left( \frac{L(f,2)}{\Omega_f} \right)  =2$ & $\widedelta^{\dagger}_{67 \cdot 73} \neq 0$ & $(\mathbb{Z}/3\mathbb{Z})^{\oplus 2}$  \\
3 & \href{https://www.lmfdb.org/ModularForm/GL2/Q/holomorphic/23/4/a/a/}{23.4.a.a} &  $61$ 
  & ordinary & $\mathrm{ord}_{s=2}L(f,s) \geq 2$ & $\widedelta^{\dagger}_{7 \cdot 79} \neq 0$ & $(\mathbb{Q}_3/\mathbb{Z}_3)^{\oplus 2}$  \\
 \hline
5 & \href{https://www.lmfdb.org/ModularForm/GL2/Q/holomorphic/23/4/a/a/}{23.4.a.a}  & $61$ 
  & ordinary & $\mathrm{ord}_{s=2}L(f,s) \geq 2$ & $\widedelta^{\dagger}_{31 \cdot 311} \neq 0$ & $(\mathbb{Q}_5/\mathbb{Z}_5)^{\oplus 2}$  \\
 \hline
3 & \href{https://www.lmfdb.org/ModularForm/GL2/Q/holomorphic/24/4/a/a/}{24.4.a.a}  & $53$ 
  & Steinberg & $\mathrm{ord}_3 \left( \frac{L(f,2)}{\Omega_f} \right)  =2$ & $\widedelta^{\dagger}_{61 \cdot 73} \neq 0$ & $(\mathbb{Z}/3\mathbb{Z})^{\oplus 2}$  \\
 \hline
3 &   \href{https://www.lmfdb.org/ModularForm/GL2/Q/holomorphic/127/4/a/a/}{127.4.a.a} & \textrm{no twist} &  ordinary & $\mathrm{ord}_{s=2}L(f,s) \geq 2$ &  $\widedelta^\dagger_{7 \cdot 13} \neq 0$& $(\mathbb{Q}_3/\mathbb{Z}_3)^{\oplus 2}$\\ 
3 & \href{https://www.lmfdb.org/ModularForm/GL2/Q/holomorphic/159/4/a/a/}{159.4.a.a} & \textrm{no twist} &  non-ordinary & $\mathrm{ord}_{s=2}L(f,s) \geq 2$ & $\widedelta^\dagger_{7 \cdot 37} \neq 0$ & $(\mathbb{Q}_3/\mathbb{Z}_3)^{\oplus 2}$\\ 
 \hline
\end{tabular}
}
\end{center}
\begin{itemize}
\item In \href{https://www.lmfdb.org/ModularForm/GL2/Q/holomorphic/5/4/a/a/}{5.4.a.a} twisted by quadratic character of conductor 457,  we know that
\begin{itemize}
\item the vanishing order of $L(f,s)$ at $s=2$ is odd, and 
\item the corresponding $p$-adic $L$-function is divisible by $X^3$ in $\Lambda$ \emph{modulo a high power of $p$}.
\end{itemize}
Our computation $\widedelta^{\dagger}_{13 \cdot 43 \cdot 61} \neq 0$ confirms the Iwasawa main conjecture. 
Thus, the Selmer corank is \emph{probably} 3, but we cannot prove the the order of vanishing of the $p$-adic $L$-function at $X$ is 3.
\item In \href{https://www.lmfdb.org/ModularForm/GL2/Q/holomorphic/11/4/a/a/}{11.4.a.a}, the defining polynomial of $F$ is $x^2-2x-2$ over $\mathbb{Q}_3$ and 3 is ramified in $F/\mathbb{Q}_3$. Denote by $\pi$ the prime of $F$ lying over $3$, so $(\pi)^2 = (3)$ in $\mathcal{O}$.
\item In \href{https://www.lmfdb.org/ModularForm/GL2/Q/holomorphic/17/4/a/b/}{17.4.a.b}, the defining polynomial of $F$ is $x^3 -x^2 - 24x + 32$  over $\mathbb{Q}_3$ and 3 splits into two primes. We choose $\pi$ which has $e =f=1$. The other prime has $e=1$ and $f=2$.
\item In \href{https://www.lmfdb.org/ModularForm/GL2/Q/holomorphic/19/4/a/b/}{19.4.a.b}, the defining polynomial of $F$ is $x^3 - 3x^2 - 18x + 38$ over $\mathbb{Q}_3$ and 3 splits into two primes. We choose $\pi$ which has $e=f=1$. The other prime has $e=2$ and $f=1$.
\item The last two examples come from \cite[\S7]{dummigan-stein-watkins}.
\end{itemize}

\bibliographystyle{amsalpha}
\bibliography{library}

\end{document}